\documentclass[a4paper,11pt]{amsart}
\usepackage[cm]{fullpage}
\usepackage{hyperref}

\usepackage{mymacros}

\usepackage{enumitem}

\title{Non-archimedean SYZ fibrations via tropical contractions}
\author{Yuto Yamamoto}
\address{
RIKEN iTHEMS, Wako, Saitama 351-0198, Japan}
\email{yuto.yamamoto@riken.jp}
\date{}
\begin{document}

\begin{abstract}
We consider a toric degeneration of Calabi--Yau complete intersections of Batyrev--Borisov in the Gross--Siebert program.
The author showed in his previous work that there exists an integral affine contraction map called a tropical contraction, from the tropical variety obtained as its tropicalization to the dual intersection complex of the toric degeneration.
In this article, we prove that the dual intersection complex is isomorphic to the essential skeleton of the Berkovich analytification as piecewise integral affine manifolds, and the composition of the tropicalization map and the tropical contraction is an affinoid torus fibration with a discriminant of codimension $2$, which induces the same integral affine structure as the one coming from the toric degeneration.
This is a generalization of an earlier work by Pille-Schneider for a specific degeneration of Calabi--Yau hypersurfaces in projective spaces.
\end{abstract}

\maketitle

\section{Introduction}\label{sc:intro}

In this article, we study relations between the Gross--Siebert program \cite{MR2213573, MR2669728, MR2846484} and the non-archimedean geometric approach to the SYZ mirror symmetry initiated by Kontsevich--Soibelman \cite{MR2181810}.
The SYZ conjecture predicts that a Calabi--Yau manifold admits the structure of a special Lagrangian torus fibration, and the mirror Calabi--Yau manifold is obtained by taking its dual fibration \cite{MR1429831}.
The fibration is called an SYZ fibration.
Based on the proposal by Kontsevich--Soibelman and the previous study \cite{MR3595497} on the relation with the minimal model program, Nicaise--Xu--Yu \cite{MR3946280} constructed an analogue of the SYZ fibration in non-archimedean geometry.
They showed that for a maximally degenerate Calabi--Yau variety, the Berkovich retraction associated with a \emph{good minimal dlt-model} is an affinoid torus fibration away from a subset of codimension $2$.
It is called a \emph{non-archimedean SYZ fibration}.
The Berkovich retraction is a map from the Berkovich analytification of the degeneration to its skeleton which is identified with the dual complex of the dlt-model, and induces an integral affine structure (with singularities) on the skeleton.
On the other hand, in the Gross--Siebert program, we consider a so-called \emph{toric degeneration}, a degeneration to a union of toric varieties.
In general, a toric degeneration is not a good minimal dlt-model (and vice versa), and there is no Berkovich retraction associated with a toric degeneration.
Nevertheless, also with a toric degeneration, we associate its dual intersection complex, and equip it with an integral affine structure with singularities in a certain way in the Gross--Siebert program.

The dual (intersection) complexes in these two approaches are basically defined in different ways (see \pref{sc:berk} and \cite[Section 4]{MR2213573} for the respective definitions, and compare them), and whether they can be naturally identified with each other seems non-trivial.
At least, their integral affine structures considered in the respective approaches are apparently different in general.
In fact, the singularities of the integral affine structure induced by a Berkovich retraction locates in the union of cells of codimension greater than or equal to $2$ of the dual complex, while in the Gross--Siebert program, they intersect also with the interiors of cells of codimension $1$.
Mazzon and Pille-Schneider \cite{MPS21} pointed out this discrepancy, and constructed a non-archimedean SYZ fibration whose base coincides with the integral affine manifold with singularity in the Gross--Siebert program for a quintic $3$-fold, generalizing the previous construction by Kontsevich--Soibelman for a quartic K3 surface \cite[Section 4.2.5]{MR2181810}.
The non-archimedean SYZ fibrations in loc.cit.~and \cite{MPS21} are constructed by considering the Berkovich retraction associated with some non-minimal model as an intermediate step, and composing it with a further retraction of the skeleton.

Subsequently, Pille-Schneider \cite{PS22} generalized it to the higher-dimensional case.
He did it by using tropical geometry rather than considering a non-minimal model as an intermediate step.
To be specific, he uses tropicalization maps and \emph{tropical contractions} of \cite{Yam21}, which will be explained shortly.
The goal of this article is to generalize his work further to the case of Calabi--Yau complete intersections of Batyrev--Borisov \cite{MR1463173}.
Our purpose of doing it is to build a bridge between the Gross--Siebert program and the non-archimedean geometric approach to mirror symmetry.

\subsection{Main result}\label{sc:main}

Let $k$ be an algebraically closed field of characteristic $0$.
We set $R:=k\ldd t \rdd$ and $K:=k \lbb t \rbb$.
Let further $d$ and $r$ be positive integers.
Consider a free $\bZ$-module $M$ of rank $d+r$ and its dual lattice $N:=\Hom(M, \bZ)$.
We set $M_\bR:=M \otimes_\bZ \bR$ and $N_\bR:=N \otimes_\bZ \bR=\Hom(M, \bR)$.
Let $\Delta \subset M_\bR$ be a reflexive polytope, and $\Delta =\Delta_1+ \cdots + \Delta_r$ be a nef partition.
From these with some additional data, Gross  \cite{MR2198802} constructed a toric degeneration $f \colon \scX \to \Spec R$ of Calabi--Yau complete intersections of Batyrev--Borisov, and its dual intersection complex $B^{\check{h}}_\nabla$ equipped with an integral affine structure with singularities.
Let $f' \colon X \to \Spec K$ be the base change of the toric degeneration $f \colon \scX \to \Spec R$ to $K$.
The tropicalization $\trop(X)$ of $X$ is defined as the image of the Berkovich analytification $X^{\an}$ of $X$ by the tropicalization map $\trop$.
In the previous work \cite{Yam21}, the author showed that one has $B^{\check{h}}_\nabla \subset \trop (X)$, and there exists a proper continuous map called a tropical contraction
\begin{align}
\delta \colon \trop (X) \to B^{\check{h}}_\nabla,
\end{align}
which preserves the integral affine structures (\cite[Theorem 1.2]{Yam21}\footnote{Throughout the article, we refer to arXiv version 2 of \cite{Yam21}; the numbering may change in later revisions.}).
The sheaves of integral affine functions on tropical spaces are regarded as the structure sheaves in tropical geometry.
In this sense, one can say that the tropical contraction $\delta$ is a \emph{morphism of tropical spaces} (cf.~\cite[Remark 3.27]{Yam21}).
We refer the reader to \cite[Section 3]{Yam21} for more details about tropical contractions.
The construction of the tropical contraction $\delta$ is briefly recalled in \pref{sc:contraction}.

Let further $\Sk(X) \subset X^{\an}$ denote the essential skeleton of $X$ of (cf.~\cite[Definition 4.10]{MR3370127}).
It has a canonical piecewise affine structure (cf.~\cite[4.10.2]{MR3370127}).
The following is the main result of this article.
The precise setup is stated in \pref{sc:toric-CY}.

\begin{theorem}\label{th:main}
The following hold:
\begin{enumerate}
\item The restriction of the tropicalization map $\trop$ to the essential skeleton $\Sk(X) \subset X^{\an}$ is a piecewise integral affine isomorphism onto the dual intersection complex $B^{\check{h}}_\nabla \subset \trop (X)$ of the toric degeneration $f \colon \scX \to \Spec R$.
\item The composition of the tropicalization map and the tropical contraction 
\begin{align}\label{eq:na-syz}
\delta \circ \mathrm{trop} \colon X^\mathrm{an} \to B^{\check{h}}_\nabla \cong \Sk (X)
\end{align} 
is an affinoid torus fibration outside the discriminant $\Gamma \subset B^{\check{h}}_\nabla$, 
and the integral affine structure on $B^{\check{h}}_\nabla \setminus \Gamma$ induced by it coincides with the one coming from the toric degeneration $f \colon \scX \to \Spec R$.
\end{enumerate}
\end{theorem}

The definition of affinoid torus fibrations is recalled in \pref{df:affinoid}.
The outline of the proof of the theorem is also explained using an example in \pref{sc:overview}.

\pref{th:main} is proved in \cite[Theorem A]{PS22} for a specific degeneration of Calabi--Yau hypersurfaces in projective spaces (cf.~\pref{rm:ps22}).
\pref{th:main}(2) is also regarded as a generalization of \cite[Section 4.2.5]{MR2181810} and \cite[Theorem C]{MPS21} (cf.~\pref{rm:KSMPS21}).

\subsection{Metric SYZ conjecture}

In the setup of \cite{PS22} (cf.~\pref{rm:ps22}), Pille-Schneider also showed a strong version of the so-called \emph{NA MA-real MA comparison property} for the Fermat family of Calabi--Yau hypersurfaces (\cite[Theorem B]{PS22}).
The NA MA-real MA comparison property requires that there exists an snc-model such that the potential of a solution to the non-archimedean Monge--Amp\`{e}re equation corresponding to the complex Monge--Amp\`{e}re equation giving the Calabi--Yau metric is constant along fibers of the associated Berkovich retraction (cf.~\cite[Definition 3.3]{MR4688155}).
It is shown by Li \cite{MR4688155} that the NA MA-real MA comparison property implies the \emph{metric SYZ conjecture} which claims that there exist special Lagrangian torus fibrations on generic regions of the Calabi--Yau manifolds near the limit $t \to 0$ (cf.~e.g.~\cite[Conjecture 1.1]{MR4701493}).
Furthermore, it is also proved in \cite[Proposition 5.8]{PS22} that for the Fermat family, the special Lagrangian torus fibrations on the generic regions, which were originally constructed in \cite{MR4460593} converge to the non-archimedean SYZ fibration of \cite[Theorem A]{PS22} (or of \pref{th:main} of this article) in the hybrid space associated with the degeneration.

In order to prove \pref{th:main}, we construct a minimal snc-model for $X$ of \pref{sc:main} (\pref{pr:snc}).
This model is also used in \cite{GY24} to prove the NA MA-real MA comparison property for a toric degeneration of Calabi--Yau complete intersections in a toric variety, under the assumption that the solution of the non-archimedean Monge--Amp\`{e}re equation is given by a toric metric.
By further applying the results of \cite{MR4701493, AH23} on the existence of a solution to the corresponding real Monge--Amp\`{e}re equation, we also generalize the existing results (\cite{MR4692438, MR4701493, AH23, MR4886023}) on the metric SYZ conjecture as corollaries (cf.~\cite[Corollary 1.3, Corollary 1.4]{GY24}).

\subsection{Other related works}

There are also many works constructing (topological or Lagrangian or hybrid) SYZ fibrations such as \cite{MR1738179, MR1821145, MR1876075, MR2487600, MR4290084, RZ20, MR4692382}.
It is conjectured by Kontsevich--Soibelman \cite[Conjecture 3]{MR2181810} that a maximally degenerating family of Calabi--Yau manifolds with Ricci-flat K\"{a}hler metrics converges to the base of the non-archimedean SYZ fibration in the Gromov--Hausdorff topology.
This is confirmed in \cite{MR4612002, MR4692382} for $K$-trivial finite quotients of abelian varieties.

\subsection{Organization of this article}

In \pref{sc:pre}, we recall some basic notions which will be used in this article, and fix notation.
In \pref{sc:toric-CY}, we recall the construction by Gross \cite{MR2198802} of toric degenerations of Calabi--Yau complete intersections,  and state the precise setup of \pref{th:main}.
In \pref{sc:contraction}, we recall the construction of tropical contractions of \cite{Yam21}.
The proof of \pref{th:main} is given in \pref{sc:proof1}.

\section{Preliminaries}\label{sc:pre}

\subsection{Notaion}

Throughout this article, $k$ is an algebraically closed field of characteristic $0$, and $R:=k\ldd t \rdd$, $K:=k \lbb t \rbb$.
The maximal ideal in $R$ is denoted by $\frakm$.
The field $K$ has the standard valuation $\ord_t$ given by
\begin{align}
\ord_t \colon K \to \bR \cup \lc \infty \rc, \quad \sum_{j \in \bZ} c_j t^j \mapsto \min \lc j \in \bZ \relmid c_j \neq 0 \rc.
\end{align}

For a subset $S \subset \bR^d$ $\lb d \in \bZ_{>0} \rb$, the \emph{convex hull} $\conv \lb S \rb \subset \bR^d$, the \emph{affine hull} $\aff(S) \subset \bR^d$, the \emph{conic hull} $\cone \lb S \rb \subset \bR^d$, and the \emph{linear span} $\vspan \lb S \rb \subset \bR^d$ are defined by
\begin{align}
\conv \lb S \rb&:=\lc \sum_{i=1}^n \lambda_i s_i \relmid n \geq 1, s_i \in S, \lambda_i \in \bR_{\geq 0}, \sum_{i=1}^n \lambda_i =1 \rc, \\
\aff \lb S \rb&:=\lc \sum_{i=1}^n \lambda_i s_i \relmid n \geq 1, s_i \in S, \lambda_i \in \bR, \sum_{i=1}^n \lambda_i =1 \rc, \\
\cone \lb S \rb&:=\lc \sum_{i=1}^n \lambda_i s_i \relmid n \geq 0, s_i \in S, \lambda_i \in \bR_{\geq 0} \rc, \\
\vspan \lb S \rb&:=\lc \sum_{i=1}^n \lambda_i s_i \relmid n \geq 0, s_i \in S, \lambda_i \in \bR \rc.
\end{align}
The relative interior of $S$ will be denoted by $\rint (S) \subset S$.

For a free $\bZ$-module $M$ of finite rank and its dual $N:=\Hom \lb M, \bZ \rb$, we let $M_\bR:=M \otimes_\bZ \bR$ and $N_\bR:=N \otimes_\bZ \bR=\Hom(M, \bR)$.
The ranks of $M$ and $N$ may vary depending on the sections.

\subsection{Berkovich geometry}\label{sc:berk}

We recall the basic notions in Berkovich geometry, which will be used in this article.
We refer the reader to \cite{MR3370127} or \cite{MR3419957} for more details. 
This subsection is based on these papers.
Let $X$ be a connected regular $K$-scheme of finite type.
The \emph{Berkovich analytification} $X^{\an}$ is defined as the set of couples of a scheme-theoretic point of $X$ and a valuation on the residue field at the scheme-theoretic point, which extends the valuation $\ord_t$ on $K$.
In this article, a point in $X^{\an}$ is denoted by $x$, and the scheme-theoretic point and the additive valuation corresponding to the point $x \in X^{\an}$ are denoted by $\xi_x \in X$ and $v_x \colon k(x) \to \bR \cup \lc \infty \rc$ respectively, where $k(x)$ denotes the residue field at the point $\xi_x \in X$.
Although the analytic space $X^{\an}$ is also endowed with a topology and a structure sheaf, we do not recall them here, since we do not use them explicitly in this article.

A \emph{model} of $X$ is a normal flat $R$-scheme $\scrX$ endowed with an isomorphism $\scrX \times_R K \cong X$.
We will write its special fiber as $\scrX_k:=\scrX \times_R k$.
An \emph{snc-model} is a regular model $\scrX$ whose special fiber $\scrX_k$ is a divisor with strict normal crossing.
Let $\scrX$ be an snc-model of $X$.
We write its special fiber as $\scrX_k=\sum_{i \in I} N_i D_i$ $\lb N_i \in \bN \rb$ with irreducible components $D_i$.
A non-empty connected component of the schematic intersection $\bigcap_{j \in J} D_j$ with $J \subset I$ is called a \emph{stratum}.
Let $\Delta \lb \scrX_k \rb$ be the \emph{dual complex} of $\scrX_k$.
It is a $\Delta$-complex consisting of simplices, which has a bijective correspondence with the set of strata of $\scrX_k$.
The simplex corresponding to a connected component $S$ of $\bigcap_{j \in J} D_j$ is
\begin{align}
\sigma_S:=\lc \alpha=\lb \alpha_j \rb_{j \in J} \in \lb \bR_{\geq 0} \rb^J \relmid \sum_{j \in J} \alpha_j N_j =1\rc,
\end{align}
and we have $\sigma_{S'} \prec \sigma_{S}$ in $\Delta \lb \scrX_k \rb$ if and only if $S \subset S'$.
Let $\alpha \in \sigma_S$ be an arbitrary point.
We consider the map
\begin{align}\label{eq:expansion}
\scO_{\scrX, \xi}^{\times} \to \bR, \quad f=\sum_{\beta \in \lb \bZ_{\geq 0} \rb^J} c_\beta \prod_{j \in J} z_j^{\beta_j} 
\mapsto \min \lc \sum_{j \in J} \alpha_j \beta_j \relmid \beta = \lb \beta_j \rb_{j \in J} \in \lb \bZ_{\geq 0} \rb^J, c_\beta \neq 0 \rc,
\end{align}
where $\xi$ is the generic point of $S$, $c_\beta$ is either zero or a unit in $\widehat{\scO}_{\scrX, \xi}$, and $z_j \in \scO_{\scrX, \xi}$ is a local equation of $D_j$ such that $\lc z_j \rc_{j \in J}$ is a regular system of parameters of $\scO_{\scrX, \xi}$ (cf.~e.g.~\cite[Section 2.4]{MR3370127}).
The above expansion of $f \in \scO_{\scrX, \xi}^{\times}$ is taken by regarding $f$ as an element of $\widehat{\scO}_{\scrX, \xi}$ and using Cohen's structure theorem.
The map \eqref{eq:expansion} naturally extends to a valuation on the function field $K(X)$.
We write it as $v_\alpha \colon K(X) \to \bR \cup \lc \infty \rc$, and call the point in $X^{\an}$ corresponding to $v_\alpha$ the \emph{monomial point} associated with $(S, \alpha)$.
When $S$ is an irreducible component of the special fiber $\scrX_k$ and $\alpha$ is the unique point of $\sigma_S$, it is also called the \emph{divisorial point} associated with $S$.
The map
\begin{align}
i_\scrX \colon \Delta \lb \scrX_k \rb \to X^{\an}, \quad \sigma_S \ni \alpha \mapsto v_\alpha,
\end{align}
is a homeomorphism onto its image (cf.~e.g.~\cite[Proposition 3.1.4]{MR3370127}), and the image is called the \emph{skeleton} of $\scrX$.
We write it as $\Sk \lb \scrX \rb \subset X^{\an}$.

Let $\frakX$ be the $\frakm$-adic formal completion of $\scrX$.
The generic fiber $\frakX_\eta$ of $\frakX$ is the compact analytic domain in $X^{\an}$, which consists of points $x \in X^{\an}$ such that the morphism $\Spec k(x) \to X$ extends to a morphism
\begin{align}\label{eq:center}
\Spec k(x)^\circ \to \scrX,
\end{align}
where $k(x)^\circ$ is the valuation ring of the valued field $\lb k(x), v_x \rb$.
The image of the closed point of $\Spec k(x)^\circ$ by \eqref{eq:center} is called the \emph{center} of $x \in X^{\an}$.
If $\scrX$ is proper, we have $\frakX_\eta=X^{\an}$ by the valuative criterion of properness.
The \emph{Berkovich retraction}
\begin{align}\label{eq:retraction}
\rho_\scrX \colon \frakX_\eta \to \Sk \lb \scrX \rb
\end{align}
associated with $\scrX$ is defined as follows:
For a point $x \in \frakX_\eta$, let $S$ be the minimal stratum of $\scrX_k$ containing the center of $x$, and suppose that it is a connected component of $\bigcap_{j \in J} D_j$.
Let further $z_j$ be a local equation of $D_j$ at the center of $x$.
The image of the point $x$ by \eqref{eq:retraction} is defined to be the monomial point associated with $(S, \alpha)$ where $\alpha=\lb \alpha_j :=v_x(z_j) \rb_{j \in J} \in \sigma_S$.

Let $S$ be a stratum of $\scrX_k$, and $\frakX_S$ be the formal completion of $\scrX$ along $S$.
The generic fiber $\frakX_{S, \eta}$ of $\frakX_S$ is the subset of $\frakX_\eta$, which consists of points $x \in \frakX_\eta$ whose center is contained in $S$.
For $\frakX_{S, \eta}$, we have the retraction 
\begin{align}
\rho_\scrX \colon \frakX_{S, \eta} \to \ostar(\sigma_S)
\end{align}
given by the restriction of \eqref{eq:retraction}, where $\ostar(\sigma_S) \subset \Sk \lb \scrX \rb$ is the open star of $\sigma_S$ in $\Sk (\scrX) \cong \Delta \lb \scrX_k \rb$.

\subsection{Tropicalizations}\label{sc:trop}

We recall the definition of tropicalizations.
We refer the reader to \cite{MR2428356, MR2511632} and \cite[Section 5]{MR2900439} for more details.
Let $(\bT:= \bR \cup \lc \infty \rc, \min, +)$ be the tropical number semifield.
We equip the set $\bT$ with the topology which makes it homeomorphic to a half line.
Let further $M$ be a free $\bZ$-module of rank $d \in \bZ_{> 0}$ and $N:=\Hom(M, \bZ)$ be its dual.
For a cone $C \subset N_\bR$, we also set
\begin{align}
	C^\vee&:= \lc m \in M_\bR \relmid \la m,n \ra \geq 0  \mathrm{\ for\ all\ } n \in C \rc, \\
	C^\perp&:= \lc m \in M_\bR \relmid \la m,n \ra =0  \mathrm{\ for\ all\ } n \in C \rc.
\end{align}

Let $\Sigma$ be a fan in $N_\bR$.
For a cone $C \in \Sigma$, we define $X_C(\bT)$ as the set of monoid homomorphisms $C^\vee \cap M \to (\bT, + )$
\begin{align}
	X_C(\bT):=\Hom(C^\vee \cap M, \bT)
\end{align}
with the compact open topology.
For elements $p_1, p_2 \in X_C(\bT)$, we define $p_1+p_2 \in X_C(\bT)$ as
\begin{align}
C^\vee \cap M \to \bT, \quad m \mapsto \lb p_1+p_2 \rb (m):=p_1(m)+p_2(m).
\end{align}
Then $\lb X_C(\bT), + \rb$ also becomes a monoid.
For an element $s \in C \subset X_C(\bT)$ and $\lambda \in \bT_{\geq 0}$, we also define $\lambda s \in X_C(\bT)$ as
\begin{align}
C^\vee \cap M \to \bT, \quad m \mapsto \lambda s(m).
\end{align}
Here we use the convention $\infty \cdot 0 =0$.
For a subset $S \subset C \subset X_C(\bT)$, we define 
\begin{align}\label{eq:tcone}
\cone_\bT \lb S \rb:=\lc \sum_{i=1}^n \lambda_i s_i \in X_C(\bT) \relmid n \geq 0, s_i \in S, \lambda_i \in \bT_{\geq 0} \rc. 
\end{align}
When a cone $C_1 \in \Sigma$ is a face of a cone $C_2 \in \Sigma$, we have a natural immersion
\begin{align}
	X_{C_1}(\bT) \to X_{C_2}(\bT),\quad \left(p \colon C_1^\vee \cap M \to \bT\right) \mapsto ({C_2}^\vee \cap M \subset C_1^\vee \cap M \xrightarrow{p} \bT).
\end{align}
By gluing $\lc X_C(\bT) \rc_{C \in \Sigma}$ together, we obtain the \emph{tropical toric variety} $X_\Sigma(\bT)$ associated with the fan $\Sigma$
\begin{align}
	X_\Sigma(\bT):=\lb \bigsqcup_{C \in \Sigma} X_C(\bT) \rb \bigg/ \sim .
\end{align}
The \emph{tropical torus orbit} $O_C(\bT)$ corresponding to the cone $C \in \Sigma$ is defined by
\begin{align}
O_C(\bT):=\Hom(C^\perp \cap M,\bR)=N_\bR / \vspan(C).
\end{align}
For a face $C' \prec C$, there is a natural inclusion
\begin{align}\label{eq:O-emb}
O_{C'}(\bT) \hookrightarrow X_C(\bT), \quad \left(p \colon {C'}^\perp \cap M \to \bR \right) \mapsto 
\lb m \mapsto 
\left\{
\begin{array}{ll}
p(m) & m \in {C'}^\perp \cap M \\
\infty & \mathrm{otherwise} \\
\end{array}
\right.
\rb.
\end{align}
We think of the set $O_C(\bT)$ as a subset of $X_C(\bT)$ by the inclusion.
One has
\begin{align}
X_\Sigma(\bT)=\bigsqcup_{C \in \Sigma} O_{C}(\bT).
\end{align}
There is also a natural projection map
\begin{align}\label{eq:proj}
\pi_C \colon X_C(\bT) \to O_C(\bT), \quad \left(p \colon C^\vee \cap M \to \bT\right) \mapsto (C^\perp \cap M \subset C^\vee \cap M \xrightarrow{p} \bT).
\end{align}

For a cone $C \in \Sigma$, let $X_{C}:= \Spec \lb K \ld C^\vee \cap M \rd \rb$ be the affine toric variety over $K$ associated with the cone $C$.
Let further $X_{\Sigma}$ denote the toric variety over $K$ associated with the fan $\Sigma$.
The \emph{tropicalization map}
\begin{align}\label{eq:trop}
\trop \colon \lb X_{\Sigma} \rb^{\an} \to X_{\Sigma} \lb \bT \rb
\end{align}
is given on each open subset $\lb X_C \rb^{\an} \subset \lb X_{\Sigma} \rb^{\an}$ by
\begin{align}
\lb X_{C} \rb^{\an} \to X_{C} \lb \bT \rb, \quad 
x \mapsto 
\lb m \mapsto v_x \lb z^m \rb \rb,
\end{align}
where $\lb X_\Sigma \rb^{\an}, \lb X_{C} \rb^{\an}$ are the Berkovich analytifications of the toric varieties $X_\Sigma, X_{C}$.
In particular, when the fan $\Sigma$ consists only of the single cone $\lc 0 \rc \subset N_\bR$, the map \eqref{eq:trop} is 
\begin{align}\label{eq:trop'}
\trop \colon \lb \Spec K \ld M \rd \rb^{\an} \to N_\bR, \quad x \mapsto 
\lb m \mapsto v_x \lb z^m \rb \rb.
\end{align}

\begin{definition}{\rm(\cite[Definition 4]{MR2181810})}\label{df:affinoid}
Let $f \colon Y \to B$ be a continuous map from a $K$-analytic space $Y$ to a topological space $B$.
We say that $f$ is a $d$-dimensional \emph{affinoid torus fibration} if for any point $b \in B$, there exists an open neighborhood $U \subset B$ of $b$, an open subset $V \subset N_\bR$, and a commutative diagram
\begin{align}
  \begin{CD}
     f^{-1} \lb U \rb @>>> \trop^{-1} \lb V \rb \\
  @V{f}VV    @V{\trop}VV \\
     U  @>>>  V,
  \end{CD}
\end{align}
where the right vertical map is the restriction of \eqref{eq:trop'}, the upper horizontal map is an isomorphism of $K$-analytic spaces, and the lower horizontal map is a homeomorphism.
\end{definition}

\subsection{Shuffles}\label{sc:shuffle}

We recall the notion called shuffles and the relation with triangulations of products of simplices.
We will use shuffles in \pref{sc:proof1}.
This subsection is based on \cite[Section B.6]{MR4406774}.
Let $p_0, \cdots, p_r$ be non-negative integers, and set $\bar{p}:=\sum_{i=0}^r p_i$.
Let further $\scS=\lc S_i \relmid 0 \leq i \leq r \rc$ be a family of mutually disjoint subsets of the set $\lc 1, 2, \cdots, \bar{p} \rc$ such that 
\begin{align}
\bigcup_{i=0}^r S_i = \lc 1, 2, \cdots, \bar{p} \rc, \quad |S_i|=p_i\quad (0 \leq i \leq r).
\end{align}
When $r=1$, such a partition $\scS$ is called a \emph{$(p_0, p_1)$-shuffle} (cf.~\cite[Definition B.6.1]{MR4406774}).
In this article, we also call $\scS$ with arbitrary $r \geq 1$ a \emph{$(p_0, \cdots, p_r)$-shuffle}.
We also call $(p_0, \cdots, p_r)$ the \emph{degree} of the shuffle $\scS$.

Consider walks on the $p_0 \times \cdots \times p_r$ grid
\begin{align}
\lc \lb j_0, \cdots, j_r \rb \in \bZ^{r+1} \relmid 0 \leq j_i \leq p_i \rc
\end{align}
from $(0, \cdots, 0)$ to $(p_0, \cdots, p_r)$ such that exactly one of the components increases by one in each step.
For such a walk and an integer $l$ such that $0 \leq l \leq \bar{p}$, let $\lb j_0(l), \cdots, j_r(l) \rb$ denote the point of the grid at which we arrive after $l$ times moves.
With a $(p_0, \cdots, p_r)$-shuffle $\scS=\lc S_i \rc_i$, we associate the walk determined by $j_i(0)=0$ $(0 \leq i \leq r)$ and
\begin{align}
j_i(l+1)
=\left\{
\begin{array}{ll}
j_i(l)+1 & l+1 \in S_i\\
j_i(l) & l+1 \nin S_i
\end{array}
\right.
\end{align}
for $l \geq 0$.
This correspondence gives a bijection between $(p_0, \cdots, p_r)$-shuffles and the walks (cf.~\cite[Section B.6.1]{MR4406774}).

For a finite partially ordered set $P$, let $\Lambda \lb P \rb$ denote the abstract simplicial complex associated with $P$, i.e., the vertex set is $P$ and $k$-simplices are the subsets consisting of totally ordered chains of the form $x_0<x_1<\cdots<x_k$ with each $x_j \in P$.
We also write its geometric realization as $|\Lambda \lb P \rb|$, i.e.,
\begin{align}
|\Lambda \lb P \rb|:=
\lc \sum_{x \in P} t_x e_x \in \bR^{|P|} \relmid \sum_{x \in P} t_x=1, t_x \geq 0, \lc x \in P \relmid t_x \neq 0 \rc \in \Lambda \lb P \rb \rc,
\end{align}
where $\lc e_x \rc_{x \in P}$ is the standard basis of $\bR^{|P|}$.

Let $\lc P_i \rc_{i \in I}$ be a family of a finite number of partially ordered sets.
We let the product $\prod_{i \in I} P_i$ have the partial order defined by $\lb x_{1, i} \rb_{i \in I} \leq \lb x_{2, i} \rb_{i \in I}$ if $x_{1, i} \leq x_{2, i}$ for all $i \in I$.
We consider the map
\begin{align}\label{eq:triangulation}
\left| \Lambda \lb \prod_{i \in I} P_i \rb \right| \to 
\prod_{i \in I} \left| \Lambda \lb  P_i \rb \right|, \quad
\sum_{x} t_x e_x \mapsto \prod_{i \in I} \lb \sum_{x} t_x e_{x_i} \rb,
\end{align}
where the sum is over all the elements $x=(x_i)_{i \in I} \in \prod_{i \in I} P_i$.
This map sends the simplex $\lb x_{0, i} \rb_{i \in I} < \cdots < \lb x_{k, i} \rb_{i \in I}$ of $\left| \Lambda \lb \prod_{i \in I} P_i \rb \right|$ into the product of the simplices of $\left| \Lambda \lb P_i \rb \right|$ spanned by the sets of the vertices $\lc x_{j, i} \rc_{j}$.

\begin{lemma}{\rm(cf.~\cite[Lemma B.6.3]{{MR4406774}})}\label{lm:triangulation}
The map \eqref{eq:triangulation} is a piecewise affine isomorphism.
\end{lemma}
\begin{proof}
We prove the lemma by induction on $|I|$.
When $|I|=1$, it is trivial.
When $|I|=2$, it is proved in \cite[Lemma B.6.3]{{MR4406774}}.
We suppose that the lemma holds when $|I|=l$, and show that it holds also when $|I|=l+1$.
Choose an arbitrary element $i_0 \in I$.
One can decompose the map \eqref{eq:triangulation} for $\lc P_i \rc_{i \in I}$ as follows:
\begin{align}
\left| \Lambda \lb P_{i_0} \times \prod_{i \in I \setminus \lc i_0 \rc} P_i \rb \right| 
\to \left| \Lambda \lb  P_{i_0} \rb \right| \times \left| \Lambda \lb \prod_{i \in I \setminus \lc i_0 \rc} P_i \rb \right|
\to \left| \Lambda \lb  P_{i_0} \rb \right| \times \prod_{i \in I \setminus \lc i_0 \rc} \left| \Lambda \lb  P_i \rb \right|,
\end{align}
where the former map is the map \eqref{eq:triangulation} for $\lc P_{i_0}, \prod_{i \in I \setminus \lc i_0 \rc} P_i \rc$, and the latter map is the product of the identity map and the map \eqref{eq:triangulation} for $\lc P_i \rc_{i \in I \setminus \lc i_0 \rc}$.
One can see that both maps are piecewise affine isomorphisms by \cite[Lemma B.6.3]{{MR4406774}} and by the induction hypothesis respectively.
Thus we obtain the lemma.
\end{proof}

We set $I:= \lc 0, \cdots, r \rc$, and let $\scS=\lc S_i \relmid i \in I \rc$ be a $(p_0, \cdots, p_r)$-shuffle.
We write $S_i=\lc  y_{i, 1}< \cdots < y_{i,  p_i} \rc \subset \lc 1, \cdots, \bar{p} \rc$.
For each $i \in I$, we also consider a totally ordered set $P_i:= \lc x_{i, 0}< \cdots < x_{i,  p_i} \rc$ consisting of $p_i+1$ elements $x_{i,  j}$ $(0 \leq j \leq p_i)$.
The abstract simplicial complex $\Lambda \lb P_i \rb$ is the standard simplex of dimension $p_i$.
For each $i \in I$, let 
\begin{align}
\eta_{\scS, i} \colon \Delta^{\bar{p}}:=\lc \sum_{j=0}^{\bar{p}} t_j e_j \in \bR^{\bar{p}+1} \relmid \sum_{j=0}^{\bar{p}} t_j=1, t_j \geq 0 \rc
\to
\left| \Lambda \lb P_i \rb \right|
\end{align}
be the affine map sending each vertex $e_j$ to $e_{x_{i, k}} \in \left| \Lambda \lb P_i \rb \right|$ with $k$ such that $y_{i, k} \leq j < y_{i, k+1}$ (letting $y_{i, 0}:=0, y_{i, p_i+1}:=\bar{p}+1$).
We set
\begin{align}\label{eq:simplex}
\eta_{\scS}:=\prod_{i \in I} \eta_{\scS, i} \colon \Delta^{\bar{p}} \to \prod_{i \in I} \left| \Lambda \lb P_i \rb \right|.
\end{align}
By \pref{lm:triangulation}, the target of \eqref{eq:simplex} (the product of standard simplices) has a triangulation induced by the map \eqref{eq:triangulation} and the simplicial structure of $\Lambda \lb \prod_{i \in I} P_i \rb$.
The triangulation coincides with the collection of the images of the maps $\eta_\scS$ of \eqref{eq:simplex} for all the $(p_0, \cdots, p_r)$-shuffles $\scS$.
This fact is proved in \cite[Corollary B.6.4]{MR4406774} for the case $r=1$.
One can show it also for general $r \geq 1$ in the same way.

\section{Toric degenerations of Calabi--Yau complete intersections}\label{sc:toric-CY}

We recall the construction by Gross \cite{MR2198802} of toric degenerations of Calabi--Yau complete intersections,  and state the precise setup of \pref{th:main}.
Let $d$ and $r$ be positive integers.
Consider a free $\bZ$-module $M$ of rank $d+r$ and its dual $N:=\Hom(M, \bZ)$.
Let $\Delta \subset M_\bR$ be a reflexive polytope, and $\Delta^\ast := \lc n \in N_\bR \relmid \la \Delta, n \ra \geq -1 \rc$ be its polar polytope.
Let further $\Sigma \subset N_\bR$ be the normal fan of $\Delta$.
Consider the convex piecewise linear function $\varphi \colon N_\bR \to \bR$ that corresponds to the anti-canonical sheaf on the toric variety associated with the fan $\Sigma$.
Let $\lc e_1, \cdots, e_l \rc$ be the set of primitive generators of one-dimensional cones of the fan $\Sigma$.
Then we have $\varphi(e_i)=1$ for any $i \in \lc 1, \cdots, r \rc$.
A Minkowski decomposition $\Delta =\Delta_1+ \cdots + \Delta_r$ is called a \emph{nef-partition} if the induced decomposition $\varphi = \varphi_1 + \cdots + \varphi_r$ satisfies $\varphi_i(e_j) \in \lc 0, 1\rc$ for any $i \in \lc 1, \cdots, r\rc$ and $j \in \lc 1, \cdots, l\rc$. 
Let $\nabla_i \subset N_\bR$ be the convex hull of $0 \in N_\bR$ and all $e_j$ such that $\varphi_i(e_j) =1$.
Then the Minkowski sum $\nabla:=\nabla_1 + \cdots + \nabla_r$ becomes a reflexive polytope.
We also write its polar polytope as $\nabla^\ast \subset M_\bR$.
Let $\Sigmav \subset M_\bR$ be the normal fan of $\nabla$, and $\varphiv \colon M_\bR \to \bR$ be the piecewise linear function corresponding to the anti-canonical sheaf on the toric variety associated with the fan $\Sigmav$.
We also write the decomposition of $\varphiv$ induced by $\nabla=\nabla_1 + \cdots + \nabla_r$ as $\varphiv = \varphiv_1 + \cdots + \varphiv_r$.
For lattice polytopes $\mu \subset \partial \Delta^\ast$ and $\eta \subset \partial \nabla^\ast$, we define
\begin{align}
\beta_i^\ast(\mu)&:= \lc n \in \mu \relmid \varphi_i(n)=1 \rc \subset \nabla_i \cap \partial \Delta^\ast, \quad \rotatebox[origin=c]{180}{$\beta$} (\mu):=\sum_{i=1}^r \beta_i^\ast(\mu) \subset \nabla, \\
\rotatebox[origin=c]{180}{$\beta$} _i^\ast(\eta)&:= \lc m \in \eta \relmid \check{\varphi}_i(m)=1 \rc \subset \Delta_i \cap \partial \nabla^\ast, \quad \beta (\eta):=\sum_{i=1}^r \rotatebox[origin=c]{180}{$\beta$} _i^\ast(\eta) \subset \Delta.
\end{align}
The subsets $\beta_i^\ast(\mu), \rotatebox[origin=c]{180}{$\beta$} _i^\ast(\eta)$ are faces of $\mu, \eta$ respectively.

We take coherent subdivisions $\Sigma' \subset N_\bR, \Sigmav' \subset M_\bR$ of the fans $\Sigma, \Sigmav$ whose fan polytopes (i.e., the convex hulls of primitive generators of all one-dimensional cones) remain $\Delta^\ast, \nabla^\ast$ respectively, and consider an integral piecewise linear function $\check{h} \colon M_\bR \to \bR$ that is strictly convex on the fan $\Sigmav'$.
We also assume that the function $\check{h}':=\check{h}-\varphiv \colon M_\bR \to \bR$ is convex (not necessarily strict convex) on $\Sigmav'$.
Let
\begin{align}
\nabla^{\check{h}}&:=\lc n \in N_\bR \relmid \la m, n \ra \geq -\check{h}(m), \forall m \in M_\bR \rc \\
\nabla^{\check{h}'}&:=\lc n \in N_\bR \relmid \la m, n \ra \geq -\check{h}'(m), \forall m \in M_\bR \rc
\end{align}
be the Newton polytopes of $\check{h}$ and $\check{h}'$, and set
\begin{align}
\scrR^{\check{h}}_{\Delta^\ast}&:= \lc (F_1, F_2) \relmid F_1 \prec \Delta^\ast,  F_2 \prec \nabla^{\check{h}'}\ \mathrm{such\ that}\ \rotatebox[origin=c]{180}{$\beta$} (F_1) \neq \emptyset, F_1+F_2 \prec \Delta^\ast + \nabla^{\check{h}'}\rc,\\
\scrP^{\check{h}}_{\Delta^\ast}&:= \lc \rotatebox[origin=c]{180}{$\beta$} (F_1)+F_2 \relmid (F_1, F_2) \in \scrR^{\check{h}}_{\Delta^\ast} \rc, \\
B^{\check{h}}_\nabla&:=\bigcup_{(F_1, F_2) \in \scrR^{\check{h}}_{\Delta^\ast}} \rotatebox[origin=c]{180}{$\beta$} (F_1)+F_2 \subset N_\bR.
\end{align}
One has $B^{\check{h}}_\nabla \subset \partial \nabla^{\check{h}}$ (\cite[Corollary 3.3]{MR2198802}).
We also consider
\begin{align}\label{eq:t-Delta-i}
\tilde{\Delta}_i:=\lc (m, l) \in M_\bR \oplus \bR \relmid m \in \Delta_i, l \geq \check{h}'(m) \rc, \quad 
\tilde{\Delta}:=\sum_{i=1}^r \tilde{\Delta}_i
\end{align}
and the normal fan $\tilde{\Sigma} \subset N_\bR \oplus \bR$ of $\tilde{\Delta}$.
The fan $\tilde{\Sigma}$ is the union of
\begin{enumerate}
\item $\lc \cone \lb F_1 \rb \times \lc 0 \rc \relmid F_1 \prec \Delta^\ast \rc$,
\item $\lc \cone(F_1) \times \lc 0 \rc + \cone \lb F_2 \times \lc 1 \rc \rb \relmid F_1 \prec \Delta^\ast,  F_2 \prec \nabla^{\check{h}'}, F_1+F_2 \prec \Delta^\ast + \nabla^{\check{h}'} \rc$, and 
\item $\lc \cone \lb F_2 \times \lc 1 \rc \rb \relmid F_2 \prec \nabla^{\check{h}'} \rc$
\end{enumerate}
(\cite[Proposition 3.7]{MR2198802}).
The polytope $\tilde{\Delta}$ determines a line bundle on the toric variety associated with $\tilde{\Sigma}$.
It is induced by the piecewise linear function $\tilde{\varphi}$ on $\tilde{\Sigma}$ defined by
\begin{align}\label{eq:tvp}
\tilde{\varphi} (\tilde{n}):=-\inf_{\tilde{m}\in \tilde{\Delta}} \la \tilde{m},\tilde{n}\ra,
\end{align}
where $\tilde{n} \in N_\bR \oplus \bR$.
The Minkowski decomposition $\tilde{\Delta}=\sum_{i=1}^r \tilde{\Delta}_i$ induces a decomposition $\tilde{\varphi}=\sum_{i=1}^r \tilde{\varphi}_i$ with $\tilde{\varphi}_i \lb \lb n , 0 \rb \rb=\varphi_i(n)$ and $\tilde{\varphi}_i \lb \lb n , 1 \rb \rb=0$ for any $n \in \nabla^{\check{h}'}$.

We take a subdivision $\tilde{\Sigma}'$ of the fan $\tilde{\Sigma} \subset N_\bR \oplus \bR$ so that it satisfies the following condition:
\begin{condition}\label{cd:orig}
The following hold:
\begin{enumerate}
\item The fan $\lc C \cap (N_\bR \oplus \lc 0\rc) \relmid C \in \tilde{\Sigma}' \rc$ coincides with the fan $\Sigma'$.
\item Every $1$-dimensional cone of $\tilde{\Sigma}'$ not contained in $N_\bR \oplus \lc 0\rc$ is generated by a primitive vector $(n, 1)$ with $n \in \nabla^{\check{h}'} \cap N$.
\end{enumerate}
\end{condition}
Such a subdivision $\tilde{\Sigma}'$ is called \emph{good} in \cite[Definition 3.8]{MR2198802}.
In this article, we suppose that one can take a good subdivision $\tilde{\Sigma}'$ that also satisfies the following:

\begin{condition}\label{cd:add}
The fan $\tilde{\Sigma}'$ is unimodular, i.e., all the cones are generated by a subset of a basis of the lattice $N \oplus \bZ \subset N_\bR \oplus \bR$.
(In particular, the subfan $\Sigma' \subset \tilde{\Sigma}'$ is required to be unimodular.)
\end{condition}

We fix such a subdivision $\tilde{\Sigma}'$ in the following.
The fan $\tilde{\Sigma}'$ consists of the following three sorts of cones (\cite[Observation 3.9]{MR2198802}):
\begin{enumerate}
\item cones of the form $\cone(\mu) \times \lc 0 \rc$ for $\cone(\mu) \in \Sigma'$ with $\mu \subset \partial \Delta^\ast$,
\item cones of the form $\cone(\mu) \times \lc 0 \rc + \cone \lb \nu \times \lc 1 \rc \rb$ where $\cone(\mu) \in \Sigma'$ for $\mu \subset \partial \Delta^\ast,  \nu \subset \partial \nabla^{\check{h}'}$, and $\mu + \nu$ is contained in a face of $\Delta^\ast + \nabla^{\check{h}'}$, and
\item cones contained in $\cone(\nabla^{\check{h}'} \times \lc 1\rc)$.
\end{enumerate}
Cones of the second type with $\rotatebox[origin=c]{180}{$\beta$}(\mu) \neq \emptyset$ are called \textit{relevant}.
We set
\begin{align}
\scrR(\tilde{\Sigma}')&:=\lc (\mu, \nu) \relmid 
\begin{array}{l}
\mu \subset \partial \Delta^\ast, \nu \subset \partial \nabla^{\check{h}'} \\
\cone(\mu) \times \lc 0 \rc + \cone \lb \nu \times \lc 1 \rc \rb \mathrm{\ is\ a\ relevant\ cone\ in\ } \tilde{\Sigma}'
\end{array}
\rc, \\
\scrP(\tilde{\Sigma}')&:=\lc \rotatebox[origin=c]{180}{$\beta$} (\mu) + \nu \relmid (\mu, \nu) \in \scrR(\tilde{\Sigma}') \rc.
\end{align}
Then one has
\begin{align}
B^{\check{h}}_\nabla=\bigcup_{(\mu, \nu) \in \scrR(\tilde{\Sigma}')} \rotatebox[origin=c]{180}{$\beta$} (\mu) + \nu,
\end{align}
and $\scrP(\tilde{\Sigma}')$ is a polyhedral decomposition of $B^{\check{h}}_\nabla$ \cite[Proposition 3.12, Definition 3.13]{MR2198802}.

For each $\tau \in \scrP(\tilde{\Sigma}')$, we take an element $a_\tau \in \rint (\tau)$, and consider the subdivision $\widetilde{\scrP} ( \tilde{\Sigma}' )$ of $\scrP ( \tilde{\Sigma}' )$ defined by
\begin{align}
\widetilde{\scrP} ( \tilde{\Sigma}' ):=\lc \conv \lb \lc a_{\tau_0}, a_{\tau_1}, \cdots, a_{\tau_l} \rc \rb \relmid \tau_0 \prec \tau_1 \prec \cdots \prec \tau_l, l \geq 0, \tau_i \in \scrP ( \tilde{\Sigma}' ) \rc.
\end{align}
We also set
\begin{align}
\Gamma(\tilde{\Sigma}'):=\bigcup_{\substack{\tau_0 \prec \tau_1 \prec \cdots \prec \tau_l, l \geq 0, \tau_i \in \scrP ( \tilde{\Sigma}' ) \\ \dim(\tau_0) \geq 1, \dim(\tau_l) \leq d-1}} \conv \lb \lc a_{\tau_0}, a_{\tau_1}, \cdots, a_{\tau_l} \rc \rb.
\end{align}
For each vertex $v \in \scrP(\tilde{\Sigma}')$, let $W_v^\circ$ be the union of interiors of all simplices of $\widetilde{\scrP}(\tilde{\Sigma}')$ containing $v$.
The vertex $v$ is written as $v=\rotatebox[origin=c]{180}{$\beta$} (\mu_v) + \nu_v$ with $(\mu_v, \nu_v) \in \scrR(\tilde{\Sigma}')$.
We define a chart on $W_v^\circ$ as the restriction of the projection
\begin{align}\label{eq:fanstr}
\psi_v \colon N_\bR \to N_\bR / \vspan (\beta^\ast_1(\mu_v), \cdots, \beta^\ast_r(\mu_v))
\end{align}
to $W_v^\circ$.
For each maximal-dimensional face $\sigma \in \scrP(\tilde{\Sigma}')$ of $B^{\check{h}}_\nabla$, we also define a chart
\begin{align}\label{eq:int-aff}
\psi_ \sigma \colon \rint(\sigma) \hookrightarrow \aff(\sigma)
\end{align}
via the inclusion.
These charts define an integral affine structure on $B^{\check{h}}_\nabla \setminus \Gamma(\tilde{\Sigma}')$.
It makes $B^{\check{h}}_\nabla$ an integral affine manifold with singularities of dimension $d$ \cite[Proposition 3.14]{MR2198802}.

Let $X_{\tilde{\Sigma}'}$ be the toric variety over $k$ associated with the fan $\tilde{\Sigma}'$, and $t$ be the regular function on $X_{\tilde{\Sigma}'}$ corresponding to $(0, 1) \in M \oplus \bZ$, which defines the morphism $f \colon X_{\tilde{\Sigma}'} \to \bA^1$.
Let further $\scL_i$ be the line bundle on $X_{\tilde{\Sigma}'}$ associated with the polytope $\tilde{\Delta}_i$, and take a general section $s_i \in \Gamma \lb X_{\tilde{\Sigma}'},  \scL_i \rb$ for every $i \in \lc 1, \cdots, r \rc$.
We consider the variety $\scX \subset X_{\tilde{\Sigma}'}$ defined by
\begin{align}\label{eq:tss}
ts_1+s_1^0= \cdots =ts_r+s_r^0=0,
\end{align}
where $s_i^0 \in \Gamma \lb X_{\tilde{\Sigma}'},  \scL_i \rb$ is the section defined by $(0, 0) \in \tilde{\Delta}_i$.
By restricting the morphism $f \colon X_{\tilde{\Sigma}'} \to \bA^1$ to $\scX$ and taking the base change to $R$, we obtain a morphism $\scX \to \Spec R$.
By abuse of notation, we also write it as $f \colon \scX \to \Spec R$.
This is a toric degeneration of Calabi--Yau varieties \cite[Theorem 3.10]{MR2198802}, and $(B^{\check{h}}_\nabla, \scrP(\tilde{\Sigma}'))$ forms its dual intersection complex \cite[Proposition 3.14]{MR2198802}.
We refer the reader to \cite[Section 4]{MR2213573} for the definitions of toric degenerations and their dual intersection complexes.

Let $f' \colon X \to \Spec K$ be the base change of the toric degeneration $f \colon \scX \to \Spec R$ to $K$.
The scheme $X$ is a subscheme of the toric variety $X_{\Sigma'}$ over $K$ associated with the fan $\Sigma'$.
Since the fan $\Sigma'$ is unimodular by \pref{cd:add}, the ambient toric variety $X_{\Sigma'}$ is smooth.
By Bertini's theorem, one can see that $X$ is also smooth.
We suppose that $X$ is irreducible.
This is satisfied, for instance, when the nef-partition $\Delta =\Delta_1+ \cdots + \Delta_r$ is $2$-independent in the sense of \cite[Definition 3.1]{MR1463173} (cf.~\cite[Theorem 3.3(ii)]{MR1463173}).
Let $\scX_k:=\scX \times_R k$ be the special fiber of the toric degeneration. 

\begin{proposition}\label{pr:canonical}
The logarithmic relative canonical divisor $K_{\scX/R}+\scX_{k, \mathrm{red}}$ of $\scX$ is linearly equivalent to $0$.
\end{proposition}
\begin{proof}
Since the fan $\Sigma'$ is unimodular, the set of the primitive generators of $1$-dimensional cones in $\Sigma'$ coincides with $\partial \Delta^\ast \cap N$.
By
\begin{align}
\partial \Delta^\ast \cap N&=\bigsqcup_{i=1}^r \lb \nabla_i \cap N \setminus \lc 0 \rc \rb, 
\quad \nabla_i \cap \partial \Delta^\ast=\lc n \in \partial \Delta^\ast \relmid \varphi_i(n)=1 \rc
\end{align}
\cite[Corollary 3.23, Proposition 3.19]{MR2405763}, we can get
\begin{align}\label{eq:varphi}
\varphi_i(n)
=\left\{
\begin{array}{ll}
1 & n \in \nabla_i \cap \partial \Delta^\ast \cap N, \\
0 & n \in \lb \partial \Delta^\ast \cap N \rb \setminus \lb \nabla_i \cap \partial \Delta^\ast \cap N \rb.
\end{array}
\right.
\end{align}
Since $\tilde{\varphi}_i \lb \lb n , 0 \rb \rb=\varphi_i(n)$ and $\tilde{\varphi}_i \lb \lb n , 1 \rb \rb=0$ for any $n \in \nabla^{\check{h}'}$, we can see from \eqref{eq:varphi} that the divisor defined by $ts_i+s_i^0$ in $X_{\tilde{\Sigma}'}$ is linearly equivalent to
\begin{align}
\sum_{\rho \in \tilde{\Sigma}'(1, i)} D_\rho,
\end{align}
where $\tilde{\Sigma}'(1, i)$ denotes the set of of $1$-dimensional cones in $\tilde{\Sigma'}$ whose primitive generators are contained in $\nabla_i \times \lc 0 \rc \lb \subset N_\bR \oplus \bR \rb$, and $D_\rho$ denotes the toric divisor on the variety $X_{\tilde{\Sigma}'}$ corresponding to the cone $\rho \in \tilde{\Sigma}'(1, i)$.
One also has
\begin{align}
\omega_{X_{\tilde{\Sigma}'}} \simeq \scO_{X_{\tilde{\Sigma}'}} \lb -\sum_{\rho \in \tilde{\Sigma}'(1)} D_\rho \rb,
\end{align}
where $\tilde{\Sigma}'(1)$ denotes the set of $1$-dimensional cones in $\tilde{\Sigma}'$ (cf.~e.g.~\cite[Theorem 8.2.3]{MR2810322}).
By using the adjunction formula repeatedly ($r$ times), one can get
\begin{align}
\omega_\scX \simeq i^\ast \scO_{X_{\tilde{\Sigma}'}} \lb -\sum_{\rho \in \tilde{\Sigma}'(1, 0)} D_\rho \rb,
\end{align}
where $i \colon \scX \hookrightarrow X_{\tilde{\Sigma}'}$ is the inclusion, and $\tilde{\Sigma}'(1, 0)$ denotes the set of $1$-dimensional cones in $\tilde{\Sigma}'$, which are not contained in $N_\bR \times \lc 0 \rc$.
By \pref{cd:orig}(2), one also has
\begin{align}
f^\ast \omega_{\bA^1} \simeq i^\ast \scO_{X_{\tilde{\Sigma}'}} \lb -\sum_{\rho \in \tilde{\Sigma}'(1, 0)} D_\rho \rb.
\end{align}
By the adjunction formula (cf.~e.g.~\cite[Section 6.4.2, Theorem 4.9 (a)]{MR1917232})
\begin{align}
\omega_{\scX} \simeq \omega_{\scX/\bA^1} \otimes_{\scO_\scX} f^\ast \omega_{\bA^1},
\end{align}
one can see that the canonical sheaf $\omega_{\scX/\bA^1}$ is trivial.
By taking the base change (cf.~e.g.~\cite[Section 6.4.2, Theorem 4.9 (b)]{MR1917232}), we can also see that $\omega_{\scX/R}$ is trivial.
Since $\scX_k$ is reduced and the principal divisor defined by the function $t$, we can conclude the claim of the lemma.
\end{proof}

\section{Tropical contractions}\label{sc:contraction}

We work in the same setup and use the same notation as \pref{sc:toric-CY}.
The \emph{tropicalization} $\trop (X)$ of $X$ is the image of the Berkovich analytification $X^{\an} \subset \lb X_{\Sigma'} \rb^{\an}$ of $X$ by the tropicalization map
\begin{align}\label{eq:trops}
\trop \colon \lb X_{\Sigma'} \rb^{\an} \to X_{\Sigma'} \lb \bT \rb,
\end{align}
which we recalled in \pref{sc:trop}.
The following is one of the main theorems of \cite{Yam21}.

\begin{theorem}{\rm(\cite[Theorem 1.2]{Yam21})}\label{th:contraction}
One has $B^{\check{h}}_\nabla \subset \trop (X)$, and there exists a proper continuous map 
\begin{align}
\delta \colon \trop (X) \to B^{\check{h}}_\nabla
\end{align}
called a \emph{tropical contraction}, which preserves the integral affine structures.
\end{theorem}
The tropical contractions also preserve invariants of tropical spaces such as tropical (co)homology groups and eigenwave/radiance obstructions \cite[Corollary 1.3, Corollary 1.4]{Yam21}.
We refer the reader to \cite[Theorem 3.19]{Yam21} for more details of the properties of tropical contractions.
The tropical contraction of \pref{th:contraction} is constructed explicitly in \cite[Section 5.3]{Yam21}.
We briefly review the construction in the following.
We refer the reader to \cite[Section 5.3]{Yam21} for more details of the construction.
Some examples are also given in \cite[Example 5.16, Section 3.3]{Yam21}.

Let $f_i \colon N_\bR \to \bR$ $(1 \leq i \leq r)$ be the tropical polynomial defined by
\begin{align}
f_i (n) :=\min_{m \in \Delta_i \cap M} \lc \check{h}(m)+\la m, n \ra \rc,
\end{align}
and $X \lb f_i \rb^\circ \subset N_\bR$ be the tropical hypersurface defined by $f_i$, i.e., the corner locus of the function $f_i$.
We write the closure of the stable intersection (cf.~\cite[(1.4)]{Yam21}) of the tropical hypersurfaces $X \lb f_i \rb^\circ$ $(1 \leq i \leq r)$ in the tropical toric variety $X_{\Sigma'} (\bT) (\supset N_\bR)$ as $X(f_1, \cdots, f_r)$.
In \cite[Proposition 5.2]{Yam21}, it is shown that we have
\begin{align}
\trop(X)=X(f_1, \cdots, f_r).
\end{align}
For every polyhedron $\tau=\rotatebox[origin=c]{180}{$\beta$} (\mu_\tau) + \nu_\tau \in \scrP(\tilde{\Sigma}')$, we set
\begin{align}
C_\tau:=\cone \lb \mu_\tau \rb \in \Sigma'.
\end{align}
Let $U_\tau \subset B^{\check{h}}_\nabla$ denote the open star of $a_{\tau}$ in $\widetilde{\scrP}(\tilde{\Sigma}')$, i.e.,
\begin{align}
U_\tau := \bigcup_{\substack{\tau_0 \prec \tau_1 \prec \cdots \prec \tau_l, \\ l \geq 0, \tau_i \in \scrP(\tilde{\Sigma}'), \\ \tau \in \lc \tau_0, \cdots, \tau_l \rc}} \rint \lb \conv \lb \lc a_{\tau_0}, a_{\tau_1}, \cdots, a_{\tau_l} \rc \rb \rb.
\end{align}
For each face $\tau' \prec \tau$, we define
\begin{align}
W_{\tau', \tau}^\circ&:=\bigcup_{\substack{\tau' \prec \tau_1 \prec \cdots \prec \tau_l, \\ l \geq 0, \tau_i \in \scrP(\tilde{\Sigma}'), \\ \tau \in \lc \tau', \tau_1, \cdots, \tau_l \rc}} \rint \lb \conv \lb \lc a_{\tau'}, a_{\tau_1}, \cdots, a_{\tau_l} \rc \rb \rb.
\end{align}
Then one has $U_\tau =\bigsqcup_{\tau' \prec \tau} W_{\tau', \tau}^\circ$.
We also define 
\begin{align}
V_{\tau', \tau}^\circ&:=\lb W_{\tau', \tau}^\circ + \cone_\bT \lb \bigcup_{i=1}^r \beta_i^\ast \lb \mu_{\tau'} \rb \rb
 \rb 
 \cap X(f_1, \cdots, f_r) \subset X_{C_{\tau'}}(\bT) \subset X_{\Sigma'}(\bT), \\
X_\tau^\circ&:=\bigcup_{\tau' \prec \tau} V_{\tau', \tau}^\circ \subset X(f_1, \cdots, f_r) \cap X_{C_\tau}(\bT).
\end{align}
Then one has $X(f_1, \cdots, f_r)=\bigcup_{\tau \in \scrP(\tilde{\Sigma}')} X_\tau^\circ$ (\cite[Lemma 5.15]{Yam21}).
The restriction of the projection map $\pi_{C_{\tau'}} \colon X_{C_{\tau'}}(\bT) \to O_{C_{\tau'}} (\bT)$ of \eqref{eq:proj} to the subset $W_{\tau', \tau}^\circ$ is injective (\cite[Lemma 5.12, Lemma 3.20]{Yam21}), and there uniquely exists a map
\begin{align}\label{eq:t}
t_{\tau'} \colon \pi_{C_{\tau'}} \lb W_{\tau', \tau}^\circ \rb \to W_{\tau', \tau}^\circ
\end{align}
such that $t_{\tau'} \circ \pi_{C_{\tau'}}$ is the identity map on $W_{\tau', \tau}^\circ$.
The restriction of the tropical contraction $\delta$ to $V_{\tau', \tau}^\circ$ is given by the composition
\begin{align}\label{eq:local-delta}
V_{\tau', \tau}^\circ \hookrightarrow X_{C_{\tau'}}(\bT) \xrightarrow{\pi_{C_{\tau'}}} \pi_{C_{\tau'}}\lb W_{\tau', \tau}^\circ \rb \xrightarrow{t_{\tau'}} W_{\tau', \tau}^\circ.
\end{align}
We show some lemmas that will be used in the proof of \pref{th:main}.

\begin{lemma}\label{lm:delta-1}
One has $\delta^{-1} \lb U_\tau \rb=X_{\tau}^\circ$.
\end{lemma}
\begin{proof}
It is obvious that we have $\delta^{-1} \lb U_\tau \rb \supset X_{\tau}^\circ$.
We will show $\delta^{-1} \lb U_\tau \rb \subset X_{\tau}^\circ$.
Let $n \in \delta^{-1} \lb U_\tau \rb$ be an arbitrary point.
Then we have $n \in X_{\tau_n}^\circ$ for some $\tau_n \in \scrP(\tilde{\Sigma}')$, and $n \in V_{\tau', \tau_n}^\circ$ for some $\tau' \prec \tau_n$.
We can see from \eqref{eq:local-delta} that we have $\delta (n) \in W_{\tau', \tau_n}^\circ$, and
\begin{align}\label{eq:ndn}
n \in \delta (n)+ \cone_\bT \lb \bigcup_{i=1}^r \beta_i^\ast \lb \mu_{\tau'} \rb \rb.
\end{align}
Since we have $\delta (n) \in U_\tau$, we also get
\begin{align}\label{eq:ndn2}
\delta (n)+ \cone_\bT \lb \bigcup_{i=1}^r \beta_i^\ast \lb \mu_{\tau'} \rb \rb
\subset
W_{\tau', \tau_n, \tau}^\circ+ \cone_\bT \lb \bigcup_{i=1}^r \beta_i^\ast \lb \mu_{\tau'} \rb \rb
\subset V_{\tau', \tau}^\circ \subset X_\tau^\circ,
\end{align}
where
\begin{align}
W_{\tau', \tau_n, \tau}^\circ:=
\bigcup_{\substack{\tau' \prec \tau_1 \prec \cdots \prec \tau_l, \\ l \geq 0, \tau_i \in \scrP(\tilde{\Sigma}'), \\ \tau_n, \tau \in \lc \tau', \tau_1, \cdots, \tau_l \rc}} \rint \lb \conv \lb \lc a_{\tau'}, a_{\tau_1}, \cdots, a_{\tau_l} \rc \rb \rb.
\end{align}
By \eqref{eq:ndn} and \eqref{eq:ndn2}, we get $n \in X_\tau^\circ$.
We obtained $\delta^{-1} \lb U_\tau \rb \subset X_{\tau}^\circ$.
\end{proof}

\begin{lemma}\label{lm:sigma-cone}
For any maximal-dimensional polyhedron $\sigma \in \scrP(\tilde{\Sigma}')$, one has
\begin{align}\label{eq:sigma-cone}
\lb \rint \lb \sigma \rb + \cone_\bT \lb \bigcup_{i=1}^r \beta_i^\ast \lb \mu_{\sigma} \rb \rb
 \rb 
 \cap X(f_1, \cdots, f_r)
 =\rint \lb \sigma \rb.
\end{align}
We also have
\begin{align}\label{eq:sigma-cone2}
\lb \rint \lb \sigma \rb + \cone_\bT \lb \bigcup_{i=1}^r \beta_i^\ast \lb \mu_{v} \rb \rb
 \rb 
 \cap X(f_1, \cdots, f_r)
 =\rint \lb \sigma \rb
\end{align}
for any vertex $v \prec \sigma$.
\end{lemma}
\begin{proof}
First, we show \eqref{eq:sigma-cone}.
It is obvious that the right hand side is contained in the left hand side, since $\rint (\sigma) \subset B^{\check{h}}_\nabla \subset \trop (X) = X(f_1, \cdots, f_r)$ and $0 \in \cone_\bT \lb \bigcup_{i=1}^r \beta_i^\ast \lb \mu_{\sigma} \rb \rb$.
We check the opposite inclusion.
Let $n$ be an arbitrary element of the left hand side.
It can be written as
\begin{align}\label{eq:n}
n=n_0+\sum_{j=1}^r \sum_{k} \lambda_{j, k} \cdot n_{j, k},
\end{align}
where $n_0 \in \rint \lb \sigma \rb$, $\lambda_{j, k} \in \bR_{\geq 0}$, $n_{j, k} \in \beta^\ast_j \lb \mu_\sigma \rb$, or as the limit of the right hand side of $\eqref{eq:n}$ as some $\lambda_{j, k}$ approach to $\infty$.
By \cite[Lemma 5.6.1]{Yam21} for $\sigma$, we can see that for each $i \in \lc 1, \cdots, r \rc$, there uniquely exists an element $m_i \in \lb \Delta_i \cap M \rb \setminus \lc 0 \rc$ such that $\check{h}(m_i)+\la m_i, \bullet \ra$ and $\check{h}(0)+\la 0, \bullet \ra \equiv 0$ are the only tropical monomials of the tropical polynomial $f_i$, which attain the minimum of $f_i$ at $n_0 \in \rint \lb \sigma \rb$.
By \cite[Lemma 5.5]{Yam21} for $\sigma$, we have $\la m_i, n_{j, k} \ra = -\delta_{i, j}$.
By using this, one can get
\begin{align}
\check{h}(m_i)+\la m_i, n \ra&=\check{h}(m_i)+\la m_i, n_0 \ra+\sum_{j=1}^r \sum_{k} \lambda_{j, k} \la m_i, n_{j, k}  \ra=-\sum_{k} \lambda_{i, k}
\end{align}
for the element $n$ of \eqref{eq:n}.
On the other hand, by \cite[Proposition 3.13]{MR2405763}, we also have $\la m, n_{j, k} \ra \geq -\delta_{i, j}$ for any $m \in \lb \Delta_i \cap M \rb$, from which we can get
\begin{align}
\check{h}(m)+\la m, n \ra&=\check{h}(m)+\la m, n_0 \ra+\sum_{j=1}^r \sum_{k} \lambda_{j, k} \la m, n_{j, k} \ra>c-\sum_{k} \lambda_{i, k}.
\end{align}
for any $m \in \lb \Delta_i \cap M \rb \setminus \lc 0, m_i \rc$, where $c$ is some positive real constant.
Therefore, if some $\lambda_{i,k}$ is greater than $0$, then $\check{h}(m_i)+\la m_i, \bullet \ra$ is the only monomial that attains the minimum of $f_i$ at $n$, and the element $n$ is not contained in the closure $X \lb f_i \rb$ of the tropical hypersurface $X \lb f_i \rb^\circ$ in the tropical toric variety $X_{\Sigma'} (\bT)$.
Thus all the elements $\lambda_{j, k}$ must be $0$, since $X(f_i) \supset X(f_1, \cdots, f_r)$.
We obtained $n=n_0 \in \rint \lb \sigma \rb$, and the equality \eqref{eq:sigma-cone}.

\eqref{eq:sigma-cone2} follows from \eqref{eq:sigma-cone}, since the left hand side of \eqref{eq:sigma-cone2} obviously contains $\rint \lb \sigma \rb$, and is contained in the left hand side of \eqref{eq:sigma-cone}.
\end{proof}

\begin{lemma}\label{lm:delta-1m}
For any maximal-dimensional polyhedron $\sigma \in \scrP(\tilde{\Sigma}')$, 
one has $\delta^{-1} \lb \rint (\sigma) \rb= \rint (\sigma)$.
Furthermore, the restriction of the tropical contraction $\delta$ to $\delta^{-1} \lb \rint (\sigma) \rb= \rint (\sigma)$ is the identity map.
\end{lemma}
\begin{proof}
We have $\rint (\sigma)=U_\sigma$.
We will show $X_\sigma^\circ = \rint (\sigma)$.
The former claim follows from this and \pref{lm:delta-1}.
Since $\beta^\ast_j \lb \mu_{\tau'} \rb \subset \beta^\ast_j \lb \mu_\sigma \rb$ for any $\tau' \prec \sigma$, we can get
\begin{align}
X_\sigma^\circ \subset \lb \rint \lb \sigma \rb + \cone_\bT \lb \bigcup_{i=1}^r \beta_i^\ast \lb \mu_{\sigma} \rb \rb
 \rb 
 \cap X(f_1, \cdots, f_r)
 =\rint \lb \sigma \rb
\end{align}
by \eqref{eq:sigma-cone}.
The opposite inclusion is obvious.
Thus we obtain $X_\sigma^\circ = \rint (\sigma)$.
The latter claim is now obvious from \eqref{eq:local-delta}.
\end{proof}

\section{Proof of Theorem 1.1}\label{sc:proof1}

\subsection{Overview of the proof}\label{sc:overview}

Let us consider
\begin{align}\label{eq:k3}
\scX:=\lc z_0 z_1 z_2 z_3+t \cdot F_4(z_0, z_1, z_2, z_3 )=0 \rc \subset \bP_R^3,
\end{align}
where $z_0, z_1, z_2, z_3$ are the homogeneous coordinates of $\bP_R^3$, and $F_4$ is a general homogeneous polynomial of degree $4$.
This is the toric degeneration of \pref{sc:toric-CY} in the case where $d=2$, $r=1$,
\begin{align}
\Delta&=\nabla^\ast=\conv \lb \lc (3, -1, -1), (-1, 3, -1), (-1, -1, 3), (-1,-1,-1) \rc \rb \subset M_\bR \cong \bR^3, \\
\Delta^\ast&=\nabla=\conv \lb \lc (1, 0, 0), (0, 1, 0), (0, 0, 1), (-1,-1,-1) \rc \rb \subset N_\bR \cong \bR^3,
\end{align}
and $\Sigma'=\Sigma, \Sigmav'=\Sigmav$, $\check{h}=\check{\varphi}, \tilde{\Sigma}'=\tilde{\Sigma}$.
The special fiber is the union of four irreducible components $D_i:=\lc z_i=t=0 \rc$ $(i \in \lc 0, \cdots, 3 \rc)$.
The inclusion $B^{\check{h}}_\nabla \subset \trop (X)$ and a tropical contraction $\delta \colon \trop (X) \to B^{\check{h}}_\nabla$ of \cite[Theorem 1.2]{Yam21} are shown in \cite[Example 5.16, Figure 5.1]{Yam21}.

This example was studied in detail in Section 3 of the first version of \cite{MPS21}.
They observed the following:
The toric degeneration $\scX$ is a minimal dlt-model, and the essential skeleton coincides with $B^{\check{h}}_\nabla$.
However, it is not a good\footnote{In general, we say that a dlt-model $\scX$ is \emph{good} if every irreducible component of  $\scX_{k, \mathrm{red}}$ is $\bQ$-Cartier.} minimal dlt-model since the irreducible components $D_i$ are not $\bQ$-Cartier.
Therefore, one can not simply associate the Berkovich retraction with the toric degeneration $\scX$.
In order to make a good minimal dlt-model out of $\scX$, consider blowing it up along the irreducible components $D_i$ one after the other.
For any permutation $(i, j, k, l)$ of $\lc 0, 1, 2, 3 \rc$, let $\scrX_{i, j, k}$ be the model obtained by blowing up $\scX$ along $D_i, D_j, D_k$ in this order.
Then the model $\scrX_{i, j, k}$ is a minimal snc-model, and we can see from \cite[Theorem 6.1]{MR3946280} that the Berkovich retraction $\rho_{\scrX_{i, j, k}}$ associated with the model $\scrX_{i, j, k}$ is an affinoid torus fibration outside all the vertices of $B^{\check{h}}_\nabla$.
Moreover, one can also show that the Berkovich retraction $\rho_{\scrX_{i, j, k}}$ is an affinoid torus fibration also around the vertex $v_l$ of $B^{\check{h}}_\nabla$ corresponding to the irreducible component $D_l$.
For showing this, they use \cite[Theorem A]{MPS21} which is a generalization of \cite[Proposition 5.4]{MR3946280}.

Pille-Schneider \cite{PS22} further showed that the composition of the tropicalization map with the tropical contraction $\delta$ coincides with the Berkovich retraction $\rho_{\scrX_{i, j, k}}$ around the vertex $v_l$, and proved \pref{th:main} in this case (and also in the higher-dimensional case, see \pref{rm:ps22}).

The outline of our proof of \pref{th:main}(2) for the general case is also basically the same as the one of \cite{PS22}.
The proof consists of the following four steps:
\begin{enumerate}
\item For every vertex $v$ of $B^{\check{h}}_\nabla$, we take an order on the set of all the irreducible components of the special fiber such that the last one is the irreducible component corresponding to $v$. 
\item We construct a minimal snc-model $\scrX$ by blowing up our toric degeneration along irreducible components of the special fiber one after the other in the order.
\item We show that the composition of the tropicalization map with the tropical contraction $\delta$ coincides with the Berkovich retraction $\rho_\scrX$ associated with the model $\scrX$ locally around the vertex $v$.
\item We also show that the Berkovich retraction $\rho_\scrX$ is an affinoid torus fibration around the vertex $v$.
\end{enumerate}

How we blow up our toric degeneration is explained in \pref{sc:blow-up}.
In \pref{sc:minimal}, we show that the model $\scrX$ obtained by the repetitive blow-ups is a minimal snc-model.
In \pref{sc:essential}, we also compute the image of the skeleton of $\scrX$ by the tropicalization map to show that the essential skeleton coincides with $B^{\check{h}}_\nabla$, i.e., \pref{th:main}(1).
The above steps (3) and (4) are done in \pref{sc:tropical-berkovich}.
The proofs in \pref{sc:minimal}-\pref{sc:tropical-berkovich} are based on an explicit computation of the repetitive blow-ups, which is given in \pref{sc:local}.

In relation to the case of \pref{eq:k3} and the higher-dimensional case considered in \cite{PS22}, we need to take some additional care in our general case.
Since the toric degeneration of \pref{eq:k3} (and the one considered in \cite{PS22}) is a minimal dlt-model, we can associate its dual complex, and it coincides with the essential skeleton (cf.~\cite[Proposition 4.1]{PS22}).
In this case, the dual intersection complex in the sense of the Gross--Siebert program is also a simplicial complex, and coincides with the dual complex of the dlt-model as simplicial complexes (cf.~\cite[Proposition 4.3]{PS22}).
However, these are not the case in our general setup.
First, our toric degenerations are not dlt-models in general, and one can not simply associate its dual complex, while we have the dual intersection complex in the sense of the Gross--Siebert program.
Because of this, in this article, we consider the dual complex of the minimal snc-model $\scrX$ obtained by the repetitive blow-ups in order to compute the essential skeleton.
Furthermore, in general, the dual intersection complex in the Gross--Siebert program may contain polytopes that are not a simplex, unlike the dual complex of a dlt-model.
In our setup, it may contain products of simplices.
Our repetitive blow-ups give a subdivision of the dual intersection complex, and this is described in terms of shuffles that we recalled in \pref{sc:shuffle}.
Also for the above step (4), we can not simply apply \cite[Theorem A]{MPS21} in our general setup, unlike in the case of \cite{PS22}, and we need to modify the argument (cf.~ \pref{rm:ps22}).

In the rest of this article, we work in the same setup and use the same notation as in \pref{sc:toric-CY} and \pref{sc:contraction}.

\subsection{Blow-ups  of toric degenerations}\label{sc:blow-up}

First, we fix a vertex $v \in \scrP(\tilde{\Sigma}')$.
It is written as $v=\rotatebox[origin=c]{180}{$\beta$} (\mu_v) + \nu_v$ with $(\mu_v, \nu_v) \in \scrR(\tilde{\Sigma}')$.
The subsets $\beta_i^\ast \lb \mu_v \rb \subset \nabla_i \cap  \partial \Delta^\ast \cap N$ and $\nu_v \subset \partial \nabla^{\check{h}'} \cap N$ consist of a single point.
For each $i \in \lc 0, \cdots, r \rc$, we set
\begin{align}\label{eq:ni}
N_i:=
\left\{
\begin{array}{ll}
\partial \nabla^{\check{h}'} \cap N & i=0, \\
\nabla_i \cap \partial \Delta^\ast \cap N & i \in \lc 1, \cdots, r \rc,
\end{array}
\right.
\end{align}
and we fix a total order on each $N_i$ so that $\nu_v$ and $\beta_i^\ast \lb \mu_v \rb$ become the greatest element in $N_i$ with respect to the order respectively.
By using these orders on $N_i$ and the total order $N_0 \leq N_1 \leq \cdots \leq N_r$, we put the lexicographical order on
\begin{align}\label{eq:prod}
\prod_{i=0}^r N_i.
\end{align}
It is seen in the beginning of the proof of \cite[Theorem 3.10]{MR2198802} that each irreducible component of the special fiber $\scX_k$ is a $d$-dimensional toric stratum in the toric variety $X_{\tilde{\Sigma}'}$, which corresponds to a relevant cone
\begin{align}
\cone(\mu) \times \lc 0 \rc + \cone \lb \nu \times \lc 1 \rc \rb \in  \tilde{\Sigma}'.
\end{align}
The irreducible component corresponds to the vertex  $\rotatebox[origin=c]{180}{$\beta$} (\mu) + \nu$ of the dual intersection complex $B^{\check{h}}_\nabla$.
The sets $\beta_i^\ast \lb \mu \rb \subset \nabla_i \cap \partial \Delta^\ast \cap N$ and $\nu \subset \partial \nabla^{\check{h}'} \cap N$ consist of a single point, and these points define an element in \eqref{eq:prod}.
Hence, all the irreducible components of $\scX_k$ are labeled by elements in \eqref{eq:prod}.
We write the total order on the set of irreducible components of $\scX_k$, which is induced by the the lexicographical order on \eqref{eq:prod} as $\leq_{\scX}$.

Let $l_\scX \in \bZ_{>0}$ be the number of irreducible components of $\scX_k$, and set $I_\scX:= \lc 1, \cdots, l_\scX \rc$.
For $i \in I_\scX$, let $D_i$ denote the irreducible component of $\scX_k$ that is the $i$-th least with respect to the order $\leq_{\scX}$.
We have
\begin{align}
\scX_k=\sum_{i \in I_{\scX}} D_i.
\end{align}
We consider blowing up $\scX$ repeatedly along the irreducible components of $\scX_k$ one by one in the order $\leq_{\scX}$.
To be precise, for each $l \in I_{\scX} \cup \lc 0 \rc$, we define the variety $\scX_l$ and the divisors $D_{i}^l \subset \scX_l$ $(i \in I_{\scX})$ inductively as follows:
\begin{itemize}
\item For $l=0$, we set $\scX_0:=\scX$ and $D_i^0:=D_i$.
\item Suppose that we have defined the variety $\scX_l$ and divisors $D_{i}^l \subset \scX_l$ $(i \in I_{\scX})$ for an element $l \in I_{\scX} \cup \lc 0 \rc$.
We define $\scX_{l+1}$ to be the blow-up of $\scX_l$ along $D_{l+1}^{l}$.
We further define $D_{l+1}^{l+1} \subset \scX_{l+1}$ to be the exceptional divisor of the blow-up, and $D_{i}^{l+1}$ $(i \in I_{\scX} \setminus \lc l+1 \rc)$ to be the strict transform of $D_{i}^{l} \subset \scX_l$.
\end{itemize}
Then we set $\scrX:=\scX_{l_\scX}$ and $\overline{D}_i:=D_i^{l_\scX} \subset \scrX$.

\subsection{Local computations of blow-ups}\label{sc:local}

Let $C:=\cone(\mu) \times \lc 0 \rc + \cone \lb \nu \times \lc 1 \rc \rb \in \tilde{\Sigma}'$ be a relevant cone of maximal dimension $d+r+1$.
The special fiber $\scX_k$ is contained in the union of affine toric varieties $U_C:=\Spec k \ld C^\vee \cap \lb M \oplus \bZ \rb \rd \subset X_{\tilde{\Sigma}'}$ associated with such cones $C$.
We set
\begin{align}
k_0:=\# \lb \nu \cap N \rb-1, \quad k_i:=\# \lb\beta_i^\ast \lb \mu \rb \cap N \rb-1 \quad (1 \leq i \leq r).
\end{align}
Let $\lc n_{0,j} \relmid j \in \lc 0, \cdots, k_0 \rc \rc$ $\lb \subset N_\bR \oplus \bR \rb$ denote all the elements in 
$\lb \nu \cap N \rb \times \lc 1 \rc \lb \subset N_0 \times \lc 1 \rc \rb$.
For each $i \in \lc 1, \cdots, r \rc$, we also let $\lc n_{i,j} \relmid j \in \lc 0, \cdots, k_i \rc \rc$ denote all the elements in $\lb \beta_i^\ast \lb \mu \rb \cap N \rb \times \lc 0 \rc \lb \subset N_i \times \lc 0 \rc \rb$.
Here we assign the indices $j$ to these elements $n_{i, j}$ so that they agree with the total order on $N_i$ that we fixed in \pref{sc:blow-up}.
By \pref{cd:add}, the elements $\lc n_{i, j} \relmid 0 \leq i \leq r, 0\leq j \leq k_i \rc$ form a basis of $N \oplus \bZ$, and $\sum_{i=0}^r k_i=d$.
Let further $\lc m_{i, j} \relmid 0 \leq i \leq r, 0\leq j \leq k_i \rc$ denote its dual basis in $M \oplus \bZ$.
Then one has $t=\prod_{j=0}^{k_0} z^{m_{0,j}}$, and can write the equations of $\scX$ in the affine toric variety $U_C$ as
\begin{align}
f_i \cdot \prod_{j=0}^{k_0} z^{m_{0,j}} =\prod_{j=0}^{k_i} z^{m_{i, j}} \quad (1 \leq i \leq r),
\end{align}
where $f_i$ is proportional to the section $s_i$ (see the proof of \cite[Theorem 3.10]{MR2198802}).
In the following, we will compute how blowing up $\scX$ repeatedly as explained in \pref{sc:blow-up} affects the open subscheme $\scX \cap U_C \subset \scX$.
The resulting scheme is naturally covered by affine open subschemes, which are labeled by shuffles recalled in \pref{sc:shuffle}.

Let $\scrS_{C, 1}$ be the set of $(k_0, k_1, \cdots, k_r)$-shuffles.
(See \pref{sc:shuffle} for the definition of shuffles.)
Let further $\scrS_{C, 2}$ be the set of shuffles whose degree $(p_0, \cdots, p_r)$ satisfies 
\begin{itemize}
\item $p_0 = k_0$,
\item $p_i \leq k_i$ $(1 \leq i \leq r)$, and $p_i \neq k_i$ for some $i \in \lc 1, \cdots, r \rc$.
\end{itemize}
We fix a shuffle $\scS=\lc S_i \relmid 0 \leq i \leq r \rc \in \scrS_{C, 1} \cup \scrS_{C, 2}$, and let $(p_0, \cdots, p_r)$ denote the degree of $\scS$.
We set $\bar{p}:=\sum_{i=0}^r p_i$.
For an integer $l$ $(0 \leq l \leq \bar{p})$, we also let $\lb j_{\scS, 0}(l), \cdots, j_{\scS, r} (l) \rb$ denote the point of the grid at which we arrive after $l$ times moves during the walk corresponding the shuffle $\scS$.
Furthermore, for an integer $l$ $(1 \leq l \leq \bar{p})$, let $i_\scS (l) \in \lc 0, \cdots, r \rc$ denote the number such that $l \in S_{i_\scS (l)}$.
We also set 
\begin{align}\label{eq:ip1}
i_\scS (\bar{p}+1)&:=0\\
j_{\scS, i} (\bar{p}+1)&:=
\left\{
\begin{array}{ll}
j_{\scS, i}(\bar{p})+1=k_0+1 & i=i_\scS (\bar{p}+1)=0,\\
j_{\scS, i}(\bar{p})=p_i & i \in \lc 1, \cdots, r \rc,
\end{array}
\right.
\end{align}
i.e., for the walk corresponding the shuffle $\scS$, we add the move $+(1, 0, \cdots, 0) \in \bZ^{r+1}$ at the end as its $(\bar{p}+1)$-th move.

We set 
\begin{align}\label{eq:us0}
\scU_{\scS, 0}&:=\scX \cap U_C,\\ 
z^{i, j}_0&:=z^{m_{i,j}}\quad (0 \leq i \leq r, 0 \leq j \leq k_i),\\ 
f_{0, i}&:=f_i\quad (1 \leq i \leq r).
\end{align}
(Although \eqref{eq:us0} does not depend of $\scS$, we put $\scS$ as a subscript in \eqref{eq:us0} in order to deal with it uniformly together with schemes $\scU_{\scS, l}$ which will be defined later and depend on $\scS$.)
Let $D_{\scS, 1}$ denote the least (with respect to the order $\leq_{\scX}$) of all the irreducible components of $\scX_k$ intersecting with $\scU_{\scS, 0}$.
When we blow up the variety $\scX$ repeatedly as explained in \pref{sc:blow-up}, the blow-up along (the strict transform $\widetilde{D}_{\scS, 1}$ of) $D_{\scS, 1}$ is the first blow-up that affects $\scU_{\scS, 0}$ among all the blow-ups that we do for $\scX$.
The irreducible component $D_{\scS, 1}$ is the toric stratum in the toric variety $X_{\tilde{\Sigma}'}$, which corresponds to the cone generated by $\lc n_{i, 0} \relmid 0 \leq i \leq r \rc$.
It is defined by
\begin{align}
z^{i, 0}_0=0\quad (0 \leq i \leq r)
\end{align}
in $\scU_{\scS, 0}$.
The blow-up of $\scU_{\scS, 0}$ along (the strict transform $\widetilde{D}_{\scS, 1}$ of) $D_{\scS, 1}$ is the Zariski closure of the graph of the morphism
\begin{align}\label{eq:mor1}
\scU_{\scS, 0} \setminus \widetilde{D}_{\scS, 1} \to \bP^r
\end{align}
defined by the functions $z^{i, 0}_0$ $(0 \leq i \leq r)$.
It is covered by the Zariski closures of the graphs of the morphisms
\begin{align}\label{eq:mor2}
\scU_{\scS, 0} \setminus \lb z^{i, 0}_0=0 \rb \to \bA^r\quad (0 \leq i \leq r)
\end{align} 
obtained as the restriction of \eqref{eq:mor1}.
For the shuffle $\scS$ that we fixed above, we define $\scU_{\scS, 1}$ to be the Zariski closure of the graph of the morphism \eqref{eq:mor2} with $i=i_\scS (1)$.
Let $z^{i, 0}_1$ $\lb 0 \leq i \leq r, i \neq i_\scS (1) \rb$ be the coordinates of the target $\bA^r$ of \eqref{eq:mor2} with $i=i_\scS (1)$.
The scheme $\scU_{\scS, 1}$ is the closed subscheme of
\begin{align}
\Spec \left. \lb k \ld z^{i, j}_0 : 0 \leq i \leq r, 0 \leq j \leq k_i \rd \otimes k \ld z^{i, 0}_1 : 0 \leq i \leq r, i \neq i_\scS (1) \rd \rb \middle/ \lb z^{i,0}_0=z^{i, 0}_1 z^{i(1), 0}_0, 0 \leq i \leq r, i \neq i_\scS (1) \rb \right.,
\end{align}
which is defined by
\begin{align}
f_{0, i} \cdot \prod_{j=1}^{k_0}z^{0,j}_0 =z^{i, 0}_1 \cdot \prod_{j=1}^{k_i} z^{i, j}_0 \quad (1 \leq i \leq r)
\end{align}
when $i_\scS (1)=0$, and by
\begin{align}
f_{0, i} \cdot z_1^{0, 0} \cdot \prod_{j=1}^{k_0}z^{0,j}_0 &
=\left\{
\begin{array}{ll}
z^{i, 0}_1 \cdot \prod_{j=1}^{k_i} z^{i, j}_0 & i \neq i_\scS (1)\\
\prod_{j=1}^{k_i} z^{i, j}_0 & i =i_\scS (1)\\
\end{array}
\right.
\quad (1 \leq i \leq r)
\end{align}
when $i_\scS (1)\neq 0$.
One has the isomorphism
\begin{align}\label{eq:spec1}
\scU_{\scS, 1} \cong \Spec \left. k \ld z^{i, j}_1 : 0 \leq i \leq r, 0 \leq j \leq k_i \rd  \middle/ \lb f_{1, i} \cdot \prod_{j=j_{\scS, 0}(1)}^{k_0} z^{0,j}_1 =\prod_{j=j_{\scS, i}(1)}^{k_i} z^{i, j}_1, 1 \leq i \leq r \rb \right.
\end{align}
given by the map determined by

\begin{align}\label{eq:us1-isom}
z_1^{i, 0} \mapsto z_1^{i, 0} \lb 0 \leq i \leq r, i \neq i_\scS (1) \rb, \quad
z_{0}^{i, j} \mapsto
\left\{
\begin{array}{ll}
z_{1}^{i, j} \cdot z_{1}^{i_\scS (1), 0}& i \neq i_\scS (1), j=0, \\
z_{1}^{i, j} & \mathrm{otherwise},\\
\end{array}
\right.
\end{align}
where $f_{1, i}$ $(1 \leq i \leq r)$ is the image of $f_{0, i}$ by \eqref{eq:us1-isom}.
Under the identification of \eqref{eq:spec1}, the restriction $\pi_1 \colon \scU_{\scS, 1} \to \scU_{\scS, 0}$ of the blow-up is given by
\begin{align}\label{eq:x01}
z_{0}^{i, j} \mapsto
\left\{
\begin{array}{ll}
z_{1}^{i, j} \cdot z_{1}^{i_\scS (1), 0}& i \neq i_\scS (1), j=0, \\
z_{1}^{i, j} & \mathrm{otherwise}.\\
\end{array}
\right.
\end{align}
Let $D$ be an irreducible component of $\scX_k$ intersecting with $\scU_{\scS, 0}$, and $\lc n_{i, j_i} \relmid 0 \leq i \leq r \rc$ be the primitive generators of the cone in $\tilde{\Sigma}'$ corresponding to $D$.
If the component $D$ is contained in $\lb z^{i_\scS (1), 0}_0=0 \rb$, i.e., the index $j_{i}$ (in $n_{i, j_i}$) with $i=i_\scS (1)$ equals $0$, then its strict transform does not intersect with $\scU_{\scS, 1}$.
If $j_{i_\scS (1)} \geq 1$, then the strict transform intersects with $\scU_{\scS, 1}$, and is defined in \eqref{eq:spec1} by
\begin{align}
z^{i, j_i}_1=0\quad (0 \leq i \leq r).
\end{align}

Let $D_{\scS, 2}$ be the least (with respect to the order $\leq_{\scX}$) of all the irreducible components $D$ of $\scX_k$ such that $D >_{\scX} D_{\scS, 1}$ and the strict transforms intersect with $\scU_{\scS, 1}$.
The blow-up along the strict transform of $D_{\scS, 2}$ is the blow-up that affects our affine chart $\scU_{\scS, 1}$ next.
The irreducible component $D_{\scS, 2}$ corresponds to the cone in $\tilde{\Sigma}'$ whose primitive generators are $\lc n_{i, j_i} \relmid 0 \leq i \leq r \rc$ with 
\begin{align}
j_i=
\left\{
\begin{array}{ll}
1 & i =i_\scS (1), \\
0 & \mathrm{otherwise}.\\
\end{array}
\right.
\end{align}
In other words, the cone is generated by $\lc n_{i, j_{\scS, i}(1)} \relmid 0 \leq i \leq r \rc$.
The blow-up of $\scU_{\scS, 1}$ along the strict transform $\widetilde{D}_{\scS, 2}$ of $D_{\scS, 2}$ is the Zariski closure of the graph of the morphism
\begin{align}\label{eq:mor3}
\scU_{\scS, 1} \setminus \widetilde{D}_{\scS, 2} \to \bP^r
\end{align} 
defined by the functions $z^{i, j_{\scS, i}(1)}_1$ $(0 \leq i \leq r)$.
It is covered by the Zariski closures of the graphs of the morphisms
\begin{align}\label{eq:mor4}
\scU_{\scS, 1} \setminus \lb z^{i, j_{\scS, i}(1)}_1=0 \rb \to \bA^r\quad (0 \leq i \leq r)
\end{align} 
obtained as the restriction of \eqref{eq:mor3}.
We define $\scU_{\scS, 2}$ to the Zariski closure of the graph of the morphism \eqref{eq:mor4} with $i=i_\scS (2)$.

We continue this process inductively as we will explain below (\pref{dl:u}).
We define
\begin{align}
I(l):=\lc 0 \rc \cup \lc i \in \lc 1, \cdots, r \rc \relmid j_{\scS, i} (l) < k_i \rc
\end{align}
for $l$ $(0 \leq l \leq \bar{p}+1)$, and 
\begin{align}
I'(l):=I(l) \setminus \lc i_\scS (l+1) \rc
\end{align}
for $l$ $(0 \leq l \leq \bar{p})$.
We also set
\begin{align}
\bar{l}:=
\left\{
\begin{array}{ll}
\min \lc l \in \bZ_{>0} \relmid I(l)=\lc 0 \rc \rc & \scS \in \scrS_{C, 1}, \\
\bar{p}+1 & \scS \in \scrS_{C, 2}.
\end{array}
\right.
\end{align}
For each $i \in \lc 0, \cdots, r \rc$, we also define
\begin{align}
l_i:=
\left\{
\begin{array}{ll}
\min \lc l \in \bZ_{>0} \relmid j_{\scS, i} (l)=k_i \rc & p_i=k_i, \\
\infty & p_i \neq k_i.
\end{array}
\right.
\end{align}
In the following, as abbreviations, we also simply write
\begin{align}
i(l):=i_\scS(l), \quad j_i (l):= j_{\scS, i} (l)
\end{align}
unless we need to empathize which shuffle $\scS$ we are considering.

\begin{definition-lemma}\label{dl:u}
\begin{enumerate}
\item Suppose that we blew up $\scU_{\scS, l-1}$ along (the strict transform of) the irreducible component $D_{\scS, l}$ of $\scX_k$, and obtained the scheme $\scU_{\scS, l}$ $(1 \leq l \leq \bar{l})$ as an affine open subscheme of the blow-up.
(We did it for $l=1, 2$ above.)
\begin{enumerate}
\item One has the isomorphism
\begin{align}\label{eq:spec2}
\scU_{\scS, l} \cong \Spec \left. k \ld z^{i, j}_l : 0 \leq i \leq r, 0 \leq j \leq k_i \rd  \middle/ \lb g_{l, i}, 1 \leq i \leq r \rb \right.
\end{align}
with
\begin{align}\label{eq:gli}
g_{l, i}:=
f_{l, i} \cdot \prod_{j=j_{0} (l)}^{k_0} z^{0,j}_l \cdot \prod_{l' \geq l_i+1}^{l} z_l^{i(l'), j_{i(l')} (l'-1)} -\prod_{j=j_{i} (l)}^{k_i} z^{i, j}_l,
\end{align}
under which the restriction $\pi_{l} \colon \scU_{\scS, l} \to \scU_{\scS, l-1}$ of the blow-up along (the strict transform of) $D_{\scS, l}$ is given by
\begin{align}\label{eq:blow-up1}
z_{l-1}^{i, j} \mapsto
\left\{
\begin{array}{ll}
z_{l}^{i, j} \cdot z_l^{i(l), j_{i(l)} (l-1)}& i \in I'(l-1), j=j_{i}(l-1), \\
z_{l}^{i, j} & \mathrm{otherwise}.\\
\end{array}
\right.
\end{align}
The function $f_{l, i}$ in \eqref{eq:gli} is the image of $f_{l-1, i}$ by the map \eqref{eq:blow-up1}.
(When $k_0 <j_0(l)$ (resp. $l <l_i+1$), the first product with respect to $j$ (resp. the product with respect to $l'$) in \eqref{eq:gli} is the empty product.)

\item The strict transform of an irreducible component $D (>_{\scX} D_{\scS, l})$ of $\scX_k$ intersects with $\scU_{\scS, l}$ if and only if the primitive generators of the cone corresponding to $D$ are $\lc n_{i, j_i} \relmid 0 \leq i \leq r \rc$ such that $j_i(l) \leq j_i \leq k_i$ for all $i \in \lc 0, \cdots, r \rc$.
When this holds, the strict transform of $D$ is defined in \eqref{eq:spec2} by
\begin{align}\label{eq:transform1}
z^{i, j_i}_l=0\quad \lb i \in I(l) \rb.
\end{align}

\item Suppose further $l < \bar{l}$.
We define $D_{\scS, l+1}$ to be the least (with respect to the order $\leq_{\scX}$) of all the irreducible components $D$ of $\scX_k$  such that $D >_{\scX} D_{\scS, l}$ and the strict transforms intersect with $\scU_{\scS, l}$, i.e., the toric stratum corresponding to the cone generated by $\lc n_{i, j_{i}(l)} \relmid 0 \leq i \leq r \rc$.
We also define $\scU_{\scS, l+1}$ to be the Zariski closure of the graph of the morphism
\begin{align}\label{eq:mor5}
\scU_{\scS, l} \setminus \lb z_l^{i(l+1), j_{i(l+1)} (l)}=0 \rb \to \bA^{|I'(l)|}
\end{align} 
obtained as the restriction of the morphism 
\begin{align}
\scU_{\scS, l} \setminus \widetilde{D}_{\scS, l+1} \to \bP^{|I'(l)|}
\end{align} 
defined by the functions $z^{i, j_{i} (l)}_l$ $\lb i \in I(l) \rb$, which we think of for blowing up along the strict transform $\widetilde{D}_{\scS, l+1}$ of $D_{\scS, l+1}$.
\end{enumerate}
\item Suppose that we have performed the above process repeatedly, and obtained $\scU_{\scS, l}$ with $l=\bar{l}$.
\begin{itemize}
\item When $\scS \in \scrS_{C, 1}$, the strict transforms of the irreducible components of $\scX_k$, which intersect with $\scU_{\scS, \bar{l}}$ are all Cartier divisors.
\item When $\scS \in \scrS_{C, 2}$, there is no irreducible component of $\scX_k$ whose strict transform intersects with $\scU_{\scS, \bar{l}}$.
\end{itemize}
Therefore, the further repetitive blow-ups do not affect $\scU_{\scS, \bar{l}}$ anymore in either case.
We stop the process at this point, and set $\scrU_\scS:=\scU_{\scS, \bar{l}} \subset \scrX, z_\scS^{i, j}:=z_{\bar{l}}^{i, j}, f_{\scS, i}:=f_{\bar{l}, i}$.
When $\scS \in \scrS_{C, 1}$, we write the irreducible component of $\scX_k$ corresponding to the cone generated by $\lc n_{i, j_i(l-1)} \relmid 0 \leq i \leq r \rc$ as $D_{\scS, l}$ also for every integer $l$ such that $\bar{l} < l \leq d+1(=\bar{p}+1)$.
\end{enumerate}
\end{definition-lemma}
\begin{proof}
First, we show the claims in (1). 
We have already seen them for $l=1$.
Assuming that the claims hold until we finish the $l$-th blow-up, we show that they hold also for $l+1$.
Let $z^{i, j_{i} (l)}_{l+1}$ $\lb i \in I'(l) \rb$ be the coordinates of the target $\bA^{|I'(l)|}$ of \eqref{eq:mor5}.
The scheme $\scU_{\scS, l+1}$ is the closed subscheme of 
\begin{align}
\Spec \left. \lb k \ld z^{i, j}_l : 0 \leq i \leq r, 0 \leq j \leq k_i \rd \otimes k \ld z^{i, j_{i} (l)}_{l+1} : i \in I'(l) \rd \rb \middle/ \lb z^{i, j_{i} (l)}_l=z^{i, j_{i} (l)}_{l+1} z_l^{i(l+1), j_{i(l+1)} (l)}, i \in I'(l) \rb \right.,
\end{align}
which is defined by
\begin{align}
\left\{
\begin{array}{ll}
f_{l, i} \cdot \prod_{j=j_0(l)+1}^{k_0} z^{0,j}_l -z^{i, j_i(l)}_{l+1} \cdot \prod_{j=j_i(l)+1}^{k_i} z^{i, j}_l & l < l_i \\
f_{l, i} \cdot \prod_{j=j_{0} (l)+1}^{k_0} z^{0,j}_l \cdot z_l^{0, j_0(l)} \cdot \prod_{l' \geq l_i+1}^{l} z_l^{i(l'), j_{i(l')} (l'-1)} -\prod_{j=j_{i} (l)}^{k_i} z^{i, j}_l & l \geq l_i \\
\end{array}
\right.
\quad (1 \leq i \leq r)
\end{align}
when $i (l+1)=0$, and by
\begin{align}
\left\{
\begin{array}{ll}
f_{l, i} \cdot z^{0, j_0(l)}_{l+1} \cdot \prod_{j=j_0(l)+1}^{k_0} z^{0,j}_l - z^{i, j_i(l)}_{l+1} \cdot \prod_{j=j_i(l)+1}^{k_i} z^{i, j}_l & l < l_i, i \neq i(l+1) \\
f_{l, i} \cdot z^{0, j_0(l)}_{l+1} \cdot \prod_{j=j_0(l)+1}^{k_0} z^{0,j}_l - \prod_{j=j_i(l)+1}^{k_i} z^{i, j}_l & l < l_i, i = i(l+1) \\
f_{l, i} \cdot z^{0, j_0(l)}_{l+1} \cdot \prod_{j=j_0(l)+1}^{k_0} z^{0,j}_l \cdot \prod_{l' \geq l_i+1}^{l+1} z_l^{i(l'), j_{i(l')} (l'-1)}-\prod_{j=j_{i} (l)}^{k_i} z^{i, j}_l & l \geq l_i \\
\end{array}
\right.
\quad (1 \leq i \leq r)
\end{align}
when $i(l+1)\neq 0$.
One can check that this is isomorphic to \eqref{eq:spec2} with $l$ replaced with $l+1$ by the map determined by
\begin{align}\label{eq:ll1}
z^{i, j_{i} (l)}_{l+1} \mapsto z^{i, j_{i} (l)}_{l+1} \lb i \in I'(l) \rb, \quad
z_{l}^{i, j} \mapsto
\left\{
\begin{array}{ll}
z_{l+1}^{i, j} \cdot z_{l+1}^{i(l+1), j_{i(l+1)} (l)}& i \in I'(l), j=j_{i}(l), \\
z_{l+1}^{i, j} & \mathrm{otherwise},\\
\end{array}
\right.
\end{align}
and the morphism $\pi_{l+1} \colon \scU_{\scS, l+1} \to \scU_{\scS, l}$ is given by \eqref{eq:blow-up1} with $l$ replaced with $l+1$.
We obtained the claim (a) of (1).

We check the claim (b) of (1).
Let $D (>_{\scX} D_{\scS, l+1})$ be an irreducible component of $\scX_k$.
Suppose that the strict transform of $D$ intersects with $\scU_{\scS, l+1}$.
Then the strict transform of $D$ for the $l$-th blow-up should also intersect with $\scU_{\scS, l}$.
By the induction hypothesis, the primitive generators of the cone corresponding to $D$ are $\lc n_{i, j_i} \relmid 0 \leq i \leq r \rc$ satisfying $j_i(l) \leq j_i \leq k_i$ for all $i \in \lc 0, \cdots, r \rc$.
If the strict transform of $D$ for the $l$-th blow-up is contained in $\lb z_l^{i(l+1), j_{i(l+1)} (l)}=0 \rb$, i.e., $j_i$ with $i=i(l+1)$ equals $j_{i(l+1)}(l)$, then the strict transform of $D$ for the $(l+1)$-th blow-up does not intersect with $\scU_{\scS, l+1}$.
If not, i.e., $j_i$ with $i=i(l+1)$ is greater than or equal to $j_{i(l+1)}(l)+1$, then it intersects with $\scU_{\scS, l+1}$.
Since
\begin{align}
j_i(l+1)
=\left\{
\begin{array}{ll}
j_i(l)+1 & i=i(l+1), \\
j_i(l) & i \neq i(l+1),
\end{array}
\right.
\end{align}
a necessary and sufficient condition for the strict transform of $D$ to intersect with $\scU_{\scS, l+1}$ is to have $j_i \geq j_i(l+1)$ for all $i \in \lc 0, \cdots, r \rc$.

Suppose that the strict transform of $D (>_{\scX} D_{\scS, l+1})$ intersects with $\scU_{\scS, l+1}$.
The exceptional divisor of the $(l+1)$-th blow-up is the Cartier divisor defined by $z_{l+1}^{i(l+1), j_{i(l+1)}(l)}$.
From this, \eqref{eq:blow-up1} with $l$ replaced with $l+1$, and the induction hypothesis, one can see that the strict transform of $D$ is defined by
\begin{align}\label{eq:transform2}
z^{i, j_i}_{l+1}=0\quad \lb i \in I(l) \rb.
\end{align}
For an element $i \nin I(l+1)$, the $i$-th equation defining $\scU_{\scS, l+1}$ is
\begin{align}
f_{l+1, i} \cdot \prod_{j=j_0(l)}^{k_0} z^{0,j}_{l+1} \cdot \prod_{l' \geq l_i+1}^{l+1} z_{l+1}^{i(l'), j_{i(l')} (l'-1)}=z^{i, k_i}_{l+1}.
\end{align}
Hence, $z^{0,j_0}_{l+1}=0$ implies $z^{i, k_i}_{l+1}=0$.
Therefore, one can remove the equation for $i$ in \eqref{eq:transform2}, and the claim of \eqref{eq:transform1} for $l+1$ also holds.
We obtained the claim (b) of (1).

Lastly, we check the claim in (2).
When $\scS \in \scrS_{C, 1}$, it follows from \eqref{eq:transform1} with $l=\bar{l}$ and $I(\bar{l})=\lc 0 \rc$.
Suppose $\scS \in \scrS_{C, 2}$.
Since $j_0(\bar{l}=\bar{p}+1)=k_0+1$, there is no irreducible component of $\scX_k$ such that the primitive generators of the corresponding cone are $\lc n_{i, j_i} \relmid 0 \leq i \leq r \rc$ satisfying $j_0 \geq j_0(\bar{l})$.
The claims follows from this and (b).
\end{proof}

\begin{lemma}\label{lm:us}
One has
\begin{align}\label{eq:Us}
\scrU_\scS \cong \Spec \left. k \ld z^{i, j}_\scS : 0 \leq i \leq r, 0 \leq j \leq k_i \rd  \middle/ \lb g_{\scS, i}, 1 \leq i \leq r \rb \right.,
\end{align}
where $g_{\scS, i}$ $\lb 1 \leq i \leq r \rb$ is defined by
\begin{align}
g_{\scS, i}:=
f_{\scS, i} \cdot \prod_{l \geq l_i+1}^{\bar{p}+1} z_\scS^{i(l), j_{i(l)}(l-1)}-\prod_{j=p_i}^{k_i} z^{i, j}_\scS.
\end{align}
(When $\bar{p}+1 < l_i+1$, the above product with respect to $l$ is the empty product.)
\end{lemma}
\begin{proof}
In either case $\scS \in \scrS_{C, 1}$ or $\scS \in \scrS_{C, 2}$, we have $j_i ( \bar{l})=p_i$ and
\begin{align}
\prod_{j=j_{0} (\bar{l})}^{k_0} z^{0,j}_\scS \cdot \prod_{l' \geq l_i+1}^{\bar{l}} z_\scS^{i(l'), j_{i(l')} (l'-1)} 
=
\prod_{l \geq l_i+1}^{\bar{p}+1} z_\scS^{i(l), j_{i(l)}(l-1)}.
\end{align}
The claim follows from these and \pref{dl:u}(1)(a).
\end{proof}

\begin{lemma}\label{lm:cover}
For any $l \in \lc 1, \cdots, \bar{l} \rc$, the family of open subschemes $\lc \scU_{\scS, l} \relmid \scS \in \scrS_{C, 1} \cup \scrS_{C, 2} \rc$ forms a covering of the scheme obtained after $\scX \cap U_C$ is blown up along $l$ components.
\end{lemma}
\begin{proof}
We prove this by induction on $l$.
It is clear that the claim holds for $l=1$.
We suppose that the claim holds for $l-1$, and show the claim for $l (\leq \bar{l})$. 
From \pref{dl:u}, we can see that it suffices to show that for any $\scS \in \scrS_{C, 1} \cup \scrS_{C, 2}$ and $i \in I(l-1)$, there exists a shuffle $\scS' \in \scrS_{C, 1} \cup \scrS_{C, 2}$ satisfying the following:
\begin{condition}\label{cd:shuffle}
The following hold:
\begin{itemize}
\item $i_\scS (l')=i_{\scS'} (l')$ for all $l' \leq l-1$, and 
\item $i_{\scS'}(l)=i$.
\end{itemize}
\end{condition}
If $i \in \lc 1, \cdots, r \rc$, then $j_{\scS, i}(l-1)<k_i$, and it is obvious that there exists $\scS' \in \scrS_{C, 1}$ satisfying \pref{cd:shuffle}.
We consider the case $i=0$.
If $j_{\scS, 0}(l-1)<k_0$, then it is again obvious that there exists such $\scS' \in \scrS_{C, 1}$.
Suppose $j_{\scS, 0} (l-1)=k_0$.
If $j_{\scS, i}(l-1)=k_i$ also for all $i \in \lc 1, \cdots, r \rc$, then $l-1=\sum_{i=0}^r k_i \geq \bar{l}$.
This contradicts $l \leq \bar{l}$.
Hence, there exists some $i_0 \in \lc 1, \cdots, r \rc$ such that $j_{\scS,  i_0}(l-1)<k_{i_0}$.
From how we set $i_\scS(\bar{p}+1)$ in \eqref{eq:ip1} and the definition of $\scrS_{C, 2}$, we can see that there exists $\scS' \in \scrS_{C, 2}$ satisfying \pref{cd:shuffle}.
Thus we can conclude the claim.
\end{proof}

By \pref{lm:cover}, we can conclude the following:

\begin{corollary}\label{cr:cover}
The collection of open subschemes $\scrU_{\scS} \subset \scrX$ ($\scS \in \scrS_{C, 1} \cup \scrS_{C, 2}$ and $C \in \tilde{\Sigma}'$ is a relevant cone of maximal dimension) covers the special fiber $\scrX_k$ of $\scrX$.
\end{corollary}

\begin{lemma}
Let $l \in \lc 1, \cdots, \bar{l} \rc$ and $\scS \in \scrS_{C, 1}$.
For $(i, j)$ such that $0 \leq i \leq r$ and $0 \leq j \leq k_i$, we have the following equalities as elements of $K \lb \scU_{\scS, l} \rb \cong K \lb \scU_{\scS, 0} \rb$:
\begin{enumerate}
\item For $(i, j) \neq (1, k_1), \cdots, (r, k_r), (i(1), 0)$, we have
\begin{align}\label{eq:xijl1}
z_l^{i, j}
=\left\{
\begin{array}{ll}
z_0^{i, j} & j > j_i(l), \\
\prod_{j'=0}^j z_0^{i, j'} \cdot \lb \prod_{j'=0}^{j_{i(l)}(l)-1} z_0^{i(l), j'} \rb^{-1} & j = j_i(l), \\
\prod_{j'=0}^j z_0^{i, j'} \cdot \lb \prod_{j'=0}^{j_{i(l_{i, j})}(l_{i, j})-1} z_0^{i(l_{i, j}), j'} \rb^{-1} & j < j_i(l),
\end{array}
\right.
\end{align}
where $l_{i, j}:=\max \lc l' \relmid 1 \leq l' \leq d, j_i(l')=j \rc$.
\item For $(i, j) = (1, k_1), \cdots, (r, k_r), (i(1), 0)$, we have
\begin{align}
z_l^{i, j}=z_0^{i, j}.
\end{align}
\end{enumerate}
\end{lemma}
\begin{proof}
We can see from \eqref{eq:blow-up1} that we have
\begin{align}\label{eq:l-l+1}
z_{l+1}^{i, j}
=\left\{
\begin{array}{ll}
z_l^{i, j} \cdot \lb z_l^{i(l+1), j_{i(l+1)}(l)} \rb^{-1} & i \in I'(l), j=j_i(l) \\
z_l^{i, j} & \mathrm{otherwise}
\end{array}
\right.
\end{align}
as elements of $K \lb \scU_{\scS, l+1} \rb \cong K \lb \scU_{\scS, l} \rb$.
Using this, we prove the claim by induction on $l$.

First, we check (2).
For $(i, k_i)$ $(1 \leq i \leq r)$, if we have $k_i=j_i(l)$, then $i \nin I'(l)$.
Hence, by \eqref{eq:l-l+1}, we have $z_{l+1}^{i, j}=z_l^{i, j}$ for all $l \in \lc 1, \cdots, \bar{l} \rc$.
Thus we obtain (2) for $(i, k_i)$.
For $(i, j) =(i(1), 0)$, we have $j_{i(1)}(l) \geq 1>0=j$ for all $l \in \lc 1, \cdots, \bar{l} \rc$.
Hence, by \eqref{eq:l-l+1}, we have $z_{l+1}^{i, j}=z_l^{i, j}$ for all $l \in \lc 1, \cdots, \bar{l} \rc$.
Thus we obtain (2) also for $(i(1), 0)$.

Next, we will prove (1).
Let $(i, j)$ be an element which is none of $(1, k_1), \cdots, (r, k_r)$, and $(i(1), 0)$.
We consider the case $l=1$.
If $j > j_i(1)$, then we have $z_{1}^{i, j}=z_{0}^{i, j}$ by \eqref{eq:l-l+1}, which coincides with \eqref{eq:xijl1} for $j > j_i(l)$.
In the case of $j < j_i(1)$, we must have $j_i(1)=1$ and $(i, j)=(i(1),0)$.
Therefore, we do not consider this case.
Suppose $j=j_i(1)$.
By \eqref{eq:l-l+1}, we have
\begin{align}
z_{1}^{i, j_i(1)}
=\left\{
\begin{array}{ll}
z_0^{i, j_i(1)} \cdot \lb z_0^{i(1), j_{i(1)}(0)} \rb^{-1}= z_0^{i, 0} \cdot \lb z_0^{i(1), 0} \rb^{-1} & i \neq i(1), \\
z_0^{i, j_i(1)}=z_0^{i, 1} & i = i(1).
\end{array}
\right.
\end{align}
This coincides with \eqref{eq:xijl1} for $j = j_i(l)$.
Thus we obtain (1) for $l=1$.

We suppose that (1) holds for $l$, and prove that it holds also for $l+1$.
First, we consider the case where $j > j_i(l+1)$.
We have
\begin{align}
z_{l+1}^{i, j}=z_{l}^{i, j}=z_{0}^{i, j}
\end{align}
by $j > j_i(l+1) \geq j_i(l)$, \eqref{eq:l-l+1}, and the induction hypothesis.
This coincides with \eqref{eq:xijl1} with $l$ replaced with $l+1$.

Next, we consider the case where $j = j_i(l+1)$.
If $i(l+1)=i$, then $i \nin I'(l)$ and $j >j_i(l)$.
Hence, we have 
\begin{align}\label{eq:xijl2}
z_{l+1}^{i, j}=z_{l}^{i, j}=z_{0}^{i, j}
\end{align}
by \eqref{eq:l-l+1} and the induction hypothesis.
On the other hand, \eqref{eq:xijl1} for $j=j_i(l)$ with $l$ replaced with $l+1$ is 
\begin{align}
z_{l+1}^{i, j}
=\prod_{j'=0}^j z_0^{i, j'} \cdot \lb \prod_{j'=0}^{j_{i(l+1)}(l+1)-1} z_0^{i(l+1), j'} \rb^{-1}
=\prod_{j'=0}^j z_0^{i, j'} \cdot \lb \prod_{j'=0}^{j-1} z_0^{i, j'} \rb^{-1}
=z_{0}^{i, j}.
\end{align}
This coincides with \eqref{eq:xijl2}.
If $i(l+1) \neq i$, then $j = j_i(l+1)=j_i(l)$ and $i \in I'(l)$.
Hence, by \eqref{eq:l-l+1} and the induction hypothesis, we have 
\begin{align}
z_{l+1}^{i, j}&=z_l^{i, j} \cdot \lb z_l^{i(l+1), j_{i(l+1)}(l)} \rb^{-1}\\
&=\prod_{j'=0}^j z_0^{i, j'} \cdot \lb \prod_{j'=0}^{j_{i(l)}(l)-1} z_0^{i(l), j'} \rb^{-1} \cdot \lb \prod_{j'=0}^{j_{i(l+1)}(l)} z_0^{i(l+1), j'} \rb^{-1} \cdot \prod_{j'=0}^{j_{i(l)}(l)-1} z_0^{i(l), j'} \\
&=\prod_{j'=0}^j z_0^{i, j'} \cdot \lb \prod_{j'=0}^{j_{i(l+1)}(l)} z_0^{i(l+1), j'} \rb^{-1}.
\end{align}
This coincides with \eqref{eq:xijl1} for $j=j_i(l)$ with $l$ replaced with $l+1$, since $j_{i(l+1)}(l)=j_{i(l+1)}(l+1)-1$.

Lastly, we consider the case where $j < j_i(l+1)$.
In this case, we have $j \leq j_i(l)$.
If $j=j_i(l)$, then $i(l+1)=i$ and $i \nin I'(l)$.
Hence, we have 
\begin{align}
z_{l+1}^{i, j}=z_{l}^{i, j}
=\prod_{j'=0}^j z_0^{i, j'} \cdot \lb \prod_{j'=0}^{j_{i(l)}(l)-1} z_0^{i(l), j'} \rb^{-1}
\end{align}
by \eqref{eq:l-l+1} and the induction hypothesis.
This coincides with \eqref{eq:xijl1} for $j < j_i(l)$ with $l$ replaced with $l+1$, since $l_{i,j}
=\max \lc l' \relmid j_{i(l+1)}(l')=j_{i(l+1)}(l) \rc=l$.
If $j<j_i(l)$, then we have 
\begin{align}
z_{l+1}^{i, j}=z_{l}^{i, j}
=\prod_{j'=0}^j z_0^{i, j'} \cdot \lb \prod_{j'=0}^{j_{i(l_{i, j})}(l_{i, j})-1} z_0^{i(l_{i, j}), j'} \rb^{-1} 
\end{align}
by \eqref{eq:l-l+1} and the induction hypothesis.
This coincides with \eqref{eq:xijl1} for $j < j_i(l)$ with $l$ replaced with $l+1$.
Thus we obtain (1).
\end{proof}

The claim of the above lemma for $l=\bar{l}$ is as follows:

\begin{lemma}\label{lm:xsx0}
Let $\scS \in \scrS_{C, 1}$.
For $(i, j)$ such that $0 \leq i \leq r$ and $0 \leq j \leq k_i$, we have the following equalities as elements of $K \lb \scrU_{\scS} \rb \cong K \lb \scU_{\scS, 0} \rb$:
\begin{enumerate}
\item For $(i, j) \neq (1, k_1), \cdots, (r, k_r), (i(1), 0)$, we have
\begin{align}\label{eq:xijl3}
z_{\scS}^{i, j}
=\prod_{j'=0}^j z_0^{i, j'} \cdot \lb \prod_{j'=0}^{j_{i(l_{i, j})}(l_{i, j})-1} z_0^{i(l_{i, j}), j'} \rb^{-1},
\end{align}
where $l_{i, j}:=\max \lc l' \relmid 1 \leq l' \leq d, j_i(l')=j \rc$.
\item For $(i, j) = (1, k_1), \cdots, (r, k_r), (i(1), 0)$, we have
\begin{align}
z_{\scS}^{i, j}=z_0^{i, j}.
\end{align}
\end{enumerate}
\end{lemma}
\begin{proof}
(2) is obvious.
We will show (1).
Let $(i, j)$ be an element which is none of $(1, k_1), \cdots, (r, k_r)$, and $(i(1), 0)$.
It is obvious that (1) holds if $j < j_i(\bar{l})$, since \eqref{eq:xijl1} for the case $j < j_i(l)$ with $l=\bar{l}$ coincides with \eqref{eq:xijl3}.
When $i \neq 0$, we have $j_i(\bar{l})=k_i>j$.
Therefore, we show (1) supposing $i = 0$ and $j \geq j_i(\bar{l})$ in the following.
If $j > j_i(\bar{l})$, then $i(l_{i, j})=0=i$ and $j_{i(l_{i, j})}(l_{i, j})=j$.
Hence, \eqref{eq:xijl3} is 
\begin{align}
z_{\scS}^{i, j}
=\prod_{j'=0}^j z_0^{i, j'} \cdot \lb \prod_{j'=0}^{j-1} z_0^{i, j'} \rb^{-1}=z_{0}^{i, j}.
\end{align}
Since this coincides with \eqref{eq:xijl1} for the case $j > j_i(l)$ with $l=\bar{l}$, (1) holds in this case.
Lastly, if $j = j_i(\bar{l})$, then $l_{i, j}=\bar{l}$.
Hence, \eqref{eq:xijl1} for the case $j = j_i(l)$ with $l=\bar{l}$ coincides with \eqref{eq:xijl3}.
Thus we obtain (1).
\end{proof}

In particular, we obtain formulas for $z^{i(l), j_{i(l)}(l-1)}_\scS$ $\lb l \in \lc 1, \cdots, d+1 \rc  \rb$.

\begin{corollary}\label{cr:zsl}
Let $\scS \in \scrS_{C, 1}$. 
For $l \in \lc 1, \cdots, d+1 \rc$, we have the equality
\begin{align}
z^{i(l), j_{i(l)}(l-1)}_\scS=z^{m_l}
\end{align}
as elements of $K \lb \scrU_{\scS} \rb \cong K \lb \scU_{\scS, 0} \rb$, where $m_l \in M \oplus \bZ$ is defined by
\begin{align}\label{eq:ml}
m_l:=
\left\{
\begin{array}{ll}
m_{i(1), 0} & l=1, \\
\sum_{j'=0}^{j_{i(l)}(l-1)} m_{i(l), j'}- \sum_{j'=0}^{j_{i(l-1)}(l-1)-1} m_{i(l-1), j'} & l \geq 2.
\end{array}
\right.
\end{align}
\end{corollary}
\begin{proof}
When $l=1$, one has $j_{i(l)}(l-1)=0$, and the claim follows from \pref{lm:xsx0}(2).
Suppose $l \geq 2$.
Then $\lb i(l), j_{i(l)} (l-1) \rb \neq (1, k_1), \cdots, (r, k_r), (i(1), 0)$.
Indeed, if $i(l)=i(1)$, then $j_{i(l)} (l-1) \geq 1$, and if $i(l) \neq 0$, then $j_{i(l)} (l-1)= j_{i(l)} (l)-1\leq k_{i(l)}-1$.
Furthermore, the number $l_{i, j}$ of \pref{lm:xsx0}(1) with $i=i(l), j=j_{i(l)} (l-1)$ is $l-1$.
Therefore, the claim follows from \pref{lm:xsx0}(1) also in this case.
\end{proof}

\begin{lemma}\label{lm:pi}
Let $l \in \lc 1, \cdots, \bar{l} \rc$.
The morphism $\scU_{\scS, l} \to \scU_{\scS, 0}$ obtained by composing the morphisms $\pi_{l'} \colon \scU_{\scS, l'} \to \scU_{\scS, l'-1}$ $(1 \leq l' \leq l)$ is given by
\begin{align}\label{eq:x0l}
z_0^{i, j} \mapsto 
z_l^{i, j} \cdot \prod_{l' \in L^{i, j}_l} z_{l}^{i(l'), j_{i(l')}(l'-1)} \quad (0 \leq i \leq r, 0 \leq j \leq k_i),
\end{align}
where
\begin{align}
L^{i, j}_l:=
\left\{
\begin{array}{ll}
\lc l' \in \bZ_{>0} \relmid l' \leq l, i \neq i(l'), j_i(l'-1)=j \rc & i=0, \\
\lc l' \in \bZ_{>0} \relmid l' \leq \min \lc l, l_i \rc, i \neq i(l'), j_i(l'-1)=j \rc & i \in \lc 1, \cdots, r \rc.\\
\end{array}
\right.
\end{align}
\end{lemma}
\begin{proof}
We prove the claim by induction on $l$.
When $l=1$, we have
\begin{align}
L_1^{i, j}=
\left\{
\begin{array}{ll}
\lc 1 \rc & i \neq i (1), j_i(0)=0=j, \\
\emptyset & \mathrm{otherwise},\\
\end{array}
\right.
\end{align}
and \eqref{eq:x0l} indeed coincides with \eqref{eq:x01}.
When we suppose that the claim holds for $l-1$ $(\leq \bar{l}-1)$, one can easily show the claim for $l$ by \eqref{eq:blow-up1}.
Notice that since we have
\begin{align}
j_{i(l')}(l'-1)=j_{i(l')}(l')-1<j_{i(l')}(l-1)
\end{align}
for $l' \leq l-1$, the monomial $z_{l-1}^{i(l'), j_{i(l')}(l'-1)}$ $\lb l' \in L_{l-1}^{i, j} \rb$ is mapped by \eqref{eq:blow-up1} to $z_{l}^{i(l'), j_{i(l')}(l'-1)}$.
\end{proof}

\begin{lemma}\label{lm:usl}
For $l \in \lc 1, \cdots, \bar{l} \rc$, one has the isomorphism 
\begin{align}\label{eq:isom-ul0}
\scU_{\scS, l} \setminus \lb \prod_{l' = 1}^{l} z_{l}^{i(l'), j_{i(l')}(l'-1)}=0 \rb \cong 
\scU_{\scS, 0} \setminus \lb \prod_{l' = 1}^{l} z_{0}^{i(l'), j_{i(l')}(l'-1)}=0 \rb
\end{align}
given by the restriction of the composition of the morphisms $\pi_{l'} \colon \scU_{\scS, l'} \to \scU_{\scS, l'-1}$ $(1 \leq l' \leq l)$.
\end{lemma}
\begin{proof}
For any $l \in \lc 1, \cdots, \bar{l} \rc$, we have the isomorphism
\begin{align}\label{eq:isom-ul-1}
\scU_{\scS, l} \setminus \lb \prod_{l' = 1}^{l} z_{l}^{i(l'), j_{i(l')}(l'-1)}=0 \rb \cong 
\scU_{\scS, l-1} \setminus \lb \prod_{l' = 1}^{l} z_{l-1}^{i(l'), j_{i(l')}(l'-1)}=0 \rb
\end{align}
given by the restriction of the morphism $\pi_{l'} \colon \scU_{\scS, l'} \to \scU_{\scS, l'-1}$.
This is because the map \eqref{eq:blow-up1} sends $z_{l-1}^{i(l'), j_{i(l')}(l'-1)}$ to $z_{l}^{i(l'), j_{i(l')}(l'-1)}$ $(1 \leq l' \leq l)$, and the inverse morphism of \eqref{eq:isom-ul-1} is given by
\begin{align}
z_{l}^{i, j} \mapsto
\left\{
\begin{array}{ll}
\left. z_{l-1}^{i, j} \middle/ z_{l-1}^{i(l), j_{i(l)} (l-1)} \right. & i \in I'(l-1), j=j_{i}(l-1), \\
z_{l-1}^{i, j} & \mathrm{otherwise}.\\
\end{array}
\right.
\end{align}
We prove the lemma by using \eqref{eq:isom-ul-1} and induction on $l$.
When $l=1$, \eqref{eq:isom-ul0} is exactly \eqref{eq:isom-ul-1} with $l=1$.
Suppose that the lemma holds for $l-1$.
Then by the induction hypothesis and \pref{lm:pi} for $l-1$, we can get
\begin{align}
\scU_{\scS, l-1} \setminus \lb \prod_{l' = 1}^{l} z_{l-1}^{i(l'), j_{i(l')}(l'-1)}=0 \rb \cong 
\scU_{\scS, 0} \setminus \lb \prod_{l' = 1}^{l} z_{0}^{i(l'), j_{i(l')}(l'-1)}=0 \rb.
\end{align}
By combining this and \eqref{eq:isom-ul-1}, we obtain the lemma for $l$.
\end{proof}

\begin{corollary}\label{cr:usk}
Let $C':=C \cap \lb N_\bR \times \lc 0 \rc \rb$.
One has
\begin{align}
\scrU_\scS \times_R K 
\cong 
\lb X \cap X_{C'} \rb \setminus \lb \prod_{l = 1}^{\bar{p}+1} z_{0}^{i(l), j_{i(l)}(l-1)}=0 \rb,
\end{align}
where $X_{C'}$ is the affine toric variety over $K$ associated with the cone $C' \subset N_\bR$.
\end{corollary}
\begin{proof}
By \pref{lm:usl} and \pref{lm:pi} for $l=\bar{l}$, we obtain the isomorphism
\begin{align}
\scrU_{\scS} \setminus \lb \prod_{l = 1}^{\bar{p}+1} z_{\scS}^{i(l), j_{i(l)}(l-1)}=0 \rb \cong 
\scU_{\scS, 0} \setminus \lb \prod_{l = 1}^{\bar{p}+1} z_{0}^{i(l), j_{i(l)}(l-1)}=0 \rb.
\end{align}
On the other hand, the image of the function $t=\prod_{j=0}^{k_0}z^{0,j}_0$ by the map \eqref{eq:x0l} with $l=\bar{l}$ is
\begin{align}\label{eq:image-t}
\prod_{j=0}^{k_0} \lb z_{\scS}^{0, j} \prod_{l' \in L^{0, j}_{\bar{l}}} z_{\scS}^{i(l'), j_{i(l')}(l'-1)} \rb
=\prod_{j=0}^{k_0} z_{\scS}^{0, j} \cdot \prod_{1 \leq l \leq \bar{l}, i(l) \neq 0} z_{\scS}^{i(l), j_{i(l)}(l-1)}
=\prod_{l = 1}^{\bar{p}+1} z_{\scS}^{i(l), j_{i(l)}(l-1)}.
\end{align}
From these, we can get
\begin{align}
\scrU_\scS \times_R K \cong \lb \scU_{\scS, 0} \times_R K \rb \setminus \lb \prod_{l = 1}^{\bar{p}+1} z_{0}^{i(l), j_{i(l)}(l-1)}=0 \rb
=\lb X \cap X_{C'} \rb \setminus \lb \prod_{l = 1}^{\bar{p}+1} z_{0}^{i(l), j_{i(l)}(l-1)}=0 \rb.
\end{align}
\end{proof}

For $\scS \in \scrS_{C, 1} \cup \scrS_{C, 2}$, we define $\scX_{\scS, l}$ $(1 \leq l \leq \bar{p}+1)$ to be the variety obtained just after the blow-up along (the strict transform of) $D_{\scS, l}$ while we perform the repetitive blow-ups for $\scX$ explained in \pref{sc:blow-up}.
We also define $D_{\scS, i}^l$ $(1 \leq i \leq \bar{p}+1)$ to be the divisor on $\scX_{\scS, l}$ obtained by continuing the following process until we blow up along (the strict transform of) $D_{\scS, l}$:
For the divisor $D_{\scS, i}$, we
\begin{itemize}
\item take the strict transform when we blow up along (the strict transform of) an irreducible component of $\scX_k$ which is not $D_{\scS, i}$, and 
\item take the exceptional divisor when we blow up along (the strict transform of) $D_{\scS, i}$.
\end{itemize}
We also write the divisor on $\scrX$ obtained by continuing this process for $D_{\scS, i}$ until we finish all the blow-ups as $\overline{D}_{\scS, i} \subset \scrX$.
 
\begin{lemma}\label{lm:comp}
The divisor $\overline{D}_{\scS, l} \subset \scrX$ $(1 \leq l \leq \bar{p}+1)$ is a Cartier divisor defined by $z_\scS^{i(l), j_{i(l)}(l-1)}$ on $\scrU_\scS \subset \scrX$.
\end{lemma}
\begin{proof}
First, we show the claim for $\overline{D}_{\scS, l}$ such that $1 \leq l \leq \bar{l}$.
The divisor $D_{\scS, l}^l$ is defined by $z_{l}^{i(l), j_{i(l)}(l-1)}$ on $\scU_{\scS, l}$.
Since we have 
\begin{align}
j_{i(l)}(l-1)=j_{i(l)}(l)-1<j_{i(l)}(l')
\end{align}
for $l' \geq l$, we can see that $z_{l}^{i(l), j_{i(l)}(l-1)}$ is mapped to $z_\scS^{i(l), j_{i(l)}(l-1)}$ by the morphism $\scrU_\scS \to \scU_{\scS, l}$ obtained by composing the morphisms of \eqref{eq:blow-up1}.
Thus we can conclude the claim.

Next, we consider $\overline{D}_{\scS, l}$ such that $\bar{l} < l \leq d+1(=\bar{p}+1)$, which may exist when $\scS \in \scrS_{C, 1}$.
By \pref{dl:u}(b) with $l=\bar{l}$, we can see that it is defined by the monomial $z_\scS^{0, j_0(l-1)}$ on $\scrU_\scS$.
Since $i(l)=0$ for $l > \bar{l}$, the claim follows.
\end{proof}

\subsection{The minimal snc-model}\label{sc:minimal}

The aim of this subsection is to prove the following proposition.
Recall that an snc-model $\scrX$ is called \emph{minimal} if the logarithmic relative canonical divisor $K_{\scrX/R}+\scrX_{k, \mathrm{red}}$ is semi-ample.

\begin{proposition}\label{pr:snc}
The model $\scrX \to \Spec R$ constructed in \pref{sc:blow-up} is a minimal snc-model of $X \to \Spec K$ with reduced special fiber.
\end{proposition}

In the following, we will use $I_\scX, \scX_l, D_{i}^l, \overline{D}_i$, which we defined in \pref{sc:blow-up}.
See the final paragraph of \pref{sc:blow-up} for the definitions of these.

\begin{lemma}\label{lm:special}
One has
\begin{align}\label{eq:central}
\scrX_k = \sum_{i \in I_\scX} \overline{D}_i,
\end{align}
where $\scrX_k$ denotes the special fiber of $\scrX$.
\end{lemma}
\begin{proof}
As we saw in \eqref{eq:image-t}, the image of the function $t=\prod_{j=0}^{k_0}z^{0,j}_0$ by the map \eqref{eq:x0l} with $l=\bar{l}$ is $\prod_{l = 1}^{\bar{p}+1} z_{\scS}^{i(l), j_{i(l)}(l-1)}$.
By this, \pref{lm:comp}, and \pref{cr:cover}, we obtain \eqref{eq:central}.
\end{proof}

\begin{lemma}\label{lm:intersect0}
Let $l \in I_\scX \cup \lc 0 \rc$.
For any integer $i \in I_\scX \cap \bZ_{>l}$, the divisor $D_i^l \subset \scX_l$ is a smooth toric variety. 
Furthermore, for any subset $J \subset I_\scX \cap \bZ_{>l}$, the intersection $\bigcap_{j \in J} D_j^l \subset \scX_l$ is a toric stratum of every toric variety $D_j^l$ $(j \in J)$, if the intersection is not empty.
\end{lemma}
\begin{proof}
We show it by induction on $l$.
The claim for $l=0$ is obvious from \pref{cd:add}.
We suppose that the claim holds for $l-1$, and show the claim for $l$.

First, we show the former claim for $l$.
By the induction hypothesis, for any integer $i \in I_\scX \cap \bZ_{>l}$, the divisor $D_i^{l-1} \subset \scX_{l-1}$ is a smooth toric variety.
For the $l$-th blow-up, the strict transform $D_i^l$ of $D_i^{l-1}$ is isomorphic to the blow-up of $D_i^{l-1}$ along the intersection $D_i^{l-1} \cap D_l^{l-1} \subset \scX_{l-1}$.
If the intersection $D_i^{l-1} \cap D_l^{l-1}$ is not empty, then it is a toric stratum of the smooth toric variety $D_i^{l-1}$ by the induction hypothesis.
Thus the strict transform $D_i^l$ is also a smooth toric variety.
If the intersection $D_i^{l-1} \cap D_l^{l-1}$ is empty, then the the strict transform $D_i^l \subset \scX_{l}$ is isomorphic to $D_i^{l-1} \subset \scX_{l-1}$, and is also a smooth toric variety.
Thus we conclude the former claim for $l$.

Next, we show the latter claim for $l$.
Suppose that the intersection $\bigcap_{j \in J} D_j^l \subset \scX_l$ is not empty for a subset $J \subset I_\scX \cap \bZ_{>l}$.
One can check that it coincides with the strict transform of the intersection $\bigcap_{j \in J} D_j^{l-1} \subset \scX_{l-1}$.
The intersection $\bigcap_{j \in J} D_j^{l-1}$ is also not empty, and is a toric stratum of every toric variety $D_j^{l-1}$ $(j \in J)$ by the induction hypothesis.
The strict transform of $\bigcap_{j \in J} D_j^{l-1}$ also coincides with the strict transform of $\bigcap_{j \in J} D_j^{l-1}$ with respect to the blow-up of $D_j^{l-1}$ along $D_l^{l-1} \cap D_j^{l-1}$ for any $j \in J$.
The intersection $D_l^{l-1} \cap D_j^{l-1}$ is a toric stratum of $D_j^{l-1}$ again by the induction hypothesis.
Since for a toric blow-up of a smooth toric variety, the strict transform of a toric stratum is again a toric stratum, the strict transform of $\bigcap_{j \in J} D_j^{l-1}$ is a toric stratum of $D_j^l \subset \scX_l$.
Thus we conclude the claim.
\end{proof}

\begin{lemma}\label{lm:intersect00}
Let $l \in I_\scX$ and $i \in I_\scX \cap \bZ_{>l}$.
\begin{enumerate}
\item For any integer $j \in \lc 1, \cdots, l \rc$, the intersection $D_j^l \cap D_i^l \subset \scX_l$ is a toric divisor of the toric variety $D_i^l$, if the intersection is not empty.
\item Such toric divisors $D_j^l \cap D_i^l \neq \emptyset$ $(1 \leq j \leq l)$ of the toric variety $D_i^l$ are all distinct.
\end{enumerate}
\end{lemma}
\begin{proof}
First, we show the claim (1) for the case $j=l$.
Suppose that $D_l^l \cap D_i^l \subset \scX_{l}$ is not empty.
The intersection $D_l^l \cap D_i^l \subset \scX_l$ is the exceptional divisor of the blow-up of $D_i^{l-1} \subset \scX_{l-1}$ along $D_l^{l-1} \cap D_i^{l-1} \subset \scX_{l-1}$.
Since $D_l^{l-1} \cap D_i^{l-1} \subset \scX_{l-1}$ is a toric stratum of the smooth toric variety $D_i^{l-1} \subset \scX_{l-1}$ by \pref{lm:intersect0}, the intersection $D_l^{l} \cap D_i^{l} \subset \scX_l$ is a toric divisor of the toric variety $D_i^l \subset \scX_{l}$.
Thus we conclude the claim (1) for the case $j=l$.
In particular, the claim (1) holds when $l=1$.

Next, we suppose that the claim (1) holds for $l-1$, and show it for $l$.
Let $j \in \lc 1, \cdots, l-1 \rc$, and suppose that the intersection $D_j^l \cap D_i^l \subset \scX_{l}$ is not empty.
One can check that the intersection $D_j^l \cap D_i^l$ coincides with the strict transform of $D_j^{l-1} \cap D_i^{l-1} \neq \emptyset$.
It also coincides with the strict transform of $D_j^{l-1} \cap D_i^{l-1} \subset \scX_{l-1}$ with respect to the blow-up of $D_i^{l-1}$ along $D_l^{l-1} \cap D_i^{l-1} \subset \scX_{l-1}$.
The intersection $D_j^{l-1} \cap D_i^{l-1} \subset \scX_{l-1}$ is a toric divisor of $D_i^{l-1} \subset \scX_{l-1}$ by the induction hypothesis, and the intersection $D_l^{l-1} \cap D_i^{l-1} \subset \scX_{l-1}$ is a toric stratum of the smooth toric variety $D_i^{l-1}$ by \pref{lm:intersect0}.
Since for a toric blow-up of a smooth toric variety, the strict transform of a toric divisor is again a toric divisor, the intersection $D_j^l \cap D_i^l \subset \scX_{l}$ is also a toric divisor of the toric variety $D_i^l$.
Thus we conclude the claim (1) for $l$ and $j \in \lc 1, \cdots, l-1 \rc$.
Since we have already proved the claim (1) for the case $j=l$ in the previous paragraph, we conclude the claim (1).

We show the claim (2).
For any $l' \in \lc 0, \cdots, l \rc$, let $\Sigma_i^{l'}(1)$ denote the set of $1$-dimensional cones in the fan corresponding to the smooth toric variety $D_i^{l'}$.
Then $\Sigma_i^{l_1}(1) \subset \Sigma_i^{l_2}(1)$ for $l_1 \leq l_2$.
For any $l' \in \lc 1, \cdots, l \rc$, the toric stratum $D_j^{l'} \cap D_i^{l'} \subset D_i^{l'}$ is
\begin{itemize}
\item the strict transform of $D_j^{l'-1} \cap D_i^{l'-1} \subset \scX_{l'}$ when $l' \neq j$, and
\item the exceptional divisor of the blow-up of $D_i^{l'-1}$ along $D_j^{l'-1} \cap D_i^{l'-1}$ when $l' = j$.
\end{itemize}
We can see the following:
If $D_j^{0} \cap D_i^{0} \subset \scX_{0}$ is a toric divisor of $D_i^0$, then $D_j^{l} \cap D_i^{l} \subset \scX_{l}$ is a toric divisor of $D_i^l$, which corresponds to a $1$-dimensional cone originally contained in $\Sigma_i^{0}(1) \lb \subset \Sigma_i^l (1) \rb$.
If $D_j^{0} \cap D_i^{0} \subset \scX_{0}$ is a toric stratum of $D_i^0$ of codimension greater than $1$, then $D_j^{l} \cap D_i^{l} \subset \scX_{l}$ is a toric divisor of $D_i^l$, which corresponds to the unique $1$-dimensional cone in $\Sigma_i^{j}(1) \setminus \Sigma_i^{j-1} (1) \lb \subset \Sigma_i^l (1) \rb$.
Since toric divisors $D_j^{0} \cap D_i^{0} \subset \scX_{0}$ of $D_i^0$ of the former case are all distinct, toric divisors $D_j^{l} \cap D_i^{l} \subset \scX_{l}$ arising from the former case correspond to distinct $1$-dimensional cones in $\Sigma_i^{0}(1) \lb \subset \Sigma_i^l (1) \rb$.
Furthermore, every toric divisor $D_j^{l} \cap D_i^{l} \subset \scX_{l}$ arising from the latter case corresponds to the unique $1$-dimensional cone in $\Sigma_i^{j}(1) \setminus \Sigma_i^{j-1} (1) \lb \subset \Sigma_i^l (1) \rb$.
Therefore, we can conclude the claim (2).
\end{proof}

\begin{lemma}\label{lm:intersect1}
Let $l \in I_\scX$, $i \in I_\scX \cap \bZ_{>l}$, and $J \subset \lc 1, \cdots, l \rc$.
The intersection $\bigcap_{j \in J \sqcup \lc i \rc} D_j^l \subset \scX_l$ is a toric stratum of codimension $|J|$ in the toric variety $D_i^l$, if the intersection is not empty.
\end{lemma}
\begin{proof}
This immediately follows from \pref{lm:intersect00}, since $D_i^l$ is a smooth toric variety (\pref{lm:intersect0}).
\end{proof}

Let $l_0 \in I_\scX$ be an arbitrary element, and $\lb n_i \rb_i$ be the element of \eqref{eq:prod} corresponding to the irreducible component $D_{l_0} \subset \scX_k$.
Let further $\scK \subset I_\scX$ be the subset of integers $k \in I_\scX$ such that there are exactly two distinct elements $i_1, i_2 \in \lc 0, \cdots, r \rc$ such that the element $\lb n_i' \rb_i$ of \eqref{eq:prod} corresponding to $D_k$ satisfies $n_i' > n_i$ for $i=i_1, i_2$ and $n_i' = n_i$ for $i \neq i_1, i_2$.
For every $j \in \lc 1, \cdots, r \rc$, we also let $\scK_j \subset I_\scX$ denote the subset of integers $k \in I_\scX$ such that the element $\lb n_i' \rb_i$ of \eqref{eq:prod} corresponding to $D_k$ satisfies $n_i' > n_i$ for $i=j$ and $n_i' = n_i$ for $i \neq j$.

\begin{lemma}\label{lm:intersect2}
The exceptional divisor $D_{l_0}^{l_0} \subset \scX_{l_0}$ of the $l_0$-th blow-up is isomorphic to the blow-up of the toric variety $D_{l_0}^{l_0-1} \subset \scX_{l_0-1}$ along
\begin{align}\label{eq:along}
D_{l_0}^{l_0-1} \cap \ld \lb \bigcup_{k \in \scK} D_k^{l_0-1} \rb \cup \bigcup_{j=1}^r \lb \bigcup_{k \in \scK_j} D_k^{l_0-1} \cap \lb \tilde{s}_j=0 \rb \rb \rd,
\end{align}
where $\tilde{s}_j$ is the section of the line bundle obtained as the pullback of the section $s_j \in \Gamma \lb X_{\tilde{\Sigma}'},  \scL_j \rb$ that we took for constructing the toric degeneration $\scX \to \Spec R$.
In particular, the divisor $D_{l_0}^{l_0}$ is irreducible.
\end{lemma}
\begin{proof}
By \pref{lm:cover}, the exceptional divisor $D_{l_0}^{l_0} \subset \scX_{l_0}$ is covered by the affine charts $\scU_{\scS, l} \subset \scX_{l_0}$ constructed in \pref{dl:u}.
Let $\scU_{\scS, l}$ be one of such charts, and let $\scS$ and $l$ be the shuffle and the integer giving rise to the chart $\scU_{\scS, l}$ in the following.
Since the exceptional divisor $D_{l_0}^{l_0}=D_{\scS, l}^l \subset \scX_{l_0}$ is defined by $z_{l}^{i(l), j_{i(l)}(l-1)}$ in $\scU_{\scS, l}$, we can see from \eqref{eq:blow-up1} that the restriction of the function $f_{l, i}$ to $D_{\scS, l}^l$ does not depend on the variables $z_{l}^{i, j}$ with $i \in I(l-1), j=j_i(l-1)$.
When $i(l)=0$, one can see from \eqref{eq:gli} that for $i \in I'(l-1)=I(l-1) \setminus \lc 0 \rc$, we have
\begin{align}\label{eq:xi0}
z_l^{i, j_i(l-1)}=
z_l^{i, j_i(l)}
=\frac{f_{l, i} \cdot \prod_{j=j_0(l)}^{k_0} z^{0,j}_l}{\prod_{j=j_i(l)+1}^{k_i} z^{i, j}_l}
=\frac{f_{l, i} \cdot \prod_{j=j_0(l-1)+1}^{k_0} z^{0,j}_l}{\prod_{j=j_i(l-1)+1}^{k_i} z^{i, j}_l}
=\frac{f_{l, i} \cdot \prod_{i' \in I(l-1) \setminus \lc i \rc} \prod_{j=j_{i'}(l-1)+1}^{k_{i'}} z^{i',j}_l}{\prod_{i' \in I(l-1) \setminus \lc 0 \rc} \prod_{j=j_{i'}(l-1)+1}^{k_{i'}} z^{i', j}_l}
\end{align}
on the open set on which the last denominator is invertible.
When $i(l) \neq 0$, one can also see from \eqref{eq:gli} that we also have
\begin{align}\label{eq:x0}
z_l^{0, j_0(l-1)}
&=z_l^{0, j_0(l)}
=\frac{\prod_{j=j_{i(l)}(l)}^{k_{i(l)}} z^{i(l), j}_l}{f_{l, i(l)} \cdot \prod_{j=j_0(l)+1}^{k_0} z^{0,j}_l}
=\frac{\prod_{j=j_{i(l)}(l-1)+1}^{k_{i(l)}} z^{i(l), j}_l}{f_{l, i(l)} \cdot \prod_{j=j_0(l-1)+1}^{k_0} z^{0,j}_l} \\ \label{eq:x0'}
&=\frac{\prod_{i' \in I(l-1) \setminus \lc 0 \rc} \prod_{j=j_{i'}(l-1)+1}^{k_{i'}} z^{i', j}_l}
{f_{l, i(l)} \cdot \prod_{i' \in I(l-1) \setminus \lc i(l) \rc} \prod_{j=j_{i'}(l-1)+1}^{k_{i'}} z^{i',j}_l},
\end{align}
and for $i \in I'(l-1) \setminus \lc 0 \rc=I(l-1) \setminus \lc 0, i(l) \rc$,
\begin{align}\label{eq:xi}
z_l^{i, j_i(l-1)}
&=
z_l^{i, j_i(l)}
=\frac{f_{l, i} \cdot \prod_{j=j_0(l)}^{k_0} z^{0,j}_l}{\prod_{j=j_i(l)+1}^{k_i} z^{i, j}_l}
=\frac{f_{l, i} \cdot \prod_{j=j_0(l-1)+1}^{k_0} z^{0,j}_l}{\prod_{j=j_i(l-1)+1}^{k_i} z^{i, j}_l}
\cdot \frac{\prod_{i' \in I(l-1) \setminus \lc 0 \rc} \prod_{j=j_{i'}(l-1)+1}^{k_{i'}} z^{i', j}_l}
{f_{l, i(l)} \cdot \prod_{i' \in I(l-1) \setminus \lc i(l) \rc} \prod_{j=j_{i'}(l-1)+1}^{k_{i'}} z^{i',j}_l} \\ \label{eq:xi2}
&=\frac{f_{l, i} \cdot \prod_{i' \in I(l-1) \setminus \lc i \rc} \prod_{j=j_{i'}(l-1)+1}^{k_{i'}} z^{i',j}_l}{f_{l, i(l)} \cdot \prod_{i' \in I(l-1) \setminus \lc i(l) \rc} \prod_{j=j_{i'}(l-1)+1}^{k_{i'}} z^{i',j}_l},
\end{align}
where we used \eqref{eq:x0}-\eqref{eq:x0'} for the last equality of \eqref{eq:xi}.

On the other hand, the divisor $D_{l_0}^{l_0-1}=D_{\scS, l}^{l-1} \subset \scX_{l_0-1}$ is defined by $z^{i, j_i(l-1)}_{l-1}=0$ $\lb i \in I(l-1) \rb$ in $\scU_{\scS, l-1}$ by \eqref{eq:transform1}.
From this and \eqref{eq:xi0}-\eqref{eq:xi2}, one can see that the exceptional divisor $D_{l_0}^{l_0}=D_{\scS, l}^l \subset \scX_{l_0}$ of the $l_0$-th blow-up is isomorphic to the blow-up of $D_{l_0}^{l_0-1}=D_{\scS, l}^{l-1} \subset \scX_{l_0-1}$ along its subvariety defined by
\begin{align}\label{eq:busub}
\prod_{i' \in I(l-1) \setminus \lc 0 \rc} \prod_{j=j_{i'}(l-1)+1}^{k_{i'}} z^{i', j}_{l-1}&=0,\\ \label{eq:busub'}
f_{l-1, i} \cdot \prod_{i' \in I(l-1) \setminus \lc i \rc} \prod_{j=j_{i'}(l-1)+1}^{k_{i'}} z^{i',j}_{l-1}&=0 \quad (i \in I(l-1) \setminus \lc 0 \rc)
\end{align}
in the affine chart $\scU_{\scS, l-1}$.
The equations \eqref{eq:busub} \eqref{eq:busub'} are satisfied if and only if we have either 
\begin{itemize}
\item $z_{l-1}^{i, j}=0$ and $f_{l-1, i}=0$ for some $i \in I(l-1) \setminus \lc 0 \rc$ and $j \geq j_i(l-1)+1$, or
\item $z_{l-1}^{i_1, j_1}=z_{l-1}^{i_2, j_2}=0$ for some distinct $i_1, i_2 \in I(l-1)$ and $j_1 \geq j_{i_1}(l-1)+1$, $j_2 \geq j_{i_2}(l-1)+1$.
\end{itemize}
We can conclude the lemma by this.
\end{proof}

\begin{lemma}\label{lm:intersect20}
For any subset $I \subset \lc 1, \cdots, l_0-1 \rc$, the intersection $\bigcap_{i \in I \sqcup \lc l_0 \rc} D_i^{l_0} \subset \scX_{l_0}$ is an irreducible subvariety of codimension $|I|+1$ in $\scX_{l_0}$, if the intersection is not empty.
\end{lemma}
\begin{proof}
The intersection $\bigcap_{i \in I \sqcup \lc l_0 \rc} D_i^{l_0} \subset \scX_{l_0}$ is isomorphic to the strict transform of $\bigcap_{i \in I \sqcup \lc l_0 \rc} D_i^{l_0-1} \subset \scX_{l_0-1}$ with respect to the blow-up of $D_{l_0}^{l_0-1} \subset \scX_{l_0-1}$ along \eqref{eq:along}.
By \pref{lm:intersect1}, the intersection $\bigcap_{i \in I \sqcup \lc l_0 \rc} D_i^{l_0-1}$ is an irreducible subvariety of $D_{l_0}^{l_0-1}$ of codimension $|I|$.
From these, we can conclude the claim.
\end{proof}

\begin{lemma}\label{lm:intersect3}
For any $l \in I_\scX$ and $I \subset \lc 1, \cdots, l \rc ( \subset I_\scX)$, the intersection $\bigcap_{i \in I} D_i^l \subset \scX_l$ is an irreducible subvariety of codimension $|I|$ in $\scX_l$, if the intersection is not empty.
In particular, for any subset $I \subset I_\scX$, the intersection $\bigcap_{i \in I} \overline{D}_i \subset \scrX$ is an irreducible subvariety of codimension $|I|$ in $\scrX$, if the intersection is not empty.
\end{lemma}
\begin{proof}
The latter claim is the same as the former one for the case $l=l_\scX$.
We have also shown the former claim in the case where $l \in I$ in \pref{lm:intersect20}.
We suppose $l \nin I$ and show the former claim for this case by induction on $l \in I_\scX$.
The claim for $l=1$ $(I = \emptyset)$ is trivial.
We suppose that the claim holds for $l-1$, and show the claim for $l$.
For a subset $I \subset \lc 1, \cdots, l \rc$ $(l \nin I)$, the intersection $\bigcap_{i \in I} D_i^l \subset \scX_l$ is the strict transform of $\bigcap_{i \in I} D_i^{l-1} \subset \scX_{l-1}$.
By the induction hypothesis, the intersection $\bigcap_{i \in I} D_i^{l-1} \subset \scX_{l-1}$ is an irreducible subvariety of codimension $|I|$ in $\scX_{l-1}$.
Thus we can conclude the claim.
\end{proof}

\begin{proof}[Proof of \pref{pr:snc}]
First, we show that the model is an snc-model.
Every irreducible component $\overline{D}_i$ of $\scrX_k$ is a Cartier divisor (\pref{lm:special} and \pref{lm:comp}), and all the intersections of irreducible components of $\scrX_k$ are of expected codimension (\pref{lm:intersect3}).
Therefore, it suffices to show that the scheme $\scrX$ and all the intersections of irreducible components of $\scrX_k$ are regular for showing that the model is snc (cf.~e.g.~\cite[Theorem 36]{MR879273}).
We will do it in the following.

Let $p \in \scU_{\scS, 0}$ be an arbitrary closed point.
We set
\begin{align}
F&:=\lc i \in \bZ \relmid 1 \leq i \leq r, f_{i}(p)= 0 \rc,\\
F^{\mathsf{c}}&:=\lc i \in \bZ \relmid 1 \leq i \leq r, f_{i}(p)\neq 0 \rc,
\end{align}
and
\begin{align}
J_i:=\lc  j \in \bZ \relmid 0 \leq j \leq k_i, z_0^{i, j} (p) = 0 \rc,\\
J_i^{\mathsf{c}}:=\lc  j \in \bZ \relmid 0 \leq j \leq k_i, z_0^{i, j} (p) \neq 0 \rc
\end{align}
for each $i \in \lc 0, \cdots, r \rc$.
Since the sections $s_i \in \Gamma \lb X_{\tilde{\Sigma}'},  \scL_i \rb$ are general, the subscheme defined by $f_i =0$ $(i \in F)$ in $\bigcap_{i=0}^r \bigcap_{j \in J_i} \lb z^{i, j}_0=0\rb$ is regular.
Therefore, there exists a subset
\begin{align}
E \subset \lc (i', j) \in \bZ^2 \relmid 0 \leq i' \leq r, j \in J_{i'}^{\mathsf{c}} \rc
\end{align}
such that $|E|=|F|$ and 
\begin{align}
\det \lb \frac{\partial f_{i}}{\partial z_0^{i', j}} (p) \rb_{i \in F, (i', j) \in E}
\neq 0.
\end{align}
For each $i \in F^{\mathsf{c}}$, we choose an element in $J_i$, and write it as $j(i) \in J_i$.
We set $G:=\lc (i, j(i)) \in \bZ^2 \relmid i \in F^{\mathsf{c}} \rc$ and
\begin{align}\label{eq:H}
H:=\left. k \ld z_0^{i, j} : 0 \leq i \leq r, 0 \leq j \leq k_i, (i, j) \nin E \cup G \rd \ld y_i : 1 \leq i \leq r \rd \middle/ \lb h_i, 1 \leq i \leq r \rb \right.,
\end{align}
where
\begin{align}\label{eq:h_i}
h_i:=\left\{
\begin{array}{ll}
y_i \prod_{j \in J_0} z_0^{0, j}- \prod_{j \in J_i} z_0^{i, j} & i \in F, \\
\prod_{j \in J_0} z_0^{0, j}- y_i \prod_{j \in J_i \setminus \lc j(i)\rc} z_0^{i, j} & i \in F^\mathsf{c}.
\end{array}
\right.
\end{align}
We consider the morphism
\begin{align}\label{eq:et-mor}
H \to \left. H \ld z_0^{i', j} : (i', j) \in E \cup G \rd \middle/ \lb h_i', 1 \leq i \leq r \rb \right.,
\end{align}
where
\begin{align}
h_i':=\left\{
\begin{array}{ll}
y_i \prod_{j \in J_i^\mathsf{c}} z_0^{i, j}- f_i \prod_{j \in J_0^\mathsf{c}} z_0^{0, j} & i \in F, \\
y_i f_i \prod_{j \in J_0^\mathsf{c}} z_0^{0, j}- z_0^{i, j(i)} \prod_{j \in J_i^\mathsf{c}} z_0^{i, j} & i \in F^\mathsf{c}.
\end{array}
\right.
\end{align}
When we localize the target of \eqref{eq:et-mor} by the element $\prod_{i \in F^\mathsf{c}} f_i \cdot \prod_{i \in F \cup \lc 0 \rc}\prod_{j \in J_i^\mathsf{c}} z_0^{i, j}$, the equations $h_i=0$ $(1 \leq i \leq r)$ become equivalent to the defining equations of $\scU_{\scS, 0}$.
We identify the spectrum of the localization of the target of \eqref{eq:et-mor} with a neighborhood of the point $p \in \scU_{\scS, 0}$.
Since we have $f_i(p)=y_i(p)=0$ for $i \in F$, and $z_0^{i, j(i)}(p)=y_i(p)=0$ for $i \in F^\mathsf{c}$, one can get
\begin{align}
\det \lb \frac{\partial h_{i}'}{\partial z_0^{i', j}} (p) \rb_{1 \leq i \leq r, (i', j) \in E \cup G}
=
\det \lb \frac{\partial f_{i}}{\partial z_0^{i', j}} (p) \rb_{i \in F, (i', j) \in E}
\cdot 
\lb \prod_{j \in J_0^\mathsf{c}} z_0^{0, j} (p) \rb^{|F|}
\cdot 
\prod_{i \in F^\mathsf{c}}\prod_{j \in J_i^\mathsf{c}} z_0^{i, j} (p)
\neq 0,
\end{align}
and the morphism \eqref{eq:et-mor} is \'{e}tale at $p$.
Therefore, we can conclude that the model is snc by checking the following:
\begin{claim}\label{cl:regular}
When we blow up $\Spec H$ repeatedly so that the blow-ups correspond to our repetitive blow-ups of $\scX \cap U_C$, it becomes regular.
Furthermore, all the intersections of irreducible components of the divisor defined by $\prod_{j \in J_0} z_0^{0, j}$ also become regular by the blow-ups.
\end{claim}
This claim can be checked by an explicit computation of the repetitive blow-ups of $\Spec H$, which is similar to that done in \pref{sc:local}.
(Using shuffles, one can make a covering of the resulting scheme, which consists of affine open subschemes whose defining equations can  be written down as in \pref{lm:us}.
One can see \pref{cl:regular} from the explicit forms of the equations.)
We omit to repeat the similar computation here since it would be redundant.

Lastly, we show that the model is minimal.
We have $\scrX_k = \sum_{i \in I_\scX} \overline{D}_i$ (\pref{lm:special}) and the image of each $\overline{D}_i$ by $\scrX \to \scX$ is $D_i$.
The resolution $\scrX \to \scX$ is a small resolution, and does not affect the canonical class.
From this and \pref{pr:canonical}, one can see that the logarithmic relative canonical divisor of $\scrX$ is linear equivalent to $0$, and the model is minimal.
\end{proof}

\subsection{The essential skeleton}\label{sc:essential}

The goal of this subsection is to compute the the essential skeleton $\Sk \lb X \rb$ of $X$, and prove \pref{th:main}(1).
The essential skeleton $\Sk \lb X \rb$ is equal to the skeleton of a minimal dlt-model \cite[Theorem 3.3.3]{MR3595497}.
Therefore, by \pref{pr:snc}, one has
\begin{align}
\Sk \lb X \rb=\Sk \lb \scrX \rb.
\end{align}
From \pref{lm:intersect3} and \pref{lm:special}, we can also see the following:
Every divisor $\overline{D}_i \subset \scrX$ $(i \in I_\scX)$ corresponds to a vertex of the skeleton $\Sk \lb \scrX \rb$, and for every non-empty intersection $\bigcap_{i \in I} \overline{D}_i \subset \scrX$ $(I \subset I_\scX)$, there is a unique corresponding cell in $\Sk \lb \scrX \rb$ of dimension $|I|-1$.

\begin{lemma}\label{lm:divisorial}
Let $D:=D_i$ $(i \in I_\scX)$ be an irreducible component of $\scX_k$, and $\cone(\mu_{D}) \times \lc 0 \rc + \cone \lb \nu_{D} \times \lc 1 \rc \rb \in \tilde{\Sigma}'$ be the corresponding cone.
Let further $\overline{D}:=\overline{D}_i$ be the irreducible component of $\scrX_k$ corresponding to $D=D_i$.
Then one has
\begin{align}\label{eq:trop-monomial}
\trop \lb v_{\overline{D}} \rb=\rotatebox[origin=c]{180}{$\beta$} (\mu_{D}) + \nu_{D},
\end{align}
where $v_{\overline{D}}$ denotes the divisorial point associated with the irreducible component $\overline{D}$ (cf.~\pref{sc:berk}).
\end{lemma}
\begin{proof}
Let $C \in \tilde{\Sigma}'$ and $\scS \in \scrS_{C, 1}$ be a relevant cone of maximal dimension and a shuffle such that the open subscheme $\scrU_\scS$ intersects with $\overline{D}$.
Using the notation in \pref{sc:local}, we write $D=D_{\scS, l}$ $\lb l \in \lc 1, \cdots, d+1 \rc \rb$.
Then $\overline{D}=\overline{D}_{\scS, l}$ is defined by $z_\scS^{i(l), j_{i(l)}(l-1)}$ in $\scrU_\scS$ (\pref{lm:comp}).
Furthermore, since the divisor $D=D_{\scS, l}$ is the toric stratum corresponding to the cone generated by $\lc n_{i, j_i(l-1)} \relmid 0 \leq i \leq r \rc$, we have
\begin{align}\label{eq:vud}
\nu_{D} \times \lc 1 \rc=n_{0, j_0(l-1)}, \quad \beta_i^\ast \lb \mu_D \rb=n_{i, j_i(l-1)} \ (1 \leq i \leq r).
\end{align}
The element $\trop \lb v_{\overline{D}} \rb \in N_\bR$ is the image of the element in $N_\bR \oplus \bR=\Hom \lb M \oplus \bZ, \bR \rb$ determined by
\begin{align}\label{eq:mr}
m \mapsto v_{\overline{D}} \lb z^m \rb \quad (m \in M \oplus \bZ)
\end{align}
by the projection $N_\bR \oplus \bR \to N_\bR$.
In order to get the element \eqref{eq:mr}, we will compute $v_{\overline{D}} \lb z^{m_{i, j}} \rb$.

First, we consider $v_{\overline{D}} \lb z^{m_{i, j}} \rb$ with $(i, j) \neq (1, k_1), \cdots, (r, k_r)$.
For such $(i, j)$, there uniquely exists an integer $l' \in \lc 1, \cdots, d+1 \rc$ such that $(i, j)=(i(l'), j_{i(l')}(l'-1))$.
By \pref{lm:pi}, we can get 
\begin{align}\label{eq:vxij}
v_{\overline{D}} \lb z^{i, j}_0:=z^{m_{i, j}} \rb=
\left\{
\begin{array}{ll}
1 & (i, j)=(i(l), j_{i(l)}(l-1)), \\
1 & i =0 \neq i(l), j=j_{i}(l-1),\\
1 & i \in \lc 1, \cdots, r \rc \setminus \lc i(l) \rc, j=j_{i}(l-1), l \leq l_i, \\
0 & \mathrm{otherwise}.\\
\end{array}
\right.
\end{align}
Next, we consider $v_{\overline{D}} \lb z^{m_{i, k_i}} \rb$ $(i \in \lc 1, \cdots, r \rc)$.
For this, we can also see from \pref{lm:pi} and \pref{lm:us} that we have 
\begin{align}
v_{\overline{D}} \lb z^{i, k_i}_0:=z^{m_{i, k_i}} \rb
=v_{\overline{D}} \lb z_\scS^{i, k_i} \rb
= v_{\overline{D}} \lb f_{\scS, i} \cdot \prod_{l' \geq l_i+1}^{d+1} z_\scS^{i(l'), j_{i(l')}(l'-1)} \rb
=\left\{
\begin{array}{ll}
1 & l \geq l_i+1, \\
0 & \mathrm{otherwise}.\\
\end{array}
\right.
\end{align}
When $l \geq l_i+1$, one has $i(l) \neq i$ and $j_i(l-1)=k_i$.
By these and \eqref{eq:vxij}, we can get
\begin{align}
v_{\overline{D}} \lb z^{m_{i, j}} \rb=
\left\{
\begin{array}{ll}
1 & j=j_{i}(l-1), \\
0 & \mathrm{otherwise}.\\
\end{array}
\right.
\end{align}
for all $(i, j)$.
Since $\lc m_{i, j} \relmid 0 \leq i \leq r, 0 \leq j \leq k_i \rc$ is a basis of $M \oplus \bZ$ and $\lc n_{i, j} \relmid 0 \leq i \leq r, 0 \leq j \leq k_i \rc$ is its dual basis, we can see that the element \eqref{eq:mr} is given by
\begin{align}\label{eq:nijl}
\sum_{i=0}^r n_{i, j_i(l-1)}.
\end{align}
From \eqref{eq:vud}, we can see that the images of $n_{0, j_0(l-1)}$ and $n_{i, j_i(l-1)}$ $(1 \leq i \leq r)$ by the projection $N_\bR \oplus \bR \to N_\bR$ are $\nu_D$ and $\beta_i^\ast \lb \mu_D \rb$ respectively.
Thus we obtain \eqref{eq:trop-monomial}.
\end{proof}

Let $C:=\cone(\mu) \times \lc 0 \rc + \cone \lb \nu \times \lc 1 \rc \rb \in \tilde{\Sigma}'$ be a relevant cone of maximal dimension.
For a shuffle $\scS \in \scrS_{C, 1} \cup \scrS_{C, 2}$, the irreducible components of $\scrX_k$ intersecting the affine chart $\scrU_\scS$ are $\overline{D}_{\scS, l}$ $(1 \leq l \leq \bar{p}+1)$.
When $\scS \in \scrS_{C, 1}$, we can see from \pref{lm:us} and \pref{lm:comp} that the intersection $\bigcap_{l=1}^{d+1} \overline{D}_{\scS, l}$ is non-empty.
(It consists of a single point defined by 
\begin{align}
z_\scS^{i(l), j_{i(l)} (l-1)}=0\ (1 \leq l \leq d+1), \quad z_\scS^{i, k_i}=0\ (1 \leq i \leq r)
\end{align}
on $\scrU_\scS$.)
We write the simplex corresponding to the stratum $\bigcap_{l=1}^{d+1} \overline{D}_{\scS, l}$ as $\tau_\scS \subset \Sk \lb \scrX \rb$.
For a shuffle $\scS' \in \scrS_{C, 2}$, there exists a shuffle $\scS \in \scrS_{C, 1}$ such that $j_{\scS', i}(l)=j_{\scS, i}(l)$ for all $i \in \lc 0, \cdots, r \rc$ and $l \in \lc 0, \cdots, \bar{p} \rc$, and any non-empty intersection $\bigcap_{l} \overline{D}_{\scS', l}$ contains the stratum $\bigcap_{l=1}^{d+1} \overline{D}_{\scS, l}$.
Therefore, the $\Delta$-complex structure of $\Sk \lb \scrX \rb$ consists of the simplices 
\begin{align}
\lc \tau_\scS \relmid \scS \in \scrS_{C, 1}, C \in \tilde{\Sigma}': \mathrm{a\ relevant\ cone\ of\ maximal\ dimension} \rc
\end{align}
and their faces.

\begin{lemma}\label{lm:trop-tri}
For a relevant cone $C=\cone(\mu) \times \lc 0 \rc + \cone \lb \nu \times \lc 1 \rc \rb \in \tilde{\Sigma}'$ of maximal dimension, the restriction of the tropicalization map $\trop$ to 
\begin{align}\label{eq:pc}
\bigcup_{\scS \in \scrS_{C, 1}} \tau_\scS \subset  \Sk \lb \scrX \rb
\end{align}
is a piecewise integral affine isomorphism onto the cell $\rotatebox[origin=c]{180}{$\beta$} (\mu) + \nu \in \scrP(\tilde{\Sigma}')$ corresponding to $C$.
\end{lemma}
\begin{proof}
It follows from \pref{lm:divisorial}, \pref{lm:pi}, and the definition of monomial points (cf.~\pref{sc:berk}) that the restriction of $\trop$ to the cell $\tau_\scS$ $\lb \scS \in \scrS_{C, 1} \rb$ is an affine isomorphism onto the convex hull of the points 
\begin{align}\label{eq:tri-vert}
\lc \nu_{D_{\scS, l}}+ \sum_{i=1}^r \beta_i^\ast (\mu_{D_{\scS, l}}) \relmid 1 \leq l \leq d+1 \rc.
\end{align}
These points are the images of \pref{eq:nijl} with $l \in \lc 1, \cdots, d+1 \rc$ by the projection $N_\bR \oplus \bR \to N_\bR$, and we can see that the convex hull of \eqref{eq:tri-vert} is a standard simplex.

Due to \pref{cd:add}, the sets $\nu$ and $\beta_i^\ast (\mu)$ $(1 \leq i \leq r)$ are standard simplices, and their tangent spaces form an internal direct sum in $N_\bR \oplus \bR$.
Hence, the sum $\rotatebox[origin=c]{180}{$\beta$} (\mu) + \nu=\nu + \sum_{i=1}^r \beta_i^\ast \lb \mu \rb$ is isomorphic to the product $\nu \times \prod_{i=1}^r \beta_i^\ast (\mu)$.
We identify them in the following.
Let $P_0$ (resp. $P_i$ $(1 \leq i \leq r)$) be the set of the vertices of $\nu$ (resp. $\beta_i^\ast (\mu)$) equipped with the the total order induced by the total order on $N_i$ of \eqref{eq:ni}, which we fixed in the beginning of \pref{sc:blow-up} respectively.
We think of the sets $\nu, \beta_i^\ast (\mu)$ as the geometric realizations $|\Lambda (P_i)|$ of the abstract simplicial complexes associated with $P_i$.
Then $\rotatebox[origin=c]{180}{$\beta$} (\mu) + \nu=\nu + \sum_{i=1}^r \beta_i^\ast \lb \mu \rb$ is identified with $\prod_{i=0}^r | \Lambda \lb P_i \rb |$.

Let $\scrP(C)$ be the subcomplex of $\Sk(\scrX)$, which consists of $\lc \tau_\scS \relmid \scS \in \scrS_{C, 1} \rc$ and their faces.
A vertex in $\scrP(C)$ corresponds to an irreducible component $D$ of $\scX_k$ (intersecting $U_C$), and the irreducible component $\overline{D}$ of $\scrX_k$.
The cone $\cone(\mu_{D}) \times \lc 0 \rc + \cone \lb \nu_{D} \times \lc 1 \rc \rb \in \tilde{\Sigma}'$ corresponding to the irreducible component $D$ is a face of the cone $C$, and $\nu_{D}, \beta_i^\ast (\mu_D)$ are vertices of $\nu, \beta_i^\ast (\mu)$ respectively.
For the vertex $v_{\overline{D}} \in \scrP(C)$ corresponding to $D$ and $\overline{D}$, we assign the vertex $(\nu_D, \beta_1^\ast (\mu_D), \cdots  \beta_r^\ast (\mu_D))$ in $\Lambda \lb \prod_{i=0}^r P_i \rb$.
One can check that this correspondence is extended to an isomorphism $\scrP(C) \to \Lambda \lb \prod_{i=0}^r P_i \rb$ as simplicial complexes.
We identify $\scrP(C)$ with $\Lambda \lb \prod_{i=0}^r P_i \rb$ by this.

Under all the above identifications, one can also check that the restriction of the tropicalization map $\trop$ to \eqref{eq:pc} coincides with the map \eqref{eq:triangulation} for our $\lc P_i \relmid 0 \leq i \leq r \rc$.
By \pref{lm:triangulation}, it is a piecewise affine isomorphism onto the cell $\rotatebox[origin=c]{180}{$\beta$} (\mu) + \nu$.
Since it sends each (standard) simplex in $\Sk \lb \scrX \rb$ to a standard simplex, it is a piecewise integral affine isomorphism.
Thus we can conclude the claim of the lemma.
\end{proof}

\begin{proof}[Proof of  \pref{th:main}(1)]
For each relevant cone $C$ of maximal dimension, the piecewise integral affine map of \pref{lm:trop-tri} gives the corresponding cell $\rotatebox[origin=c]{180}{$\beta$} (\mu) + \nu \in \scrP(\tilde{\Sigma}')$ a triangulation.
The triangulation is determined according to the orders of vertices of $\nu$ and $\beta_i^\ast (\mu)$.
Those orders come from the total orders of $N_i$ of \eqref{eq:ni}.
Hence, the triangulations for any two cells in $\scrP(\tilde{\Sigma}')$, which are given the piecewise affine maps of \pref{lm:trop-tri} coincide on their intersection, and the triangulations for all the cells in $\scrP(\tilde{\Sigma}')$ give rise to a triangulation of $B^{\check{h}}_\nabla$.
This simplicial structure of $B^{\check{h}}_\nabla$ coincides with that of $\Sk \lb \scrX \rb$.
Therefore, since the restriction of the tropicalization map $\trop$ to each simplex $\tau_\scS$ of $\Sk(\scrX)$ is an integral affine isomorphism, we can conclude \pref{th:main}(1).
\end{proof}

\subsection{The tropical contraction and Berkovich retractions}\label{sc:tropical-berkovich}

The goal of this subsection is to prove \pref{th:main}(2).
Recall that we fixed a vertex $v=\rotatebox[origin=c]{180}{$\beta$} (\mu_v) + \nu_v \in \scrP(\tilde{\Sigma}')$ in the beginning of \pref{sc:blow-up}.
We will prove \pref{th:main}(2) by showing that
\begin{enumerate}
\item locally around the vertex $v$, the composition $\delta \circ \trop$ coincides with the Berkovich retraction $\rho_\scrX$ associated with the model $\scrX$ constructed in \pref{sc:blow-up} (\pref{lm:g-fiber}, \pref{lm:berkovich-tropical}), and
\item the Berkovich retraction $\rho_\scrX$ is an affinoid torus fibration around $v$ (\pref{lm:affinoid}).
\end{enumerate}

The irreducible component of $\scX_k$ to which the vertex $v$ corresponds is $D_{l_\scX} \subset \scX$, i.e., the irreducible component that is the greatest with respect to the order $\leq_\scX$.
In the following, we write $D_v:=D_{l_\scX} \subset \scX$ and $\overline{D}_v:=\overline{D}_{l_\scX} \subset \scrX$.
Let further $\frakX_\eta \subset X^{\an}$ denote the generic fiber of the formal completion $\frakX$ of $\scrX$ along $\overline{D}_v$.

\begin{lemma}\label{lm:g-fiber}
One has
\begin{align}\label{eq:g-fiber}
\frakX_\eta=X^{\an} \cap \trop^{-1} 
\lb \ostar(v) +  \cone_\bT \lb \bigcup_{i=1}^r \beta_i^\ast \lb \mu_{v} \rb \rb \rb,
\end{align}
where $\trop$ is the tropicalization map \eqref{eq:trops}, and $\ostar(v) \subset B^{\check{h}}_\nabla$ is the open star of $v$ in $\scrP(\tilde{\Sigma}')$.
\end{lemma}
\begin{proof}
Let $\rho_v \in \tilde{\Sigma}'$ denote the cone corresponding to the toric subvariety $D_v (\subset X_{\tilde{\Sigma}'})$.
The irreducible component $\overline{D}_v \subset \scrX$ is covered by the family of open subschemes
\begin{align}\label{eq:dv-cover}
\lc \scrU_\scS \relmid \scS \in \scrS_{C, 1}, C \in \tilde{\Sigma}': \mathrm{a\ relevant\ cone\ of\ maximal\ dimension\ such\ that\ } C \succ \rho_v \rc.
\end{align}
(Notice that $D_v$ is the greatest irreducible component with respect to the order $\leq_{\scX}$, and $\overline{D}_v$ does not intersect $\scrU_\scS$ if $\scS \in \scrS_{C, 2}$.)
In an open subscheme $\scrU_\scS$ in \eqref{eq:dv-cover}, the divisor $\overline{D}_{\scS, l}$ $(1 \leq l \leq d+1)$ is defined by $z_\scS^{i(l), j_{i(l)}(l-1)}=0$ (\pref{lm:comp}), and $\overline{D}_v=\overline{D}_{\scS, d+1}$.
We set
\begin{align}\label{eq:xcs0}
\frakX_{C, \scS}:=\lc x \in \lb \scrU_\scS \times_R K \rb^{\an} \relmid
\begin{array}{l}
v_x \lb z_\scS^{i(d+1), j_{i(d+1)}(d)} \rb > 0 \\
v_x \lb z_\scS^{i(l), j_{i(l)}(l-1)} \rb \geq 0, v_x \lb z_\scS^{i, k_i} \rb \geq 0 \ \lb 1 \leq l \leq d, 1 \leq i \leq r \rb 
\end{array}
\rc,
\end{align}
where $v_x \colon k(x) \to \bR \cup \lc \infty \rc$ denotes the additive valuation corresponding to the point $x$ as stated in \pref{sc:berk}.
Then one has
\begin{align}\label{eq:g-union}
\frakX_\eta = \bigcup_{C \in \tilde{\Sigma}'} \bigcup_{\scS \in \scrS_{C, 1}} \frakX_{C, \scS},
\end{align}
where the union with respect to $C$ is taken over all the relevant cones $C \in \tilde{\Sigma}'$ of maximal dimension such that $C \succ \rho_v$.

Let $C':=C \cap \lb N_\bR \times \lc 0 \rc \rb$, and consider the map
\begin{align}\label{eq:trop-a}
\trop_{C'} \colon \lb X_{C'} \rb^{\an} \to \Hom \lb \lb C' \rb^\vee \cap \lb M \oplus \bZ \rb , \bT \rb=X_{C'} \lb \bT \rb \times \bR, 
\quad x \mapsto 
\lb m \mapsto v_x \lb z^m \rb \rb.
\end{align}
In \eqref{eq:trop-a}, the cone $C'$ is regarded as a cone in $N_\bR$ when we consider the affine toric variety $X_{C'}$ over $K$ and the tropical toric variety $X_{C'} \lb \bT \rb$, and as a cone in $N_\bR \times \bR$ when we take the dual cone $\lb C' \rb^\vee$.
Since $v_x \lb z^{(0, 1)}=t \rb=1$ with $(0, 1) \in M \oplus \bZ$ for any $x \in \lb X_{C'} \rb^{\an}$, the image of $\trop_{C'}$ is contained in $X_{C'} \lb \bT \rb \times \lc 1 \rc$, and the composition of $\trop_{C'}$ with the projection $X_{C'} \lb \bT \rb \times \bR \to X_{C'} \lb \bT \rb$ coincides with the tropicalization map \eqref{eq:trop}.

For $l \in \lc 1, \cdots, d+1\rc$, we set
\begin{align}\label{eq:n_l}
n_l:=\sum_{i=0}^r n_{i, j_i(l-1)} \in N \oplus \bZ.
\end{align}
Then one can check
\begin{align}\label{eq:nr}
N_\bR \oplus \bR 
\cong 
\lb \bigoplus_{l=1}^{d+1} \bR \cdot n_l \rb \oplus \lb \bigoplus_{i=1}^r \bR \cdot n_{i, k_i} \rb
\end{align}
and 
\begin{align}\label{eq:mlnl}
\la m_l, n_{l'} \ra=\delta_{l, l'},
\quad
\la m_l, n_{i, k_i} \ra=0 \quad \lb 1 \leq i \leq r \rb, \\ \label{eq:minl}
\la m_{i, k_i}, n_{l} \ra:=\left\{
\begin{array}{ll}
1 & l \geq l_i+1, \\
0 & l <l_i+1,
\end{array}
\right.
\end{align}
where $m_l \in M \oplus \bZ$ $(1 \leq l \leq d+1)$ are the elements defined in \eqref{eq:ml}.

By \eqref{eq:nr}, any element in the target of the map $\trop_{C'}$ of \eqref{eq:trop-a} is written as
\begin{align}\label{eq:n1}
\sum_{l=1}^{d+1} r_l \cdot n_l + \sum_{i=1}^r s_i \cdot n_{i, k_i} \quad \lb r_l, s_i \in \bR \rb,
\end{align}
or as the limit of $\eqref{eq:n1}$ as some of $r_l, s_i$ approach to $\pm \infty$.
By \pref{cr:zsl} and \eqref{eq:mlnl}, the conditions $v_x \lb z_\scS^{i(d+1), j_{i(d+1)}(d)} \rb > 0, v_x \lb z_\scS^{i(l), j_{i(l)}(l-1)} \rb \geq 0$ in \eqref{eq:xcs0} are equivalent to the conditions $r_{d+1} > 0$ and $r_l \geq 0$ for their images by the map $\trop_{C'}$ of \eqref{eq:trop-a} respectively.
We also have $\sum_{l=1}^{d+1} r_l=1$ for the image by the map $\trop_{C'}$, since the image of $\trop_{C'}$ is contained in $X_{C'} \lb \bT \rb \times \lc 1 \rc$ as already mentioned.
Furthermore, by the equation $f_{\scS, i} \cdot \prod_{l \geq l_i+1}^{d+1} z_\scS^{i(l), j_{i(l)}(l-1)}=z^{i, k_i}_\scS$ (\pref{lm:us}), we also have
\begin{align}\label{eq:vf}
v_x \lb z^{i, k_i}_\scS \rb -v_x \lb \prod_{l \geq l_i+1}^{d+1} z_\scS^{i(l), j_{i(l)}(l-1)} \rb=v_x \lb f_{\scS, i} \rb \geq 0
\end{align}
for $x \in \frakX_{C, \scS}$.
From this, \eqref{eq:mlnl}, and \eqref{eq:minl}, we can see that we have
\begin{align}
\lb s_i + \sum_{l \geq l_i+1}^{d+1} r_l \rb -\sum_{l \geq l_i+1}^{d+1} r_l=s_i \geq 0 \quad \lb i \in \lc 1, \cdots, r \rc \rb
 \end{align}
for the image of $\frakX_{C, \scS}$ by the map $\trop_{C'}$.
Conversely, we can see that if the image \eqref{eq:n1} of $x \in \frakX_{C, \scS}$ by the map $\trop_{C'}$ satisfies 
$r_{d+1} > 0$, $r_l \geq 0$ $\lb l \in \lc 1, \cdots ,d \rc \rb$, and $s_i \geq 0$ $\lb i \in \lc 1, \cdots, r \rc \rb$, then $v_x \lb z_\scS^{i, k_i} \rb \geq 0$ is also satisfied by \eqref{eq:minl}.
Therefore, $\frakX_{C, \scS}$ is equal to
\begin{align}\label{eq:xcs}
\lb \scrU_\scS \times_R K \rb^{\an} \cap 
\trop_{C'}^{-1} \lb \lc \sum_{l=1}^{d+1} r_l \cdot n_l + \sum_{i=1}^r s_i \cdot n_{i, k_i} \relmid 
r_{d+1} \in \bR_{>0}, r_l \in \bR_{\geq 0} \lb 1 \leq l \leq d \rb, s_i \in \bT_{\geq 0}, \sum_{l=1}^{d+1} r_l=1 \rc \rb.
\end{align}
By \pref{cr:usk}, this is also equal to
\begin{align}\label{eq:xcs}
X^{\an} \cap 
\trop_{C'}^{-1} \lb \lc \sum_{l=1}^{d+1} r_l \cdot n_l + \sum_{i=1}^r s_i \cdot n_{i, k_i} \relmid 
r_{d+1} \in \bR_{>0}, r_l \in \bR_{\geq 0} \lb 1 \leq l \leq d \rb, s_i \in \bT_{\geq 0}, \sum_{l=1}^{d+1} r_l=1 \rc \rb.
\end{align}
Since $\beta_i^\ast \lb \mu_{v} \rb=n_{i, k_i}$, one can get by \eqref{eq:g-union} and \eqref{eq:xcs} 
\begin{align}\label{eq:xv}
\frakX_\eta=\bigcup_{C \in \tilde{\Sigma}'} \bigcup_{\scS \in \scrS_{C, 1}} 
X^{\an} \cap \trop^{-1} 
\lb U_{C, \scS} +  \cone_\bT \lb \bigcup_{i=1}^r \beta_i^\ast \lb \mu_{v} \rb \rb \rb,
\end{align}
where the union with respect to $C$ is taken over all the relevant cones $C \in \tilde{\Sigma}'$ of maximal dimension such that $C \succ \rho_v$, and $U_{C, \scS}$ is the subset of $N_\bR$ defined by
\begin{align}
U_{C, \scS}:=
\lc -(0, 1) + \sum_{l=1}^{d+1} r_l \cdot n_l \relmid 
r_d \in \bR_{>0}, r_l \in \bR_{\geq 0} \lb 1 \leq l \leq d \rb, \sum_{l=1}^{d+1} r_l=1 \rc \subset N_\bR\ (\times \lc 0 \rc)
\end{align}
with $(0, 1) \in N_\bR \oplus \bR$.
One has
\begin{align}
-(0, 1) + \sum_{l=1}^{d+1} r_l \cdot n_l
= \sum_{l=1}^{d+1} r_l \cdot \lb n_l - (0, 1)\rb.
\end{align}
By \pref{lm:divisorial} and \eqref{eq:vud}, the element $n_l-(0, 1) \in N_\bR \ (\times \lc 0 \rc)$ is the image by the tropicalization map $\trop$ of the divisorial point associated with the divisor $\overline{D}_{\scS, l}$.
We also have $v=n_{d+1}-(0, 1)$.
From these, we can get
\begin{align}
\bigcup_{C \in \tilde{\Sigma}'} \bigcup_{\scS \in \scrS_{C, 1}} U_{C, \scS}=\ostar (v).
\end{align}
It turns out that \eqref{eq:xv} is written as \eqref{eq:g-fiber}.
We obtained the claim of the lemma.
\end{proof}

Let $\rho_\scrX \colon X^{\an} \to \Sk \lb \scrX \rb$ be the Berkovich retraction associated with the model $\scrX$.
Let further $\pi_{C_{v}} \colon X_{C_{v}}(\bT) \to O_{C_{v}} (\bT)$ be the map \eqref{eq:proj} for the cone $C_v:=\cone (\mu_v)$.
For the same reason as \eqref{eq:t}, there uniquely exists a map $t_{v} \colon \pi_{C_{v}}\lb \ostar(v) \rb \to \ostar (v)$ such that 
$t_v \circ \pi_{C_v}$ is the identity map on $\ostar (v)$.

\begin{lemma}\label{lm:berkovich-tropical}
The restriction to $\frakX_\eta$ of the composition
\begin{align}\label{eq:ret-trop}
X^{\an} \xrightarrow{\rho_\scrX} \Sk \lb \scrX \rb=\Sk \lb X \rb \xrightarrow[\cong]{\trop} B^{\check{h}}_\nabla
\end{align}
coincides with the composition 
\begin{align}\label{eq:trop-proj}
\varphi_v \colon \frakX_\eta \xrightarrow{\trop} X_{C_v} \lb \bT \rb \xrightarrow{\pi_{C_{v}}} \pi_{C_{v}}\lb \ostar(v) \rb \xrightarrow{t_{v}} \ostar (v).
\end{align}
\end{lemma}
\begin{proof}
We use the same notation as in the proof of \pref{lm:g-fiber}.
As we saw in the proof of \pref{lm:g-fiber}, the element $n_l-(0, 1) \in N_\bR$ $(l \in \lc 1, \cdots d+1 \rc)$ is the image by the tropicalization map $\trop$ of the divisorial point associated with the divisor $\overline{D}_{\scS, l}$.
The subset $\frakX_\eta \subset X^{\an}$ consists of elements $x \in X^{\an}$ whose centers are contained in the stratum $\overline{D}_{\scS, d+1}$.
Therefore, for any $x \in \frakX_{C, \scS} \subset \frakX_\eta$, we have
\begin{align}\label{eq:trop-rho}
\trop \circ \rho_\scrX (x)
=\sum_{l =1}^{d+1} v_x \lb z_\scS^{i(l), j_{i(l)}(l-1)} \rb \cdot \lb n_l-(0, 1) \rb.
\end{align}
One can also check by using \eqref{eq:mlnl} that the image of $\trop (x)=\trop_{C'} (x)-(0,1)$ by the map $t_v \circ \pi_{C_v}$ coincides with \eqref{eq:trop-rho}.
Thus the lemma holds.
\end{proof}

\begin{lemma}\label{lm:affinoid}
The restriction of the map \eqref{eq:ret-trop} to $\frakX_\eta$ is an affinoid torus fibration.
Furthermore, it induces the same integral affine structure on $\ostar(v)$ as the one defined by the projection map $\psi_v$ of \eqref{eq:fanstr}.
\end{lemma}
\begin{proof}
In order to prove the lemma, we will construct a regular toric model $\scrY \to \Spec R$ whose formal completion along a stratum is isomorphic to $\frakX$.
By using the fact that Berkovich retractions locally depend only on the formal completions, and that the Berkovich retraction associated with a regular toric model is an affinoid torus fibration (cf.~\cite[Example 3.5]{MR3946280}, \cite[Section 1.5]{MPS21}), we will see that the restriction of the map \eqref{eq:ret-trop} to $\frakX_\eta$ is also an affinoid torus fibration.

We set
\begin{align}
\overline{N} := 
\left. N \middle/ \bigoplus_{i=1}^r \bZ \cdot \beta_i^\ast \lb \mu_v \rb \right., \quad
\overline{M}
:=\Hom \lb \overline{N}, \bZ \rb,
\end{align}
and write the triangulation of $B^{\check{h}}_\nabla$ mentioned in the end of \pref{sc:essential} as $\scrT$.
We consider the polyhedral complex
\begin{align}
\scrP_v:=\lc \psi_v \lb \tau \rb \relmid \tau \in \scrT, \tau \succ v \rc
\end{align}
in $\overline{N}_\bR:=\overline{N} \otimes_\bZ \bR$, where $\psi_v$ is the projection map given in \eqref{eq:fanstr}.
Let $\Sigma_v$ be the fan in $\overline{N}_\bR \oplus \bR$ obtained by taking the conic hull of the image of $\scrP_v$ by the map
\begin{align}
\overline{N}_\bR
\hookrightarrow
\overline{N}_\bR \oplus \bR, \quad n \mapsto (n, 1).
\end{align}
Let further $Y \lb \Sigma_v \rb$ be the toric variety over $k$ associated with the fan $\Sigma_v$.
The fan $\Sigma_v$ is contained in $\overline{N}_\bR \times \bR_{\geq 0}$, and the projection $\overline{N}_\bR \times \bR_{\geq 0} \to \bR_{\geq 0}$ induces a toric morphism $Y \lb \Sigma_v \rb \to \bA^1_k$.
We write its base change to $R$ as $\scrY \to \Spec R$.
Since $\scrT$ is a unimodular triangulation, the fan $\Sigma_v$ is also unimodular.
Hence, the toric model $\scrY \to \Spec R$ is regular.

The maximal-dimensional cones in $\Sigma_v$ are
\begin{align}
\lc C_\scS:=\cone \lb \psi_v \lb \tau_\scS \rb \times \lc 1 \rc \rb \relmid \scS \in \scrS_{C, 1}, C \in \tilde{\Sigma}': \mathrm{a\ relevant\ cone\ of\ maximal\ dimension\ such\ that\ } C \succ \rho_v \rc,
\end{align}
and the cone $C_\scS$ is generated by the elements $n_l \in \overline{N} \oplus \bZ$ $(1 \leq l \leq d+1)$ given by \eqref{eq:n_l}.
We can see from \eqref{eq:mlnl} that the elements $m_l \in \overline{M} \oplus \bZ$ given by \eqref{eq:ml} are dual to $n_l$.
Therefore, the affine toric subscheme of $\scrY$ associated with the cone $C_\scS$ is written as
\begin{align}\label{eq:cs-sub}
\left. \Spec R \ld C_\scS^\vee \cap \lb \overline{M} \oplus \bZ \rb \rd \middle/ \lb t- \prod_{l=1}^{d+1} y_\scS^l \rb \right.
\end{align}
with $y_\scS^l:=z^{m_l} \in k \ld C_\scS^\vee \cap \lb \overline{M} \oplus \bZ \rb \rd$, and the toric model $\scrY$ is covered by such open subschemes.

We take the formal completion $\frakY$ of $\scrY$ along the toric divisor corresponding to the $1$-dimensional cone $\cone \lb \psi_v (v) \times \lc 1 \rc \rb \in \Sigma_v$.
Then \eqref{eq:cs-sub} turns into
\begin{align}
\frakV_\scS:=\Spf \left. R \lc y^l_\scS \lb 1 \leq l \leq d \rb \rc
\ldd y_\scS^{d+1} \rdd \middle/ \lb t- \prod_{l=1}^{d+1} y_\scS^l \rb \right..
\end{align}
On the other hand, when we take the formal completion $\frakX$ of $\scrX$ along $\overline{D}_v$, the open subscheme $\scrU_\scS$ of \eqref{eq:Us} turns into
\begin{align}
\frakU_\scS:=\Spf \left. R \lc z_\scS^l \lb 1 \leq l \leq d \rb \rc
\lc z^{i, k_i}_\scS \lb 1 \leq i \leq r \rb \rc
\ldd z_\scS^{d+1} \rdd \middle/ \lb t- \prod_{l=1}^{d+1} z_\scS^l, g_{\scS, i}, \lb 1 \leq i \leq r \rb \rb \right.,
\end{align}
where $z_\scS^l:=z_\scS^{i(l), j_{i(l)}(l-1)}$ $(1 \leq l \leq d+1)$.
We consider the morphism $f_\scS \colon \frakU_\scS \to \frakV_\scS$ defined by
\begin{align}
\scO \lb \frakV_\scS \rb \to \scO \lb \frakU_\scS \rb, \quad y_\scS^l \mapsto z_\scS^l \ (1 \leq l \leq d+1).
\end{align}
For any shuffles $\scS, \scS' \in \scrS_{C, 1}$, the coordinates $z_\scS^l, z_{\scS'}^l$ $(1 \leq l \leq d+1)$ are written with $z^{m_{i, j}}$ as in \pref{cr:zsl}, and the coordinate transformation between them is written down using the formula \eqref{eq:ml}.
The coordinate transformation from $y_\scS^l$ to $y_{\scS'}^l$ is also written down in exactly the same way.
Therefore, we can see that the morphisms $\lc f_\scS \rc_\scS$ can be glued together, and get the morphism of formal $R$-schemes
\begin{align}
f \colon \frakX \to \frakY.
\end{align}

\begin{claim}\label{cl:f-isom}
The morphism $f \colon \frakX \to \frakY$ is an isomorphism.
\end{claim}

\pref{cl:f-isom} can be proved in exactly the same way as \cite[Proposition 2.5.3]{MPS21}.
We omit to repeat the same proof here.
The argument was originally considered in \cite[Proposition 5.4]{MR3946280}.

The generic fiber of the toric model $\scrY$ is the torus $T:= \Spec K \ld \overline{M} \rd$, and the tropicalization map $\trop \colon T^{\an} \to \overline{N}_\bR$ gives an identification between the Berkovich skeleton $\Sk \lb \scrY \rb \subset T^{\an}$ associated with the model $\scrY$ and the polyhedral complex $\scrP_v \subset \overline{N}_\bR$.
Furthermore, the generic fiber $\frakY_\eta$ of $\frakY$ coincides with the preimage of the open star $\ostar \lb \psi_v \lb v \rb \rb$ of $\psi_v \lb v \rb$ in $\scrP_v$ by the tropicalization map, and the Berkovich retraction $\rho_\scrY \colon \frakY_\eta \to \Sk \lb \scrY \rb$ coincides with the restriction of the tropicalization map $\trop$ to $\frakY_\eta \subset T^{\an}$.
For these facts, we refer the reader to \cite[Example 3.5]{MR3946280} or \cite[Section 1.5]{MPS21}.

By \pref{cl:f-isom}, we have $\frakX_\eta \cong \frakY_\eta$, and the Berkovich retractions $\rho_{\scrX} \colon \frakX_\eta \to \ostar(v) \subset \Sk \lb \scrX \rb$ and $\rho_\scrY \colon \frakY_\eta \to \ostar \lb \psi_v (v) \rb \subset \Sk \lb \scrY \rb$ are identified.
(We are thinking of $\ostar(v)$ and $\ostar \lb \psi_v (v) \rb$ as subsets of $\Sk \lb \scrX \rb$ and $\Sk \lb \scrY \rb$ respectively via the tropicalization maps.)
By this identification and what was mentioned in the previous paragraph, we can see that the Berkovich retractions $\rho_{\scrX}$ is isomorphic to a restriction of the tropicalization map.
Thus we can conclude the former claim of \pref{lm:affinoid}.
The latter claim also holds, since the sets $\ostar(v)$ and $\ostar \lb \psi_v (v) \rb$ are identified via the projection map $\psi_v$ of \eqref{eq:fanstr}.
\end{proof}

\begin{lemma}\label{lm:torus-fib}
The map $\varphi_v \colon \frakX_\eta \to \ostar(v)$ of \eqref{eq:trop-proj} is an affinoid torus fibration.
Furthermore, it induces the same integral affine structure on $\ostar(v)$ as the one defined by the projection map $\psi_v$ of \eqref{eq:fanstr}.
\end{lemma}
\begin{proof}
This immediately follows from \pref{lm:berkovich-tropical} and \pref{lm:affinoid}.
\end{proof}

\begin{lemma}\label{lm:torus-fib2}
The map $\delta \circ \trop$ of \eqref{eq:na-syz} is an affinoid torus fibration over $W_v^\circ \subset B^{\check{h}}_\nabla$.
Furthermore, it induces the same integral affine structure on $W_v^\circ$ as the one defined by the projection map $\psi_v$ of \eqref{eq:fanstr}.
\end{lemma}
\begin{proof}
By \pref{lm:delta-1} for $\tau=v$, we have
\begin{align}
\lb \delta \circ \trop \rb^{-1} \lb W_v^\circ \rb
&=
\trop^{-1} \lc \lb W_{v}^\circ + \cone_\bT \lb \bigcup_{i=1}^r \beta_i^\ast \lb \mu_{v} \rb \rb \rb \cap X(f_1, \cdots, f_r) \rc \\ \label{eq:d-trop-w}
&=X^{\an} \cap \trop^{-1} 
\lb W_{v}^\circ +  \cone_\bT \lb \bigcup_{i=1}^r \beta_i^\ast \lb \mu_{v} \rb \rb \rb.
\end{align}
We can see from \pref{lm:g-fiber} and the construction of the map $\varphi_v \colon \frakX_\eta \to \ostar(v)$ of \eqref{eq:trop-proj} that \eqref{eq:d-trop-w} is the preimage of $W_v^\circ (\subset \ostar (v))$ by the map $\varphi_v$.
Furthermore, according to \eqref{eq:local-delta}, the map $\delta \circ \trop$ coincides with the map $\varphi_v=t_v \circ \pi_{C_v} \circ \trop$ over $W_v^\circ$.
Therefore, by \pref{lm:torus-fib}, we can conclude \pref{lm:torus-fib2}.
\end{proof}

\begin{lemma}\label{lm:torus-fib3}
Let $\sigma \in \scrP(\tilde{\Sigma}')$ be a maximal-dimensional polyhedron containing $v$ as its vertex.
The map $\delta \circ \trop$ of \eqref{eq:na-syz} is an affinoid torus fibration over $\rint \lb \sigma \rb \subset B^{\check{h}}_\nabla$.
Furthermore, it induces the same integral affine structure on $\rint \lb \sigma \rb$ as the one defined by the map $\psi_\sigma$ of \eqref{eq:int-aff}.
\end{lemma}
\begin{proof}
By \pref{lm:delta-1m} and \eqref{eq:sigma-cone2}, one has 
\begin{align}
\lb \delta \circ \trop \rb^{-1} \lb \rint (\sigma) \rb
&=
\trop^{-1} \lb \rint (\sigma) \rb \\
&=
\trop^{-1} \lb \lb \rint \lb \sigma \rb + \cone_\bT \lb \bigcup_{i=1}^r \beta_i^\ast \lb \mu_{v} \rb \rb
 \rb 
 \cap X(f_1, \cdots, f_r) \rb \\ \label{eq:d-trop-s}
&=X^{\an} \cap \trop^{-1} 
\lb \rint \lb \sigma \rb +  \cone_\bT \lb \bigcup_{i=1}^r \beta_i^\ast \lb \mu_{v} \rb \rb \rb.
\end{align}
We can see from \pref{lm:g-fiber} and the construction of the map $\varphi_v \colon \frakX_\eta \to \ostar(v)$ of \eqref{eq:trop-proj} that \eqref{eq:d-trop-s} is the preimage of $\rint \lb \sigma \rb (\subset \ostar (v))$ by the map $\varphi_v$.
Furthermore, we can see from \pref{lm:delta-1m} that the map $\delta \circ \trop$ coincides with the map $\varphi_v =t_v \circ \pi_{C_v} \circ \trop$ over $\rint \lb \sigma \rb$.
Therefore, by \pref{lm:torus-fib}, we can conclude \pref{lm:torus-fib3}.
(Notice that the projection map $\psi_v$ gives an integral affine isomorphism from $\aff \lb \sigma \rb$ to $\overline{N}_\bR$, and the integral affine structures on $\rint (\sigma)$ induced by $\psi_v$ and $\psi_\sigma$ are the same.)
\end{proof}

\begin{proof}[Proof of  \pref{th:main}(2)]
In the beginning of \pref{sc:proof1}, we took a vertex $v \in \scrP(\tilde{\Sigma}')$ arbitrarily.
Hence, \pref{lm:torus-fib2} and \pref{lm:torus-fib3} hold for any vertex $v \in \scrP(\tilde{\Sigma}')$.
Thus we can conclude \pref{th:main}(2).
\end{proof}

\begin{remark}\label{rm:ps22}
Let $\lc e_i \relmid 1 \leq i \leq d+1 \rc$ be a basis of the lattice $M$.
\pref{th:main} is proved in \cite[Theorem A]{PS22} in the case where $r=1$,
\begin{align}
\Delta=\conv \lb \lc -\sum_{i=1}^{d+1} e_i, (d+2) e_j-\sum_{i=1}^{d+1} e_i \relmid 1 \leq j \leq d+1 \rc \rb,
\end{align}
$\Sigma' =\Sigma, \Sigmav' =\Sigmav$, and $\check{h}=\check{\varphi}$.
In his proof, he applies \cite[Theorem A]{MPS21} for the snc-model $\scrX$ of \pref{pr:snc} and the stratum $\overline{D}_v$ to show that the Berkovich retraction $\rho_\scrX$ associated with $\scrX$ is an affinoid torus fibration around $v$.
In order to apply \cite[Theorem A]{MPS21}, we need the conormal bundle of the stratum $\overline{D}_v$ to be nef.
In our general setup stated in \pref{sc:toric-CY}, this does not hold in general.
Therefore, in our setup, we are not able to apply \cite[Theorem A]{MPS21} directly as done in \cite{PS22}.
Instead, we considered the toric model $\scrY \to \Spec R$ in this article.
\end{remark}

\begin{remark}
In the final part \cite[Corollary 4.18]{PS22} of the proof of \cite[Theorem A]{PS22}, he combines
\begin{align}\label{eq:PS22-1}
\rho_{\frakX_i}=(q_i)^{-1} \circ p_i \circ \val_{\bP}
\end{align}
on $\frakX^\eta_i$ and 
\begin{align}\label{eq:PS22-2}
\delta_{\underline{a}}=(q_i)^{-1} \circ p_i
\end{align}
on $Y_{v_i}$ to obtain 
\begin{align}\label{eq:PS22-3}
\rho_{\frakX_i}=\delta_{\underline{a}} \circ \val_{\bP}
\end{align}
on $\rho_{\frakX_i}^{-1} \lb \widetilde{\mathrm{Star}}(v_i) \rb$.
(Here we are using the notation in \cite{PS22}.
If we use the notation of this article, then
\begin{align}
\rho_{\frakX_i}, (q_i)^{-1}, p_i, \val_{\bP}, \frakX^\eta_i, \delta_{\underline{a}}, Y_{v_i}, \widetilde{\mathrm{Star}}(v_i)
\end{align} 
are
\begin{align}
\rho_\scrX, t_v, \pi_{c_v}, \trop, \frakX_\eta, \delta, W_v^\circ +  \cone_\bT \lb \bigcup_{i=1}^r \beta_i^\ast \lb \mu_{v} \rb \rb, W_v^\circ
\end{align}
respectively. 
\eqref{eq:PS22-1} and \eqref{eq:PS22-2} are the equalities mentioned in \pref{lm:berkovich-tropical} and \eqref{eq:local-delta} respectively.)
We also did essentially the same thing in the proof of \pref{lm:torus-fib2}.
For doing this, we need the inclusion
\begin{align}\label{eq:PS22-4}
\val_{\bP} \lb \rho_{\frakX_i}^{-1} \lb \widetilde{\mathrm{Star}}(v_i) \rb \rb
=
\trop \lb \rho_\scrX^{-1} \lb W_v^\circ \rb \rb 
\subset
W_v^\circ +  \cone_\bT \lb \bigcup_{i=1}^r \beta_i^\ast \lb \mu_{v} \rb \rb=Y_{v_i},
\end{align}
since \eqref{eq:PS22-2} holds only on $Y_{v_i}=W_v^\circ +  \cone_\bT \lb \bigcup_{i=1}^r \beta_i^\ast \lb \mu_{v} \rb \rb$.
\eqref{eq:PS22-4} is not discussed in \cite{PS22}.
In this article, \eqref{eq:PS22-4} is essentially clarified in the proof of \pref{lm:torus-fib2} by using \pref{lm:g-fiber} that was proved by the explicit computation of blow-ups in \pref{sc:local}.
\end{remark}

\begin{remark}\label{rm:KSMPS21}
The non-archimedean SYZ fibration of \pref{th:main} (2) (or of \cite[Theorem A]{PS22}) coincides with the one in \cite[Section 4.2.5]{MR2181810} in the case of a quartic K3 surface (cf.~\cite[Proposition 4.2.4]{pilleschneider:tel-04107501}).
It is also similar to the one constructed in \cite[Theorem C]{MPS21} for a quintic $3$-fold.
However, they differ over the discriminant (cf.~\cite[Section 3.7.5]{MPS21}).
\end{remark}

\section*{Acknowledgement}

I am greatly indebted to Sam Payne for motivating discussions.
In particular, studying the relations between the Gross--Siebert program and Berkovich geometry was suggested by him.
I am grateful to Keita Goto and Enrica Mazzon for their kind answers to my questions on \cite{MR3946280, MPS21}.
I thank Igor Krylov and Yat-Hin Suen for helpful communications.
I also thank Keita Goto, Yuji Odaka, and Kazushi Ueda for valuable comments on the draft of this article.
I also appreciate many comments from the anonymous referee, which helped me improve this article.
This work was supported by RIKEN iTHEMS Program.

\bibliographystyle{amsalpha}
\bibliography{bibs}

\end{document}